\documentclass[10pt]{article}
\usepackage[preprint]{tmlr}

\usepackage{hyperref}
\usepackage{url}
\usepackage{color}
\title{A general framework of Riemannian adaptive optimization methods
with a convergence analysis}

\author{\name Hiroyuki Sakai \email sakai0815@cs.meiji.ac.jp \\
      \addr Meiji University \\
      \AND
      \name Hideaki Iiduka \email iiduka@cs.meiji.ac.jp \\
      \addr Meiji University}

\def\month{01}
\def\year{2025}
\def\openreview{\url{https://openreview.net/forum?id=knv4lQFVoE}}

\usepackage{multirow}
\usepackage{booktabs}
\usepackage{enumitem}
\usepackage{algorithm}
\usepackage{algorithmicx}
\usepackage{algpseudocode}
\usepackage{subfigure}
\usepackage{graphicx}
\usepackage{amsmath, amssymb, amsthm}
\usepackage{cite}
\usepackage{amsmath,amssymb,amsthm,mathrsfs}

\newtheorem{theorem}{Theorem}[section]

\newtheorem{lemma}[theorem]{Lemma}
\newtheorem{proposition}[theorem]{Proposition}
\newtheorem{definition}[theorem]{Definition}
\newtheorem{assumption}[theorem]{Assumption}

\DeclareMathOperator*{\grad}{grad}
\DeclareMathOperator*{\diag}{diag}
\DeclareMathOperator*{\tr}{tr}
\DeclareMathOperator*{\id}{id}
\DeclareMathOperator*{\St}{St}
\DeclareMathOperator*{\Gr}{Gr}
\DeclareMathOperator*{\qf}{qf}
\DeclareMathOperator*{\sym}{sym}

\DeclareMathOperator*{\argmin}{arg\,min}

\newcommand{\real}{\mathbb{R}}
\newcommand{\ex}{\mathbb{E}}
\newcommand{\abs}[1]{\left|{#1}\right|}
\newcommand{\norm}[1]{\left\lVert{#1}\right\rVert}
\newcommand{\ip}[2]{\left\langle{#1,#2}\right\rangle}

\newcommand{\bs}{\boldsymbol{s}}

\begin{document}
\maketitle

\begin{abstract}
This paper proposes a general framework of
Riemannian adaptive optimization methods.
The framework encapsulates several stochastic optimization algorithms
on Riemannian manifolds and
incorporates the mini-batch
strategy that is often used in deep learning. 
Within this framework, we also propose AMSGrad
on embedded submanifolds of Euclidean space.
Moreover, we give convergence analyses valid for both
a constant and a diminishing step size.
Our analyses also reveal the relationship between the convergence rate
and mini-batch size.
In numerical experiments,
we applied the proposed algorithm to principal
component analysis and the low-rank matrix completion problem,
which can be considered to be Riemannian optimization problems.
Python implementations of the methods used
in the numerical experiments are available at
\url{https://github.com/iiduka-researches/202408-adaptive}.
\end{abstract}

\section{Introduction}
Riemannian optimization \citep{absil2008optimization, sato2021riemannian}
has received much attention in machine learning.
For example, batch normalization \citep{cho2017riemannian},
representation learning \citep{nickel2017poincare}, and
the low-rank matrix completion problem
\citep{vandereycken2013low, cambier2016robust, boumal2015low}
can be considered optimization problems on Riemannian manifolds.
This paper focuses on Riemannian adaptive optimization
algorithms for solving stochastic optimization problems on
Riemannian manifolds.
In particular, we treat Riemannian submanifolds of Euclidean space
(e.g., unit spheres and the Stiefel manifold).

In Euclidean settings, adaptive optimization methods
are widely used for training deep neural networks.
There are many adaptive optimization methods, such as
Adaptive gradient (AdaGrad) \citep{duchi2011adaptive},
Adadelta \citep{zeiler2012adadelta},
Root mean square propagation (RMSProp) \citep{hinton2012neural},
Adaptive moment estimation (Adam) \citep{kingma2015adam},
Yogi \citep{zaheer2018adaptive},
Adaptive mean square gradient (AMSGrad) \citep{reddi2018convergence},
AdaFom \citep{chen2019convergence},
AdaBound \citep{Luo2019AdaBound},
Adam with decoupled weight decay (AdamW)
\citep{loshchilov2019decoupled} and
AdaBelief \citep{zhuang2020adabelief}.
\citet{reddi2018convergence}
proposed a general framework of adaptive optimization methods
that encapsulates many of the popular adaptive methods
in Euclidean space.

\citet{bonnabel2013stochastic}
proposed Riemannian stochastic gradient descent
(RSGD), the most basic Riemannian stochastic optimization
algorithm.
In particular,
Riemannian stochastic variance reduction algorithms, such as
Riemannian stochastic variance-reduced gradient (RSVRG) 
\citep{zhang2016riemannian},
Riemannian stochastic recursive gradient (RSRG)
\citep{kasai2018riemannian}, and
Riemannian stochastic path-integrated differential estimator
(R-SPIDER) \citep{zhang2018r, zhou2019faster},
are based on variance reduction methods in Euclidean space.
There are several prior studies on Riemannian adaptive optimization
methods for specific Riemannian manifolds.
In particular, \citet{kasai2019riemannian} proposed a
Riemannian adaptive stochastic gradient algorithm
on matrix manifolds (RASA). RASA is an adaptive
optimization method on matrix manifolds
(e.g., the Stiefel manifold or the Grassmann manifold),
and gave a convergence analysis under the upper-Hessian bounded
and retraction $L$-smooth assumptions
(see \citep[Section 4]{kasai2019riemannian} for details).
However, RASA is not a direct extension of
the adaptive optimization methods commonly used in deep learning,
and it works only for diminishing step sizes.
RAMSGrad \citep{becigneul2019riemannian} and modified RAMSGrad \citep{sakai2021riemannian},
direct extensions of AMSGrad, have been proposed
as methods that work on Cartesian products of Riemannian manifolds.
In particular, \citet{roy2018geometry}
proposed cRAMSProp and applied it
to several Riemannian stochastic optimizations.
However, they did not provide a convergence analysis of it.
More recently, Riemannian stochastic
optimization methods,
sharpness-aware minimization on Riemannian manifolds
(Riemannian SAM) \citep{yun2024riemannian} and
Riemannian natural gradient descent (RNGD) \citep{hu2024riemannian},
were proposed.

\subsection{Motivations}
\citet{kasai2019riemannian} presented useful results indicating RASA with a diminishing step size converges to a stationary point of a continuously differentiable function defined on a matrix manifold. Meanwhile, other results have indicated that using constant step sizes makes the algorithm useful for training deep neural networks \citep{zaheer2018adaptive} and that the setting of the mini-batch size depends on the performance of the algorithms used to train deep neural networks \citep{smith2018dont, sato2023existence}. The main motivation of this paper is thus to gain an understanding of the theoretical performance relationship between Riemannian adaptive optimization methods using constant step sizes and ones using the mini-batch size. We are also interested in understanding the theoretical performance relationship between the Riemannian adaptive optimization methods using diminishing step sizes and those using the mini-batch size.

\subsection{Contributions}
Our first contribution is to propose a framework of
adaptive optimization methods on Riemannian submanifolds of
Euclidean space (Algorithm \ref{alg:general}) that is
based on the framework \citep[Algorithm 1]{reddi2018convergence}
proposed by Reddi, Kale and Kumar for Euclidean space.
A key challenge in designing adaptive gradient methods over Riemannian manifolds is the absence of a coordinate system like that in Euclidean space. In this work, we specifically focus on manifolds that are embedded in Euclidean space, a common setting in many applications. We address the issue by constructing algorithms that rely on projections onto the tangent space at each point on the manifold.
Our framework incorporates the mini-batch strategy that is often used
in deep learning.
Important examples of Riemannian submanifolds of
the Euclidean space include the unit sphere
and the Stiefel manifold.
Note that \citep{kasai2019riemannian} focuses exclusively on the RASA algorithm, which is defined only on matrix manifolds. In contrast, our framework can be defined on any manifold embedded in Euclidean space.
Moreover, within this framework, we propose AMSGrad
on embedded submanifolds of Euclidean space
(Algorithm \ref{alg:AMSGrad}) as a direct extension of AMSGrad.

Our second contribution is to give convergence analyses of the proposed algorithm (Algorithm \ref{alg:general}) by using a mini-batch size $b$ valid for both a constant step size (Theorem \ref{thm:constant}) and diminishing step size (Theorem \ref{thm:diminishing}). In particular, Theorem \ref{thm:constant} indicates that, under some conditions, the sequence $(x_k)_{k=1}^{\infty}$ generated by the proposed algorithm with a constant step size $\alpha$ satisfies
\begin{align*}
\min_{k=1,\ldots, K} \ex \left[\norm{\grad f(x_k)}_2^2\right] \leq \frac{1}{K} \sum_{k=1}^K \ex \left[\norm{\grad f(x_k)}_2^2\right] = \mathcal{O} \left(\frac{1}{K} + \frac{1}{b}  \right). 
\end{align*}
We should note that Theorem \ref{thm:constant} is an extension of the result in \citet[Corollary 2]{zaheer2018adaptive}.
Theorem \ref{thm:diminishing} indicates that, under some conditions, the sequence $(x_k)_{k=1}^{\infty}$ generated by the proposed algorithm with a diminishing step size $\alpha_k$ satisfies
\begin{align*}
\min_{k=1,\ldots, K} \ex\left[\norm{\grad f(x_k)}_2^2\right] \leq \frac{1}{K} \sum_{k=1}^K \ex \left[\norm{\grad f(x_k)}_2^2\right] = \mathcal{O} \left( \left(1 + \frac{1}{b}  \right) \frac{\log K}{\sqrt{K}} \right). 
\end{align*} 
Since the stochastic gradient (see \eqref{eq:bgrad}) using large mini-batch sizes is approximately the full gradient of $f$; intuitively, using large mini-batch sizes decreases the number of steps needed to find stationary points of $f$. Theorems \ref{thm:constant} and \ref{thm:diminishing} indicate that, the larger the mini-batch size $b$ is, the smaller the upper bound of $\frac{1}{K} \sum_{k=1}^K \mathbb{E} [ \|  \mathrm{grad} f(x_k)  \|_2^2]$ becomes and the fewer the required steps become. Hence, Theorems \ref{thm:constant} and \ref{thm:diminishing} would match our intuition.

Theorem \ref{thm:constant} indicates that, although the proposed algorithm with a constant step size $\alpha$ and constant mini-batch size $b$ does not always converge, we can expect that converges with an increasing mini-batch size $b_k$ such that $b_k \leq b_{k+1}$.
In practice, it has been shown that increasing the mini-batch size \citep{richard2012sample,balles2016coupling,automated2017soham,smith2018dont,goyal2018accurate} is useful for training deep neural networks with mini-batch optimizers.
Motivated by Theorems \ref{thm:constant} and \ref{thm:diminishing}, we give convergence analyses of the proposed algorithm by using an increasing mini-batch size $b_k$ valid for both a constant step size (Theorem \ref{thm:constant_1}) and diminishing step size (Theorem \ref{thm:diminishing_1}). 
Theorem \ref{thm:constant_1} indicates that, under some conditions, the sequence $(x_k)_{k=1}^{\infty}$ generated by the proposed algorithm with a constant step size $\alpha$ and an increasing mini-batch size $b_k$ satisfies 
\begin{align*}
    \min_{k=1,\ldots, K} \ex\left[\norm{\grad f(x_k)}_2^2\right] \leq \frac{1}{K}\sum_{k=1}^K\ex\left[\norm{\grad f(x_k)}_2^2\right]
    =\mathcal{O}\left(\frac{1}{K}\right).
\end{align*}
That is, with increasing mini-batch sizes, it can find a stationary point of $f$, in contrast to Theorem \ref{thm:constant}.
Moreover, Theorem \ref{thm:diminishing_1} indicates that, under some conditions, the sequence $(x_k)_{k=1}^{\infty}$ generated by the proposed algorithm with a diminishing step size $\alpha_k$ and an increasing mini-batch size $b_k$ satisfies  
\begin{align*}
    \min_{k=1,\ldots, K} \ex\left[\norm{\grad f(x_k)}_2^2\right] \leq \frac{1}{\sum_{k=1}^K \alpha_k}\sum_{k=1}^K \alpha_k \ex\left[\norm{\grad f(x_k)}_2^2\right]
    =\mathcal{O}\left(
    \frac{1}{\sqrt{K}}\right).
\end{align*}
That is, with increasing mini-batch sizes, it has a better convergence rate than with constant mini-batch sizes (Theorem \ref{thm:diminishing}).

The third contribution is to numerically compare
the performances of several methods based on
Algorithm \ref{alg:general}, including Algorithm \ref{alg:AMSGrad},
with the existing methods.
In the numerical experiments, we applied the algorithms
to principal component analysis (PCA)
\citep{kasai2018riemannian, roy2018geometry}
and the low-rank matrix completion (LRMC) problem
\citep{boumal2015low, kasai2019riemannian, hu2024riemannian},
which can be considered to be Riemannian optimization problems.
Based on our experiments, we can summarize the following key findings. For the PCA problem, RAMSGrad with a constant step size minimizes the objective function in fewer or comparable iterations compared with RSGD, RASA-L, RASA-R and RASA-LR, regardless of the dataset.
For the LRMC problem, RAdam with a diminishing step size minimizes the objective function in fewer or about the same number of iterations compared with RSGD, RASA-L and RASA-R, irrespective of the dataset.
As reported in \citet{sakai2021riemannian}, RAMSGrad demonstrates robust performance regardless of the initial step size. For the
PCA and LRMC problems (where LRMC adopts a diminishing step size), a grid search was conducted, allowing the selection of an initial step size that yielded even better results. This likely contributed to the strong performance of RAMSGrad across datasets.
However, for the LRMC problem, RAMSGrad may underperform relative to RASA, depending on the dataset.
It is worth noting that the objective function in the LRMC problem is complex, and assumptions such as retraction $L$-smoothness may not hold in this case. This could explain the less favorable results observed for RAMSGrad in this particular setting.
Moreover, to support our theoretical analyses, we show that increasing the batch size decreases the number of steps needed to a certain index. In particular, Table \ref{tab:mnist-b-c} (resp. Table \ref{tab:mnist-b-d}) supports the result shown in Theorem \ref{thm:constant}, demonstrating that the proposed algorithms with a constant (resp. diminishing) step size require fewer iterations to reduce the gradient norm below the threshold as the batch size increases. Tables \ref{tab:mnist-b-c} and \ref{tab:mnist-b-d} confirm that larger batch sizes lead to faster convergence in terms of the number of iterations, aligning with our theoretical analysis. Similar tables (Tables \ref{tab:coil100-b-c}--\ref{tab:jester-b-d}) for the remaining datasets and experiments can be found in Appendix \ref{apx:batch-comparisons}.
Note that the choice of batch size plays a critical role in balancing the computational cost per iteration and the total number of iterations required for convergence.
This is related to the concept of the critical batch size, as discussed in \citet{sakai2024convergence}, which minimizes a stochastic first-order oracle (SFO).
However, estimating this optimal batch size is challenging due to the dependency on unknown parameters such as the gradient Lipschitz constant and the variance of the stochastic gradient.
To address this challenge, an effective approach may involve dynamically increasing the batch size at regular intervals during optimization, such as doubling or tripling it after a fixed number of iterations. This method has theoretical guarantees for convergence (see Theorems \ref{thm:constant_1} and \ref{thm:diminishing_1}) and can adaptively approach an efficient balance between computational cost and convergence speed.

\section{Mathematical preliminaries}
Let $\real^d$ be a $d$-dimensional Euclidean space
with inner product $\ip{x}{y}_2:=x^\top y$,
which induces the norm $\norm{\cdot}_2$.
Let $\real_{++}$ be the set of positive real numbers,
i.e., $\real_{++}:=\{x\in\real\mid x>0\}$.
$I_d$ denotes a $d\times d$ identity matrix.
For square matrices $X,Y\in\real^{d\times d}$,
we write $X\prec Y$ (resp. $X\preceq Y$)
if $Y-X$ is a positive-definite (resp. positive-semidefinite) matrix.
For two matrices $X$ and $Y$ of the same dimension,
$X\odot Y$ denotes the Hadamard product, i.e., element-wise product.
Let $\max(X,Y)$ be the element-wise maximum.
Let $\mathcal{S}^d$ (resp. $\mathcal{S}^d_{+}$, $\mathcal{S}^d_{++}$)
be the set of $d\times d$ symmetric
(resp. symmetric positive-semidefinite,
symmetric positive-definite) matrices,
i.e., $S^d:=\{X\in\real^{d\times d}\mid X^\top=X\}$,
$\mathcal{S}^d_{+}:=\{X\in\real^{d\times d}\mid X\succeq O\}$
and $\mathcal{S}^d_{++}:=\{X\in\real^{d\times d}\mid X\succ O\}$.
Let $\mathcal{D}^d$ be the set of $d\times d$ diagonal matrices.
Let $\mathcal{O}_d$ be the orthogonal group, i.e.,
$\mathcal{O}_d:=\{X\in\real^{d\times d}\mid X^\top X=I_d\}$.
We denote the trace of a square matrix $A$ by $\tr(A)$.
Let $\mathcal{O}$ be Landau's symbol; i.e.,
for a sequence $(x_k)_{k=0}^\infty$, we write $x_k=\mathcal{O}(g(k))$ as $k\to\infty$ if there exist a constant $C>0$ and an index $k_0$ such that $|x_k|\leq C|g(k)|$ for all $k\geq k_0$.

In this paper, we focus on Riemannian manifolds \citep{sakai1996riemannian} embedded in Euclidean space.
Let $M$ be an embedded submanifold of $\real^d$.
Moreover, let $T_xM$ be the tangent space at a point $x\in M$ and
$TM$ be the tangent bundle of $M$.
Let $0_x$ be the zero element of $T_xM$.
The inner product $\ip{\cdot}{\cdot}_2$ of a Euclidean space $\real^d$
induces a Riemannian metric $\ip{\cdot}{\cdot}_x$ of $M$ at $x\in M$
according to $\ip{\xi}{\eta}_x=\ip{\xi}{\eta}_2=\xi^\top\eta$
for $\xi,\eta\in T_xM\subset T_x\real^d\cong\real^d$.
The norm of $\eta\in T_xM$ is defined as
$\norm{\eta}_x=\sqrt{\eta^\top\eta}=\norm{\eta}_2$.
Let $P_x:T_x\real^d\cong\real^d\to T_xM$ be the orthogonal projection
onto $T_xM$ (see \citet{absil2008optimization}).
For a smooth map $F:M\to N$ between two manifolds $M$ and $N$,
$\mathrm{D}F(x):T_xM\to T_{F(x)}N$
denotes the derivative of $F$ at $x\in M$.
The Riemannian gradient $\grad f(x)$ of a
smooth function $f:M\to\real$ at $x\in M$ is defined
as a unique tangent vector at $x$ satisfying
$\ip{\grad f(x)}{\eta}_x=\mathrm{D}f(x)[\eta]$ for any $\eta\in T_xM$.

\begin{definition}[Retraction]\label{def:retraction}
Let $M$ be a manifold.
Any smooth map $R:TM\to M$ is called a retraction on $M$
if it has the following properties.
\begin{itemize}
\item $R_x(0_x)=x$ for all $x\in M$;
\item With the canonical identification $T_{0_x}T_xM\cong T_xM$, $\mathrm{D}R_x(0_x)=\id_{T_xM}:T_xM\to T_xM$ for all $x\in M$,
\end{itemize}
where $R_x$ denotes the restriction of $R$ to $T_xM$.
\end{definition}

\subsection{Examples}
The unit sphere $\mathbb{S}^{d-1}:=\{x\in\real^d\mid \norm{x}_2=1\}$
is an embedded manifold of $\real^d$.
The tangent space $T_x\mathbb{S}^{n-1}$ at $x\in\mathbb{S}^{d-1}$ is
given by $T_x\mathbb{S}^{n-1}=\{\eta\in\real^d\mid\eta^\top x=0\}$.
The induced Riemannian metric on $\mathbb{S}^{d-1}$ is given by
$\ip{\xi}{\eta}_x=\ip{\xi}{\eta}_2:=\xi^\top\eta$ for
$\xi,\eta\in T_x\mathbb{S}^{d-1}$.
The orthogonal projection $P_x:\real^d\to T_x\mathbb{S}^{d-1}$ onto
the tangent space $T_x\mathbb{S}^{n-1}$ is
given by $P_x(\eta)=(I_d-xx^\top)\eta$ for $x\in\mathbb{S}^{d-1}$
and $\eta\in T_x\mathbb{S}^{d-1}$.

An important example is the Stiefel manifold \citep[Chapter 3.3.2]{absil2008optimization}, which is defined as
$\St(p,n):=\{X\in\real^{n\times p}\mid X^\top X=I_p\}$ for $n\geq p$.
$\St(p,n)$ is an embedded manifold of $\real^{n\times d}$.
The tangent space $T_x\St(p,n)$ at $X\in\St(p,n)$ is given by
\begin{align*}
    T_X\St(p,n)=\{\eta\in\real^{n\times p}\mid
    X^\top\eta+\eta^\top X=O\}.
\end{align*}
The induced Riemannian metric on $\St(p,n)$ is given by
$\ip{\xi}{\eta}_X=\tr(\xi^\top\eta)$ for $\xi,\eta\in T_X\St(p,n)$.
The orthogonal projection onto
the tangent space $T_X\St(p,n)$ is given by
$P_X(\eta)=\eta-X\sym(X^\top\eta)$
for $X\in\St(p,n)$, $\eta\in T_X\St(p,n)$, where $\sym(A):=(A+A^\top)/2$.
The Stiefel manifold $\St(p,n)$ reduces to the orthogonal groups when
$n=p$, i.e., $\St(p,p)=\mathcal{O}_p$.

Moreover, we will also consider the Grassmann manifold \citep[Chapter 3.4.4]{absil2008optimization} $\Gr(p,n):=\St(p,n)/\mathcal{O}_p$.
Let $X\in\St(p,n)$ be a representative of
$[X]:=\{XQ\mid Q\in\mathcal{O}_p\}\in\Gr(p,n)$.
We denote the horizontal lift of $\eta\in T_{[X]}\Gr(p,n)$ at $X$
by $\bar{\eta}_X\in T_X\St(p,n)$.
The Riemannian metric of the Grassmann manifold $\Gr(p,n)$ is endowed with
$\ip{\xi}{\eta}_{[X]}:=\ip{\bar{\xi}_X}{\bar{\eta}_X}_2$
for $\xi,\eta\in T_{[X]}\Gr(p, n)$.
The orthogonal projection onto the tangent space
$T_{[X]}\Gr(p,n)$ is defined through
\begin{align*}
    \overline{P_{[X]}(\eta)}=(I_n-XX^\top)\bar{\eta}_X,
\end{align*}
for $[X]\in\Gr(p,n)$ and $\eta\in T_{[X]}\Gr(p,n)$.
Note that while the Grassmann manifold itself is not an embedded submanifold, in practical implementations, computations involving it is often performed via the Stiefel manifold. This allows the proposed method to be applied. Such an approach was used in previous studies, including \citet{kasai2019riemannian, hu2024riemannian}.

\subsection{Riemannian stochastic optimization problem}
We focus on minimizing an objective function $f:M\to\real$ of the form,
\begin{align*}
    f(x)=\frac{1}{N}\sum_{i=1}^Nf_i(x),
\end{align*}
where $f_i$ is a smooth function for $i=1,\ldots,N$.
We use the mini-batch strategy as follows
(see \citet{iiduka2024theoretical} for detail).
$s_{k,i}$ is a random variable generated from the $i$-th sampling
at the $k$-th iteration, and
$\bs_k:=(s_{k,1},\ldots,s_{k,b_k})^\top$
is independent of $(x_k)_{k=1}^\infty$,
where $b_k$ $(\leq N)$ is the batch size.
To simplify the notation, we denote the expectation $\mathbb{E}_{\bs_k}$
with respect to $\bs_k$ by $\mathbb{E}_k$.
From the independence of $\bs_1,\bs_2,\ldots,\bs_k$,
we can define the total expectation $\ex$
by $\ex_1\ex_2\cdots\ex_k$.
We define the mini-batch stochastic gradient $\grad f_{B_k}(x_k)$
of $f$ at the $k$-th iteration by
\begin{align}\label{eq:bgrad}
    \grad f_{B_k}(x_k):=\frac{1}{b_k}\sum_{i=1}^{b_k} \grad f_{s_{k,i}}(x_k).
\end{align}
Our main objective is to find a stationary point $x_\star\in M$
satisfying $\grad f(x_\star)=0_{x_\star}$.

\subsection{Proposed general framework of Riemannian adaptive methods}
\citet{reddi2018convergence}
provided a general framework of adaptive gradient
methods in Euclidean space.
We devised Algorithm \ref{alg:general} by
generalizing that framework to an embedded manifold of $\real^d$.
The main difference from the Euclidean setting is
computing the projection of $H_k^{-1}m_k$
onto the tangent space $T_{x_k}M$
by the orthogonal projection $P_{x_k}$.
Algorithm \ref{alg:general} requires sequences of maps,
$(\phi_k)_{k=1}^\infty$ and $(\psi_k)_{k=1}^\infty$,
such that
$\phi_k:T_{x_1}M\times\cdots\times T_{x_k}M\to\real^d$ and
$\psi_k:T_{x_1}M\times\cdots\times 
T_{x_k}M\to\mathcal{D}^d\cap\mathcal{S}^d_{++}$, respectively.
Note that Algorithm \ref{alg:general} is still abstract because the
maps $(\phi_k)_{k=1}^\infty$ and $(\psi_k)_{k=1}^\infty$
are not specified.
Algorithm \ref{alg:general} is the extension of
a general framework in Euclidean space
proposed by Reddi, Kale and Kumar
\citep{reddi2018convergence}.
In the Euclidean setting (i.e., $M=\real^d$),
the orthogonal projection $P_{x_k}$ yields an identity map and
this corresponds to the Euclidean version of the general framework.

\begin{algorithm}
\caption{The general framework of Riemannian adaptive
optimization methods on an embedded submanifold of $\real^d$.
\label{alg:general}}
\begin{algorithmic}[1]
\Require
Initial point $x_1 \in M$, retraction $R:TM\to M$,
step sizes $(\alpha_k)_{k=1}^{\infty}\subset\real_{++}$,
sequences of maps $(\phi_k)_{k=1}^\infty$,
$(\psi_k)_{k=1}^\infty$.
\Ensure Sequence $(x_k)_{k=1}^{\infty} \subset M$.
\State $k\leftarrow 1$.
\Loop
\State $g_k=\grad f_{B_k}(x_k)$.
\State $m_k=\phi_k(g_1,\ldots,g_k)\in\real^d$.
\State $H_k=\psi_k(g_1,\ldots,g_k)
\in\mathcal{D}^d\cap\mathcal{S}^d_{++}$.
\State $x_{k+1}=R_{x_k}(-\alpha_kP_{x_k}(H_k^{-1}m_k))$.
\State $k\leftarrow k + 1$.
\EndLoop
\end{algorithmic}
\end{algorithm}

By setting the maps $(\phi_n)_{n=1}^\infty$ and $(\psi_n)_{n=1}^\infty$ in Algorithm \ref{alg:general} to those used by an adaptive optimization method in Euclidean space, any adaptive optimization method can be extended to Riemannian manifolds.
In the following, we show some examples of extending the optimization algorithm in Euclidean space to Riemannian manifolds by using Algorithm \ref{alg:general}.

Here, SGD is the most basic method; it uses
\begin{align*}
    \phi_k(g_1,\ldots,g_k)=g_k,\quad
    \psi_k(g_1,\ldots,g_k)=I_d.    
\end{align*}
Algorithm \ref{alg:general} with these maps corresponds
to RSGD \citep{bonnabel2013stochastic} in the Riemannian setting.
AdaGrad \citep{duchi2011adaptive}, the first adaptive gradient method
in Euclidean space that  propelled research on adaptive methods,
uses the sequences of maps $\phi_k(g_1,\ldots,g_k)=g_k$ and
\begin{align*}
    v_k&=v_{k-1}+g_k\odot g_k, \\
    \psi_k(g_1,\ldots,g_k)&=\diag(\sqrt{v_{k,1}},\ldots,
\sqrt{v_{k,d}})+\epsilon I_d,
\end{align*}
where $v_0=0\in\real^d$ and $\epsilon>0$.
Here, we will denote the $i$-th component of $v_k$ by $v_{k,i}$.
The exponential moving average variant of AdaGrad is often used in deep-learning
training. The basic variant is RMSProp \citep{hinton2012neural},
which uses the sequences of maps $\phi_k(g_1,\ldots,g_k)=g_k$ and 
\begin{align*}
    v_k&=\beta_2v_{k-1}+(1-\beta_2)g_k\odot g_k, \\
    \psi_k(g_1,\ldots,g_k)&=\diag(\sqrt{v_{k,1}},\ldots,
\sqrt{v_{k,d}})+\epsilon I_d,
\end{align*}
where $v_0=0\in\real^d$ and $\epsilon>0$.
Both Algorithm \ref{alg:general} with these maps and cRMSProp \citep{roy2018geometry}
can be considered extensions of RMSProp to Riemannian manifolds.
They differ from each other in that parallel transport is needed
to compute the search direction of cRMSProp, but it is not needed in our method. 

Adam \citep{kingma2015adam} is one of the most common variants;
it uses the sequence of maps, 
\begin{align}\label{eq:Adam-mk}
    m_k=\beta_1m_{k-1}+(1-\beta_1)g_k, \quad
    \phi_k(g_1,\ldots,g_k)=\frac{m_k}{1-\beta_1^{k+1}},
\end{align}
and
\begin{align}
    v_k=\beta_2v_{k-1}+(1-\beta_2)g_k\odot g_k, \quad
    \hat{v}_k=\frac{v_k}{1-\beta_2^{k+1}}, \nonumber \\
    \psi_k(g_1,\ldots,g_k)=\diag(\sqrt{\hat{v}_{k,1}},\ldots,
    \sqrt{\hat{v}_{k,d}})+\epsilon I_d, \label{eq:Adam-Hk}
\end{align}
where $m_0=0\in\real^d$ and $v_0=0\in\real^d$. $\beta_1=0.9$, $\beta_2=0.999$ and
$\epsilon=10^{-8}$ are typically recommended values. Moreover, within the general
framework (Algorithm \ref{alg:general}),
we propose the following algorithm as an extension of AMSGrad
\citep{reddi2018convergence} in Euclidean space.

\begin{algorithm}
\caption{AMSGrad on an embedded submanifold of $\real^d$.
\label{alg:AMSGrad}}
\begin{algorithmic}[1]
\Require
Initial point $x_1\in M$, retraction $R:TM\to M$,
step sizes $(\alpha_k)_{k=1}^{\infty}\subset\real_{++}$,
mini-batch sizes $(b_k)_{k=1}^\infty \subset \mathbb{R}_{++}$
hyperparameters $\beta_1,\beta_2\in[0,1)$, $\epsilon>0$.
\Ensure Sequence $(x_k)_{k=1}^{\infty} \subset M$.
\State Set $m_0=0$, $v_0=0$ and $\hat{v}_0=0$.
\State $k\leftarrow 1$.
\Loop
\State $g_k=\grad f_{B_k}(x_k)$.
\State $m_k=\beta_1m_{k-1}+(1-\beta_1)g_k$.
\State $v_k=\beta_2v_{k-1}+(1-\beta_2)g_k\odot g_k$.
\State $\hat{v}_k=\max(\hat{v}_{k-1},v_k)$.
\State $H_k=\diag(\sqrt{\hat{v}_{k,1}},\ldots,
\sqrt{\hat{v}_{k,d}})+\epsilon I_d$.
\State $x_{k+1}=R_{x_k}(-\alpha_kP_{x_k}(H_k^{-1}m_k))$.
\State $k\leftarrow k + 1$.
\EndLoop
\end{algorithmic}
\end{algorithm}

\section{Convergence analysis}
\subsection{Assumptions and useful lemmas}\label{sec:asm}
We make the following Assumptions \ref{asm:main} \ref{asm:unbiased}--\ref{asm:Lipschitz}.
\ref{asm:unbiased} and \ref{asm:variance} include the standard conditions.
\ref{asm:bound} assumes the boundedness of the gradient.
\ref{asm:Lipschitz} is an assumption on the Lipschitz continuity of the gradient.
\ref{asm:below} assumes that a lower bound exists.
\begin{assumption}\label{asm:main}
Let $(x_k)_{k=1}^\infty$ be a sequence generated
by Algorithm \ref{alg:general}.
\begin{enumerate}[label=(A\arabic*)]
\item\label{asm:unbiased}
$\ex_k[\grad f_{s_{k,i}}(x_k)]=\grad f(x_k)$
for all $k\geq 1$ and $i=1,\ldots,b_k$.
\item\label{asm:variance}
There exists $\sigma^2>0$ such that
\begin{align*}
    \ex_k\left[\norm{\grad f_{s_{k,i}}(x_k)-\grad f(x_k)}_2^2\right]
    \leq\sigma^2,
\end{align*}
for all $k\geq 1$ and $i=1,\ldots,b_k$.
\item\label{asm:bound}
There exists $G,B>0$ such that $\norm{\grad f(x_k)}_2\leq G$
and $\norm{\grad f_{B_k}(x_k)}_2\leq B$ for all $k\geq 1$.
\item\label{asm:Lipschitz}
There exists a constant $L>0$ such that
\begin{align*}
    \abs{\mathrm{D}(f\circ R_x)(\eta)[\eta]-\mathrm{D}f(x)
    [\eta]}\leq L\norm{\eta}_2^2,
\end{align*}
for all $x\in M$, $\eta\in T_xM$.
\item\label{asm:below} $f$ is bounded below by $f_\star\in\real$.
\end{enumerate}
\end{assumption}

Assumptions \ref{asm:unbiased} and \ref{asm:variance} are satisfied when the probability distribution is uniform (see Appendix \ref{apx:asm-variance}) or when the full gradient is used, i.e., when the batch size equals the total number of data. \ref{asm:unbiased} and \ref{asm:variance} have been used to analyze adaptive methods in Euclidean space \citep[Pages 3, 12, and 16]{zaheer2018adaptive} and on the Hadamard manifold \citep[(2.2) and (2.3)]{sakai2024convergence}.
Assumptions \ref{asm:bound}--\ref{asm:below} are often used in both Euclidean spaces and Riemannian manifolds. Intuitively, in practical applications, these assumptions are satisfied because the search space of the algorithm is bounded. If $f$ is unbounded on an unbounded search space $X$, then \ref{asm:below} does not hold. Although there is a possibility such that \ref{asm:Lipschitz} is satisfied, \ref{asm:variance} would not hold, even when the probability distribution is uniform.
When $f$ is unbounded, there is no minimizer of $f$. Hence, \ref{asm:below} is essential in order to analyze optimization methods. Moreover, \ref{asm:Lipschitz} is needed in order to use a retraction $L$-smooth (Proposition \ref{prp:ret_Lsmooth}) that is an extension of the Euclidean-type descent lemma. \ref{asm:bound} holds when the manifold is compact \citep{absil2008optimization} or through a slight modification of the objective function and the algorithm \citep{kasai2018riemannian}.

It is known that if Assumption \ref{asm:main} \ref{asm:Lipschitz} holds,
so does the following Proposition \ref{prp:ret_Lsmooth}.
This property is known as retraction $L$-smoothness
(see \citet{huang2015broyden, kasai2018riemannian} for details). 

\begin{proposition}\label{prp:ret_Lsmooth}
    Suppose that Assumption \ref{asm:main} \ref{asm:Lipschitz} holds.
    Then,
    \begin{align*}
        f(R_{x}(\eta))\leq f(x)+\ip{\grad f(x)}{\eta}_2
        +\frac{L}{2}\norm{\eta}_2^2,
    \end{align*}
    for all $x\in M$ and $\eta\in T_xM$.
\end{proposition}

\subsection{Convergence analysis of Algorithm \ref{alg:general}}
The main difficulty in analyzing the convergence of adaptive gradient methods
is due to the stochastic momentum $m_k=\phi_k(g_1,\ldots,g_k)$.
As a way to overcome this challenge in Euclidean space,
\citet{zhou2024convergence, yan2018unified, chen2019convergence}
defined a new sequence $z_k$,
\begin{align*}
z_k=x_k+\frac{\beta_1}{1-\beta_1}(x_k-x_{k-1}).
\end{align*}
However,
since the embedded submanifold is not closed under addition in Euclidean space,
this strategy does not work in the Riemannian setting.
Therefore, by following the policy of \citet{zaheer2018adaptive},
let us analyze the case in which $\phi_k(g_1,\ldots,g_k)=g_k$.
To simplify the notation, we denote the $i$-th component of $g_k$
(resp. $v_k$, $\hat{v}_k$) by $g_{k,i}$ (resp. $v_{k,i}$, $\hat{v}_{k,i}$). 

\begin{lemma}\label{lem:embedded_Lipsitz}
Suppose that Assumption \ref{asm:main} \ref{asm:Lipschitz} holds.
If $\phi_k(g_1,\ldots,g_k)=g_k$ and
$H_k^{-1}\preceq\nu I_d$ for all $k\geq 1$ and some $\nu>0$,
then the sequence $(x_k)_{k=1}^{\infty}\subset M$ generated by 
Algorithm \ref{alg:general} satisfies
\begin{align*}
    f(x_{k+1})&\leq f(x_k)
    +\ip{\grad f(x_k)}{-\alpha_kH_k^{-1}g_{k}}_2
    +\frac{L\alpha_k^2\nu^2}{2}\norm{g_{k}}^2_2,
\end{align*}
for all $k\geq 1$.
\end{lemma}
\begin{proof}
    See Appendix \ref{apx:embedded_Lipsitz}.
\end{proof}

\begin{theorem}\label{thm:main}
Suppose that Assumptions \ref{asm:main}
\ref{asm:unbiased}--\ref{asm:below} hold.
Moreover, let us assume that $\alpha_{k+1}\leq\alpha_k$, $\phi_k(g_1,\ldots,g_k)=g_k$,
$\alpha_kH_k^{-1}\succeq\alpha_{k+1}H_{k+1}^{-1}$ and
there exist $\mu,\nu>0$ such that $\mu I_d\preceq H_k^{-1}\preceq\nu I_d$
for all $k\geq 1$.
Then, the sequence $(x_k)_{k=1}^{\infty}\subset M$ generated by 
Algorithm \ref{alg:general} satisfies
\begin{align*}
    \sum_{k=1}^K\alpha_k\left(\mu
    -\frac{L\alpha_k\nu^2}{2}\right)
    \ex\left[\norm{\grad f(x_k)}_2^2\right]
    \leq C_1+C_2\sum_{k=1}^K \frac{\alpha_k^2}{b_k},
\end{align*}
for some constant $C_1,C_2>0$.
\end{theorem}

\textbf{Remark: }Since Algorithm \ref{alg:AMSGrad} satisfies
$\hat{v}_{k+1,i}:=\max(\hat{v}_{k,i},v_{k+1,i})\geq\hat{v}_{k,i}$,
it together with $\alpha_{k+1}\leq\alpha_k$, leads to
$\alpha_k/(\sqrt{\hat{v}_{k,i}}+\epsilon)
\geq\alpha_{k+1}/(\sqrt{\hat{v}_{k+1,i}}+\epsilon)$.
Moreover, from Lemma \ref{lem:v_bound},
\begin{align*}
    \frac{1}{B+\epsilon}\leq\frac{1}{\sqrt{\hat{v}_{k,i}}+\epsilon}
    \leq\frac{1}{\epsilon},
\end{align*}
which implies
$(B+\epsilon)^{-1}I_d\preceq H_k^{-1}\preceq\epsilon^{-1}I_d$.
Therefore, Algorithm \ref{alg:AMSGrad} satisfies the assumption
$\alpha_kH_k^{-1}\succeq\alpha_{k+1}H_{k+1}^{-1}$ and
$\mu I_d\preceq H_k^{-1}\preceq\nu I_d$ with
$\mu=(B+\epsilon)^{-1}$ and $\nu=\epsilon^{-1}$.

\begin{proof}
We denote $\grad f(x_k)$ by $g(x_k)$.
First, let us consider the case of $k=1$.
From Lemma \ref{lem:embedded_Lipsitz}, we have
\begin{align*}
    f(x_2)&\leq f(x_1)
    +\ip{g(x_1)}{-\alpha_1H_1^{-1}g_1}_2
    +\frac{L\alpha_1^2\nu^2}{2}\norm{g_1}^2_2.
\end{align*}
By taking $\ex_{1}[\cdot]$ of both sides, we obtain
\begin{align*}
    \ex_{1}[f(x_2)]&\leq f(x_1)
    +\ip{g(x_1)}{-\alpha_1\ex_{1}[H_1^{-1}g_1]}_2
    +\frac{L\alpha_1^2\nu^2}{2}\ex_{1}
    \left[\norm{g_1}^2_2\right] \\
    &\leq f(x_1)
    +\ip{g(x_1)}{-\alpha_1\ex_{1}[H_1^{-1}g_1]}_2
    +\frac{L\alpha_1^2\nu^2}{2}
    \left(\frac{\sigma^2}{b_1}+\norm{g(x_1)}_2^2\right)
\end{align*}
where the second inequality comes from Lemma \ref{lem:sn_bgrad}.
By taking $\ex[\cdot]$ of both sides and rearranging terms,
we get
\begin{align*}
    -\frac{L\alpha_1\nu^2}{2}
    \ex\left[\norm{g(x_1)}_2^2\right]
    \leq f(x_1)-\ex[f(x_2)]
    +\ip{g(x_1)}{-\alpha_1\ex[H_1^{-1}g_1]}_2
    +\frac{L\alpha_1^2\sigma^2\nu^2}{2b_1}.
\end{align*}
By adding $\alpha_1\mu G^2$ to both sides, we obtain
\begin{align*}
    \alpha_1\mu G^2
    -\frac{L\alpha_1\nu^2}{2}
    \ex\left[\norm{g(x_1)}_2^2\right]
    \leq f(x_1)-\ex[f(x_2)]
    +\frac{L\alpha_1^2\sigma^2\nu^2}{2b_1}
    +\underbrace{
    \ip{g(x_1)}{-\alpha_1\ex[H_1^{-1}g_1]}_2
    +\alpha_1\mu G^2}_{C_0}.
\end{align*}
Here, we note that
\begin{align*}
\alpha_1\mu\ex\left[\norm{g(x_1)}_2^2\right]\leq
\alpha_1\mu G^2.
\end{align*}
Therefore, we have
\begin{align}\label{eq:sum_gradient_bound1}
    \alpha_1\left(\mu-\frac{L\alpha_1\nu^2}{2}\right)
    \ex\left[\norm{g(x_1)}_2^2\right]
    \leq f(x_1)-\ex[f(x_2)]
    +\frac{L\alpha_1^2\sigma^2\nu^2}{2b_1}+C_0.
\end{align}
Next, let us consider the case of $k\geq 2$.
From Lemma \ref{lem:embedded_Lipsitz}, we have
\begin{align*}
    f(x_{k+1})
    \leq f(x_k)+\ip{g(x_k)}{-\alpha_{k-1}H_{k-1}^{-1}g_{k}}_2
    +\ip{g(x_k)}{(\alpha_{k-1}H_{k-1}^{-1}
    -\alpha_kH_k^{-1})g_{k}}_2
    +\frac{L\alpha_k^2\nu^2}{2}\norm{g_{k}}^2_2
\end{align*}
for all $k\geq 2$. From Assumption \ref{asm:main} \ref{asm:bound},
Lemma \ref{lem:linalg}, and
$\alpha_{k-1}H_{k-1}^{-1}-\alpha_kH_k^{-1}\succeq O$, we have
\begin{align*}
    f(x_{k+1})
    \leq f(x_k)+\ip{g(x_k)}{-\alpha_{k-1}H_{k-1}^{-1}g_{k}}_2
    +GB\tr(\alpha_{k-1}H_{k-1}^{-1}-\alpha_kH_k^{-1})
    +\frac{L\alpha_k^2\nu^2}{2}\norm{g_{k}}^2_2.
\end{align*}
By taking $\ex_{k}[\cdot]$ of both sides, we obtain
\begin{align*}
    \ex_{k}[f(x_{k+1})]
    &\leq f(x_k)+\ip{g(x_k)}{-\alpha_{k-1}H_{k-1}^{-1}
    \ex_{k}[g_{k}]}_2 \\
    &\quad+GB\ex_{k}
    [\tr(\alpha_{k-1}H_{k-1}^{-1}-\alpha_kH_k^{-1})]
    +\frac{L\alpha_k^2\nu^2}{2}\ex_{k}
    \left[\norm{g_{k}}^2_2\right] \\
    &\leq f(x_k)-\alpha_{k-1}\ip{g(x_k)}{H_{k-1}^{-1}g(x_k)}_2 \\
    &\quad+GB\ex_{k}
    [\tr(\alpha_{k-1}H_{k-1}^{-1}-\alpha_kH_k^{-1})]
    +\frac{L\alpha_k^2\nu^2}{2}\left(
    \frac{\sigma^2}{b_k}+\norm{g(x_k)}_2^2\right),
\end{align*}
where the first inequality comes from the independence of $H_{k-1}^{-1}$
for $\bs_k$ and the second inequality comes from Lemmas
\ref{lem:ex_bgrad} and \ref{lem:sn_bgrad}.
Here, since $H_{k-1}^{-1}\succeq\mu I_d$ and 
$\alpha_k\leq\alpha_{k-1}$, it follows that
\begin{align*}
    -\alpha_{k-1}\ip{g(x_k)}{H_{k-1}^{-1}g(x_k)}_2
    \leq -\alpha_k\mu\norm{g(x_k)}_2^2,
\end{align*}
which implies
\begin{align*}
    \ex_{k}[f(x_{k+1})]
    \leq f(x_k)
    -\alpha_k\left(\mu
    -\frac{L\alpha_k\nu^2}{2}\right)
    \norm{g(x_k)}_2^2
    +GB\ex_{k}
    [\tr(\alpha_{k-1}H_{k-1}^{-1}-\alpha_kH_k^{-1})]
    +\frac{L\alpha_k^2\sigma^2\nu^2}{2b_k}.
\end{align*}
By taking $\ex[\cdot]$ of both sides, we have
\begin{align*}
    \ex[f(x_{k+1})]\leq \ex[f(x_k)]
    -\alpha_k\left(\mu
    -\frac{L\alpha_k\nu^2}{2}\right)
    \ex\left[\norm{g(x_k)}_2^2\right]
    +GB\ex[\tr(\alpha_{k-1}H_{k-1}^{-1}-\alpha_kH_k^{-1})]
    +\frac{L\alpha_k^2\sigma^2\nu^2}{2b_k}.
\end{align*}
Rearranging the above inequality gives us
\begin{align}\label{eq:two2K}
    \alpha_k\left(\mu
    -\frac{L\alpha_k\nu^2}{2}\right)
    \ex\left[\norm{g(x_k)}_2^2\right]
    \leq\ex[f(x_k)]-\ex[f(x_{k+1})]
    +GB\ex[\tr(\alpha_{k-1}H_{k-1}^{-1}-\alpha_kH_k^{-1})]
    +\frac{L\alpha_k^2\sigma^2\nu^2}{2b_k}.
\end{align}
By summing \eqref{eq:two2K} from $k=2$ to $k=K$, we have
\begin{align*}
    &\sum_{k=2}^{K}\alpha_k\left(\mu
    -\frac{L\alpha_k\nu^2}{2}\right)
    \ex\left[\norm{g(x_k)}_2^2\right] \\
    &\quad\leq \ex[f(x_2)]-\ex[f(x_{K+1})]
    +GB\ex[\tr(\alpha_1H_1^{-1}-\alpha_KH_K^{-1})]
    +\sum_{k=2}^K\frac{L\alpha_k^2\sigma^2\nu^2}{2b_k}.
\end{align*}
Since $\mu I_d \preceq H_K^{-1} \preceq \nu I_d$ for all $k\geq 1$,
it follows that
$\tr(\alpha_1H_1^{-1})\leq \alpha_1\nu d$ and
$\tr(\alpha_KH_K^{-1})\geq 0$.
Here, we note that $\ex[f(x_{K+1})]\geq f_\star$,
from Assumption \ref{asm:main} \ref{asm:bound}.
Therefore, we have
\begin{align}\label{eq:sum_gradient_bound2}
    \sum_{k=2}^{K}\alpha_k\left(\mu
    -\frac{L\alpha_k\nu^2}{2}\right)
    \ex\left[\norm{g(x_k)}_2^2\right]
    \leq\ex[f(x_2)]-f_\star
    +GB\alpha_1\nu d
    +\sum_{k=2}^K\frac{L\alpha_k^2\sigma^2\nu^2}{2b_k}.
\end{align}
Here, by adding both sides of
\eqref{eq:sum_gradient_bound1} and \eqref{eq:sum_gradient_bound2},
we have
\begin{align*}
    \sum_{k=1}^K\alpha_k\left(\mu
    -\frac{L\alpha_k\nu^2}{2}\right)
    \ex\left[\norm{g(x_k)}_2^2\right]
    \leq \underbrace{f(x_1)-f_\star+C_0
    +GB\alpha_1\nu d}_{C_1}
    +\underbrace{\frac{L\sigma^2\nu^2}{2}}_{C_2}
    \sum_{k=1}^K \frac{\alpha_k^2}{b_k}.
\end{align*}
This completes the proof.
\end{proof}

Our convergence analysis (Theorem \ref{thm:main}) allows the proposed framework (Algorithm \ref{alg:general}) to use both constant and diminishing steps sizes.
Theorems \ref{thm:constant} and \ref{thm:diminishing} are convergence analyses of Algorithm \ref{alg:general}
with constant and diminishing steps sizes and a fixed mini-batch size,
respectively.

\begin{theorem}\label{thm:constant}
Under the assumptions in Theorem \ref{thm:main} and assuming that
the constant step size $\alpha_k:=\alpha$ satisfies $0<\alpha<2\mu L^{-1}\nu^{-2}$
and the mini-batch size $b_k := b$ satisfies $0< b \leq N$,
the sequence $(x_k)_{k=1}^{\infty}\subset M$
generated by Algorithm \ref{alg:general} satisfies
\begin{align*}
    \frac{1}{K}\sum_{k=1}^K\ex\left[\norm{\grad f(x_k)}_2^2\right]
    =\mathcal{O}\left(\frac{1}{K}+\frac{1}{b}\right).
\end{align*}
\end{theorem}
\begin{proof}
    We denote $\grad f(x_k)$ by $g(x_k)$.
    From Theorem \ref{thm:main}, we obtain
    \begin{align}\label{eq:c_sum_grad}
        \frac{1}{K}\sum_{k=1}^K
        \alpha\left(\mu-\frac{L\alpha\nu^2}{2}\right)
        \ex\left[\norm{g(x_k)}_2^2\right]
        \leq \frac{C_1}{K}+\frac{C_2\alpha^2}{b}.
    \end{align}
    Since $0<\alpha<2\mu L^{-1}\nu^{-2}$, it follows that
    $(2\alpha\mu-L\alpha^2\nu^2)/2>0$.
    Therefore, dividing both sides of \eqref{eq:c_sum_grad}
    by $(2\alpha\mu-L\alpha^2\nu^2)/2$ gives
    \begin{align*}
        \frac{1}{K}\sum_{k=1}^K
        \ex\left[\norm{g(x_k)}_2^2\right]
        \leq\frac{2C_1}{2\alpha\mu-L\alpha^2\nu^2}\cdot\frac{1}{K}
        +\frac{2C_2\alpha^2}{2\alpha\mu-L\alpha^2\nu^2}\cdot\frac{1}{b}.
    \end{align*}
    This completes the proof.
\end{proof}

\begin{theorem}\label{thm:diminishing}
Under the assumptions in Theorem \ref{thm:main} and assuming that
the diminishing step size $\alpha_k:=\alpha/\sqrt{k}$ satisfies $\alpha\in (0,1]$
and the mini-batch size $b_k := b$ satisfies $0< b \leq N$,
the sequence $(x_k)_{k=1}^{\infty}\subset M$
generated by Algorithm \ref{alg:general} satisfies 
\begin{align*}
    \frac{1}{K}\sum_{k=1}^K\ex\left[\norm{\grad f(x_k)}_2^2\right]
    =\mathcal{O}\left(\left(1+\frac{1}{b}\right)
    \frac{\log K}{\sqrt{K}}\right).
\end{align*}
\end{theorem}
\begin{proof}
    We denote $\grad f(x_k)$ by $g(x_k)$.
    Since $(\alpha_k)_{k=1}^{\infty}$ satisfies
    $\alpha_k\to0$ $(k\to\infty)$,
    there exists a natural number $k_0\geq 1$
    such that, for all $k\geq 1$,
    if $k\geq k_0$, then $0<\alpha_k<2\mu L^{-1}\nu^{-2}$.
    Therefore, we obtain
    \begin{align*}
        0<\mu-\frac{L\alpha_k\nu^2}{2}<\mu,
    \end{align*}
    for all $k\geq k_0$.
    From Theorem \ref{thm:main}, we have
    \begin{align*}
        \sum_{k=k_0}^K\alpha_k\left(\mu
        -\frac{L\alpha_k\nu^2}{2}\right)
        \ex\left[\norm{g(x_k)}_2^2\right]
        \leq C_1+\frac{C_2}{b}\sum_{k=1}^K\alpha_k^2
        -\sum_{k=1}^{k_0-1}\alpha_k\left(\mu
        -\frac{L\alpha_k\nu^2}{2}\right)
        \ex\left[\norm{g(x_k)}_2^2\right],
    \end{align*}
    for all $K\geq k_0$.
    Since $(\alpha_k)_{k=1}^\infty$ is monotone decreasing and
    $\alpha_k>0$, we obtain
    \begin{align*}
        \alpha_K\left(\mu-\frac{L\alpha_{k_0}\nu^2}{2}\right)
        \sum_{k=k_0}^K\ex\left[\norm{g(x_k)}_2^2\right]
        \leq C_1+\frac{C_2}{b}\sum_{k=1}^K\alpha_k^2
        +\sum_{k=1}^{k_0-1}L\alpha_k^2\nu^2
        \ex\left[\norm{g(x_k)}_2^2\right].
    \end{align*}
    Dividing both sides of this inequality by
    $2^{-1}K\alpha_K(2\mu-L\alpha_{k_0}\nu^2)>0$ yields
    \begin{align*}
        \frac{1}{K}\sum_{k=k_0}^K
        \ex\left[\norm{g(x_k)}_2^2\right]
        &\leq\frac{2}{K\alpha_K(2\mu-L\alpha_{k_0}\nu^2)}
        \left(C_1+\frac{C_2}{b}\sum_{k=1}^K\alpha_k^2
        +\sum_{k=1}^{k_0-1}L\alpha_k^2\nu^2
        \ex\left[\norm{g(x_k)}_2^2\right]\right)\\
        &=\frac{1}{K\alpha_K}\cdot
        \underbrace{\frac{2}{2\mu-L\alpha_{k_0}\nu^2}
        \left(C_1+\sum_{k=1}^{k_0-1}L\alpha_k^2\nu^2
        \ex\left[\norm{g(x_k)}_2^2\right]\right)}_{C_3}\\
        &\quad+\frac{1}{bK\alpha_K}\cdot
        \underbrace{\frac{2C_2}{2\mu-L\alpha_{k_0}\nu^2}}_{C_4}
        \sum_{k=1}^K\alpha_k^2.
    \end{align*}
    From this and $\alpha_K:=\alpha/\sqrt{K}<1$, we obtain
    \begin{align*}
        \frac{1}{K}\sum_{k=1}^K\ex\left[\norm{g(x_k)}_2^2\right]
        &\leq\frac{1}{K\alpha_K}\left(C_3+
        \frac{C_4}{b}\sum_{k=1}^K\alpha_k^2
        \right)+\frac{1}{K\alpha_K}\sum_{k=1}^{k_0-1}
        \ex\left[\norm{g(x_k)}_2^2\right] \\
        &=\frac{1}{\alpha\sqrt{K}}\left(C_3+\sum_{k=1}^{k_0-1}
        \ex\left[\norm{g(x_k)}_2^2\right]+
        \frac{C_4}{b}\sum_{k=1}^K\alpha_k^2\right).
    \end{align*}
    From $\alpha\in(0,1]$, we have that
    \begin{align*}
        \sum_{k=1}^K\alpha_k^2=\sum_{k=1}^K\frac{\alpha^2}{k}
        \leq \sum_{k=1}^K\frac{1}{k}
        \leq 1+\int_1^K\frac{dt}{t}=1+\log K.
    \end{align*}
    Therefore,
    \begin{align*}
        \frac{1}{K}\sum_{k=1}^K\ex\left[\norm{g(x_k)}_2^2\right]
        \leq\frac{1}{\alpha\sqrt{K}}\left(C_3+\sum_{k=1}^{k_0-1}
        \ex\left[\norm{g(x_k)}_2^2\right]+
        \frac{C_4}{b}+\frac{C_4}{b}\log K\right).
    \end{align*}
    This completes the proof.
\end{proof}

Theorems \ref{thm:constant_1} and \ref{thm:diminishing_1} are convergence analyses of Algorithm \ref{alg:general}
with constant and diminishing steps sizes and increasing mini-batch sizes,
respectively.

\begin{theorem}\label{thm:constant_1}
Under the assumptions in Theorem \ref{thm:main} and assuming that
the constant step size $\alpha_k:=\alpha$ satisfies $0<\alpha<2\mu L^{-1}\nu^{-2}$
and the mini-batch size $b_k$ satisfies $0< b_k \leq N$ and $\sum_{k=1}^\infty b_k^{-1} < \infty$,
the sequence $(x_k)_{k=1}^{\infty}\subset M$
generated by Algorithm \ref{alg:general} satisfies
\begin{align*}
    \min_{k=1,\ldots, K} \ex\left[\norm{\grad f(x_k)}_2^2\right] \leq \frac{1}{K}\sum_{k=1}^K\ex\left[\norm{\grad f(x_k)}_2^2\right]
    =\mathcal{O}\left(\frac{1}{K}\right).
\end{align*}
\end{theorem}
\begin{proof}
    We denote $\grad f(x_k)$ by $g(x_k)$.
    From Theorem \ref{thm:main}, we obtain
    \begin{align}\label{eq:c_sum_grad_1}
        \frac{1}{K} \sum_{k=1}^K
        \alpha\left(\mu-\frac{L\alpha\nu^2}{2}\right)
        \ex\left[\norm{g(x_k)}_2^2\right]
        \leq \frac{C_1}{K} + \frac{C_2\alpha^2}{K} \sum_{k=1}^K \frac{1}{b_k}.
    \end{align}
    Since $0<\alpha<2\mu L^{-1}\nu^{-2}$, it follows that
    $(2\alpha\mu-L\alpha^2\nu^2)/2>0$.
    Therefore, dividing both sides of \eqref{eq:c_sum_grad_1}
    by $(2\alpha\mu-L\alpha^2\nu^2)/2$ gives
    \begin{align*}
        \frac{1}{K}\sum_{k=1}^K
        \ex\left[\norm{g(x_k)}_2^2\right]
        \leq\frac{2C_1}{2\alpha\mu-L\alpha^2\nu^2}\cdot\frac{1}{K}
        +\frac{2C_2\alpha^2}{2\alpha\mu-L\alpha^2\nu^2} \cdot\frac{1}{K}\sum_{k=1}^K \frac{1}{b_k}.
    \end{align*}
    This completes the proof.
\end{proof}

An example of $b_k$ such that $\sum_{k=1}^\infty b_k^{-1} < \infty$ used in Theorem \ref{thm:constant_1} is such that the mini-batch size is multiplied by $\delta > 1$ every step, i.e., 
\begin{align}\label{exponential_bs}
    b_t = \delta^t b_0.
\end{align}
For example, the mini-batch size defined by \eqref{exponential_bs} with $\delta = 2$ makes batch size double each step.
Moreover, the mini-batch size defined by \eqref{exponential_bs} and the diminishing step size $\alpha_k := \alpha/\sqrt{k}$, where $\alpha \in (0,1]$, satisfy that $\sum_{k=1}^\infty \alpha_k^2 b_k^{-1} < \infty$ (see the assumption in Theorem \ref{thm:diminishing_1}).

\begin{theorem}\label{thm:diminishing_1}
Under the assumptions in Theorem \ref{thm:main} and assuming that
the diminishing step size $\alpha_k:=\alpha/\sqrt{k}$ satisfies $\alpha\in (0,1]$
and the mini-batch size $b_k$ satisfies $0< b_k \leq N$ and $\sum_{k=1}^\infty \alpha_k^2 b_k^{-1} < \infty$,
the sequence $(x_k)_{k=1}^{\infty}\subset M$
generated by Algorithm \ref{alg:general} satisfies 
\begin{align*}
    \min_{k=1,\ldots, K} \ex\left[\norm{\grad f(x_k)}_2^2\right] \leq \frac{1}{\sum_{k=1}^K \alpha_k}\sum_{k=1}^K \alpha_k \ex\left[\norm{\grad f(x_k)}_2^2\right]
    =\mathcal{O}\left(
    \frac{1}{\sqrt{K}}\right).
\end{align*}
\end{theorem}
\begin{proof}
    We denote $\grad f(x_k)$ by $g(x_k)$.
    Since $(\alpha_k)_{k=1}^{\infty}$ satisfies
    $\alpha_k\to0$ $(k\to\infty)$,
    there exists a natural number $k_0\geq 1$
    such that, for all $k\geq 1$,
    if $k\geq k_0$, then $0<\alpha_k<2\mu L^{-1}\nu^{-2}$.
    Therefore, we obtain
    \begin{align*}
        0<\mu-\frac{L\alpha_k\nu^2}{2}<\mu,
    \end{align*}
    for all $k\geq k_0$.
    From Theorem \ref{thm:main}, we have
    \begin{align*}
        \sum_{k=k_0}^K\alpha_k\left(\mu
        -\frac{L\alpha_k\nu^2}{2}\right)
        \ex\left[\norm{g(x_k)}_2^2\right]
        \leq C_1+C_2\sum_{k=1}^K \frac{\alpha_k^2}{b_k}
        -\sum_{k=1}^{k_0-1}\alpha_k\left(\mu
        -\frac{L\alpha_k\nu^2}{2}\right)
        \ex\left[\norm{g(x_k)}_2^2\right],
    \end{align*}
    for all $K\geq k_0$.
    Since $(\alpha_k)_{k=1}^\infty$ is monotone decreasing and
    $\alpha_k>0$, we obtain
    \begin{align*}
        \left(\mu-\frac{L\alpha_{k_0}\nu^2}{2}\right)
        \sum_{k=k_0}^K \alpha_k \ex\left[\norm{g(x_k)}_2^2\right]
        \leq C_1+ C_2 \sum_{k=1}^K \frac{\alpha_k^2}{b_k}
        +\sum_{k=1}^{k_0-1}L\alpha_k^2\nu^2
        \ex\left[\norm{g(x_k)}_2^2\right].
    \end{align*}
    Dividing both sides of this inequality by
    $2^{-1} (2\mu-L\alpha_{k_0}\nu^2)>0$ yields
    \begin{align*}
        \sum_{k=k_0}^K
        \alpha_k \ex \left[\norm{g(x_k)}_2^2\right]
        &\leq\frac{2}{2\mu-L\alpha_{k_0}\nu^2}
        \left(C_1+C_2\sum_{k=1}^K \frac{\alpha_k^2}{b_k}
        +\sum_{k=1}^{k_0-1}L\alpha_k^2\nu^2
        \ex\left[\norm{g(x_k)}_2^2\right]\right)\\
        &=
        \underbrace{\frac{2}{2\mu-L\alpha_{k_0}\nu^2}
        \left(C_1+\sum_{k=1}^{k_0-1}L\alpha_k^2\nu^2
        \ex\left[\norm{g(x_k)}_2^2\right]\right)}_{C_3}\\
        &\quad+
        \underbrace{\frac{2C_2}{2\mu-L\alpha_{k_0}\nu^2}}_{C_4}
        \sum_{k=1}^K \frac{\alpha_k^2}{b_k}.
    \end{align*}
    From this and $\sum_{k=1}^\infty\alpha_k^2{b_k}^{-1}<\infty$, we have that 
    \begin{align*}
        \sum_{k=1}^K \alpha_k \ex\left[\norm{g(x_k)}_2^2\right]
        &\leq \underbrace{C_3+\sum_{k=1}^{k_0-1}
        \alpha_k \ex\left[\norm{g(x_k)}_2^2\right]
        +
        C_4\sum_{k=1}^K \frac{\alpha_k^2}{b_k}}_{C_5},
    \end{align*}
    which implies that 
    \begin{align*}
        \frac{1}{\sum_{k=1}^K \alpha_k} \sum_{k=1}^K \alpha_k \ex\left[\norm{g(x_k)}_2^2\right]
        &\leq \frac{C_5}{\sum_{k=1}^K \alpha_k}.
    \end{align*}
    From $\alpha_k :=\alpha/\sqrt{k}$, we obtain
    \begin{align*}
       \sum_{k=1}^K\alpha_k =\sum_{k=1}^K\frac{\alpha}{\sqrt{k}}
        \geq \alpha \int_1^{K+1} \frac{dt}{\sqrt{t}}=2\alpha (\sqrt{K+1} - 1).
    \end{align*}
    Therefore,
    \begin{align*}
        \frac{1}{\sum_{k=1}^K \alpha_k}\sum_{k=1}^K \alpha_k \ex\left[\norm{g(x_k)}_2^2\right]
        \leq \frac{C_5}{2\alpha (\sqrt{K+1} - 1)}.
    \end{align*}
    This completes the proof.
\end{proof}

\section{Numerical experiments}
We experimentally compared our general framework of Riemannian adaptive optimization
methods (Algorithms \ref{alg:general}) with several
choices of $(\phi_n)_{n=1}^\infty$ and $(\psi_n)_{n=1}^\infty$
with the following algorithms: 
\begin{itemize}
    \item RSGD \citep{bonnabel2013stochastic}:
        Algorithm \ref{alg:general} with
        $\phi_k(g_1,\ldots,g_k)=g_k$ and
        $\psi_k(g_1,\ldots,g_k)=I_d$.
    \item RASA-LR, RASA-L, RASA-R \citep[Algorithm 1]{kasai2019riemannian}:
        $\beta=0.99$.
    \item RAdam: Algorithm \ref{alg:general}
        with $(\phi_n)_{n=1}^\infty$ defined by \eqref{eq:Adam-mk},
        $(\psi_n)_{n=1}^\infty$ defined by \eqref{eq:Adam-Hk},
        $\beta_1=0.9$, $\beta_2 = 0.999$ and $\epsilon=10^{-8}$.
    \item RAMSGrad: Algorithm \ref{alg:AMSGrad} with
        $\beta_1=0.9$, $\beta_2 = 0.999$ and $\epsilon=10^{-8}$.
\end{itemize}
We experimented with both constant and diminishing step sizes.
For each algorithm, we searched in the set $\{10^{-1},10^{-2},\ldots, 10^{-8}\}$
for the best initial step size $\alpha$ (both constant and diminishing).
Note that the constant (resp. diminishing) step size was determined
to be $\alpha_k=\alpha$ (resp. $\alpha_k=\alpha/\sqrt{k}$)
for all $k\geq 1$.
For each experiment, each data set was split into a training set and a test set.
The experiments used a MacBook Air (M1, 2020)
and macOS Monterey version 12.2 operating system.
The algorithms were written in Python 3.12.1 with the NumPy 1.26.0
package and the Matplotlib 3.9.1 package.
The Python implementations of the methods used in the numerical
experiments are available at
\url{https://github.com/iiduka-researches/202408-adaptive}. 

\subsection{Principal component analysis}
We applied the algorithms to a principal component
analysis (PCA) problem \citep{kasai2018riemannian, roy2018geometry}.
For $N$ given data points $x_1,\ldots,x_N\in\real^n$ and $p$ $(\leq n)$,
the PCA problem is equivalent to minimizing
\begin{align}\label{eq:pca}
    f(U):=\frac{1}{N}\sum_{i=1}^N\norm{x_i-UU^\top x_i}_2^2,
\end{align}
on the Stiefel manifold $\St(p,n)$.
Therefore, the PCA problem can be considered to be optimization problem
on the Stiefel manifold.
In the experiments, we set $p$ to $10$ and the batch size $b$ to $2^{10}$.
We used the QR-based retraction on the Stiefel manifold $\St(p,n)$
\citep[Example 4.1.3]{absil2008optimization}, which is defined by
\begin{align*}
    R_X(\eta):=\qf(X+\eta),
\end{align*}
for $X\in\St(p,n)$ and $\eta\in T_X\St(p,n)$,
where $\qf(\cdot)$ returns the $Q$-factor of the QR decomposition.

We evaluated the algorithms on training images of the MNIST dataset
\citep{lecun1998gradient} and the COIL100 dataset \citep{nene1996columbia}.
The MNIST dataset contains $28\times 28$ gray-scale images of handwritten digits. It has a training set of 60,000 images and a test set of 10,000 images.
We transformed every image into a 784-dimensional vector
and normalized its pixel values to lie in the range of $[0,1]$.
Thus, we set $N=60000$ and $n=784$.
The COIL100 dataset contains 7,200 normalized color camera images
of the 100 objects taken from different angles.
As in the previous study \citep{kasai2019riemannian},
we resized them to 32$\times$32 pixels.
Moreover, we split them into an $80\%$ training set and a $20\%$ test set.
Thus, we set $N=5760$ and $n=1024$.

Figure \ref{fig:mnist-fv-c} (resp. Figure \ref{fig:mnist-fv-d}) shows the performances of the algorithms with a constant (resp. diminishing) step size for the objective function values defined by \eqref{eq:pca} with respect to the number of iterations on the
training set of the MNIST dataset,
while Figures \ref{fig:mnist-fv-c-test} and \ref{fig:mnist-fv-d-test}
present those on the test set.
Moreover, Figures \ref{fig:coil-fv-c} and \ref{fig:coil-fv-d} show those on the training set of the COIL100 dataset, while
Figures \ref{fig:coil-fv-c-test} and \ref{fig:coil-fv-d-test}
present those on the test set.
Figure \ref{fig:mnist-gn-c} (resp. \ref{fig:mnist-gn-d}) presents the performances of the algorithms with a constant (resp. diminishing) step size for the norm of the gradient of objective function defined by \eqref{eq:pca} with respect to the number of iterations on the
training set of the MNIST dataset,
while Figures \ref{fig:mnist-gn-c-test} and \ref{fig:mnist-gn-d-test}
present those on the test set of the MNIST dataset.
Moreover, Figures \ref{fig:coil-gn-c} and \ref{fig:coil-gn-d} show those on the training set of the COIL100 dataset,
Figures \ref{fig:coil-gn-c-test} and \ref{fig:coil-gn-d-test}
present those on the test set.
The experiments were performed for three random initial points, and the thick line plots the average of all experiments. The area bounded by the maximum and minimum values is painted the same color as the corresponding line.
Note that the upper part of Figures \ref{fig:mnist-gn-c}, \ref{fig:mnist-gn-d}, \ref{fig:mnist-gn-c-test} and \ref{fig:mnist-gn-d-test} are cut off. The initial values of the gradient norm in these figures are approximately 38, highlighting that the gradient norm was minimized during the optimization process. Furthermore, our results differ from those presented in \citet{kasai2019riemannian} because of three key factors. In \citet{kasai2019riemannian}, the mini-batch size was fixed at 10, while in our experiments we used $2^{10}$. Furthermore, the scope of the grid search differs from that of \citet{kasai2019riemannian}.
Moreover, we experimented with a doubly increasing batch size every 100 steps from an initial batch size of $b_1=2^7$ and constant and diminishing step sizes.
The results are given in Appendix \ref{apx:inc}.

\begin{figure}[htbp]
\centering
\subfigure[constant learning rate]{
    \includegraphics[clip, width=0.4\columnwidth]{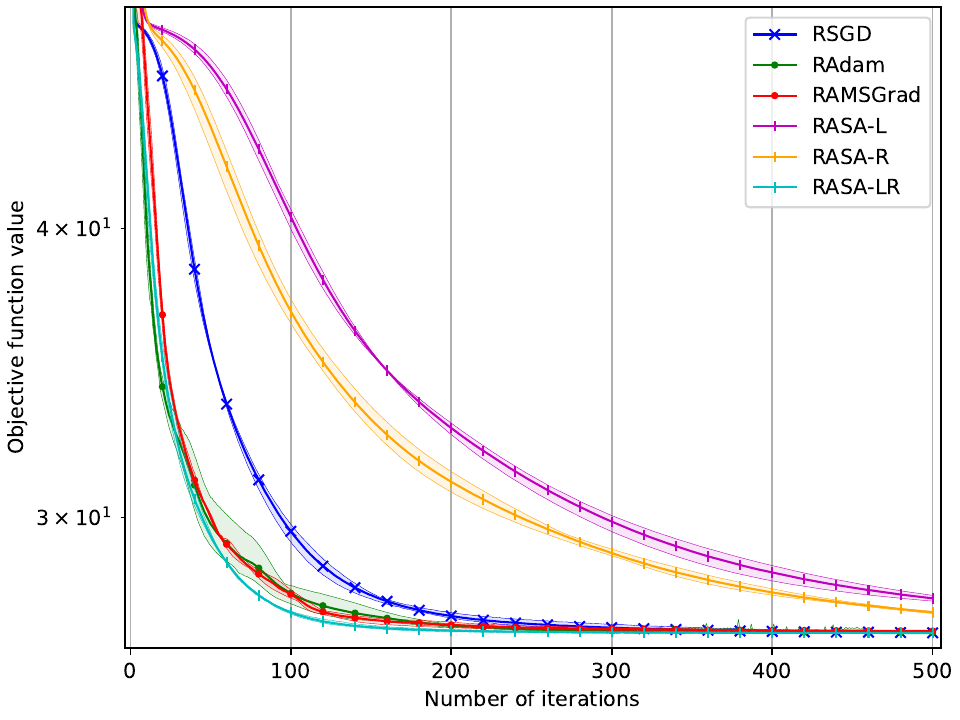}
    \label{fig:mnist-fv-c}
}    
\subfigure[diminishing learning rate]{
    \includegraphics[clip, width=0.4\columnwidth]{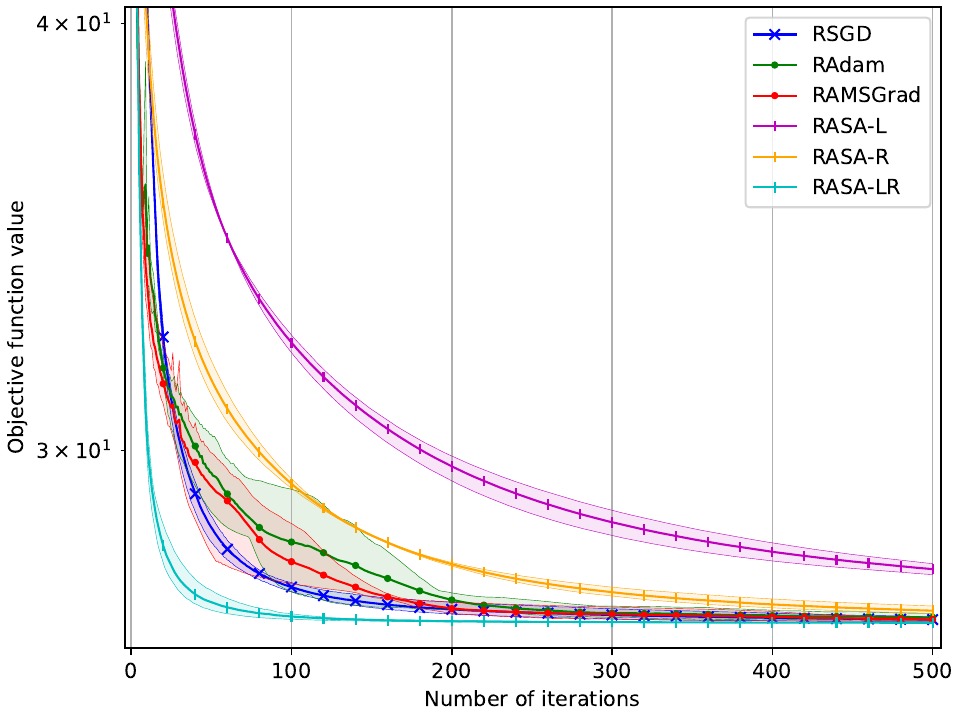}
    \label{fig:mnist-fv-d}
}
\caption{Objective function value defined by \eqref{eq:pca} versus number of iterations on the training set of the MNIST datasets.}
\end{figure}

\begin{figure}[htbp]
\centering
\subfigure[constant learning rate]{
    \includegraphics[clip, width=0.4\columnwidth]{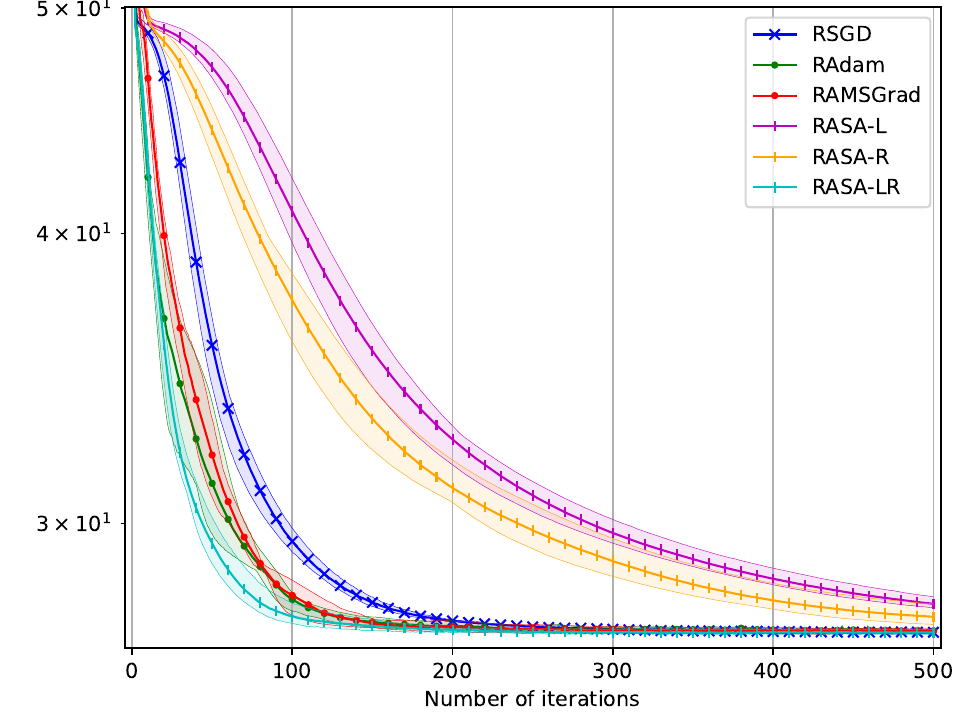}
    \label{fig:mnist-fv-c-test}
}    
\subfigure[diminishing learning rate]{
    \includegraphics[clip, width=0.4\columnwidth]{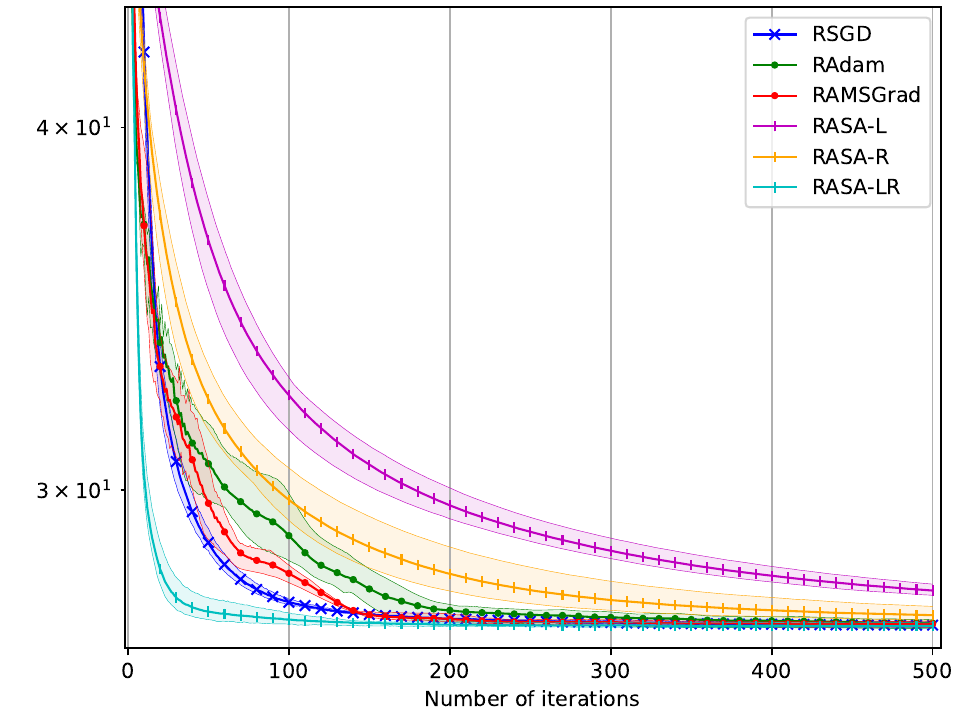}
    \label{fig:mnist-fv-d-test}
}
\caption{Objective function value defined by \eqref{eq:pca} versus number of iterations on the test set of the MNIST datasets.}
\end{figure}

\begin{figure}[htbp]
\centering
\subfigure[constant learning rate]{
    \includegraphics[clip, width=0.4\columnwidth]{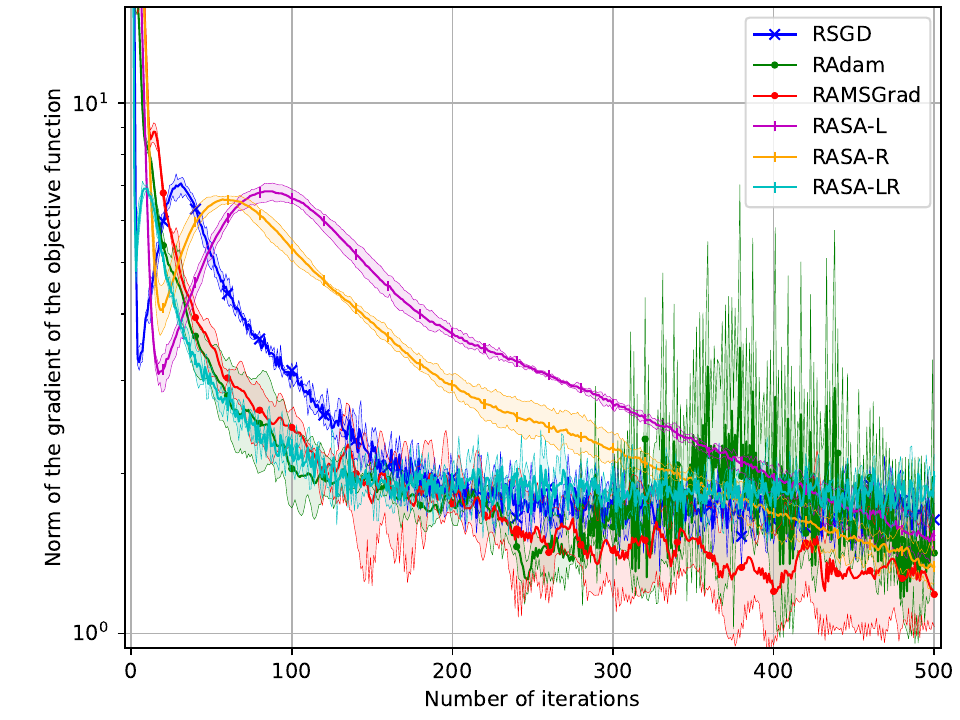}
    \label{fig:mnist-gn-c}
}    
\subfigure[diminishing learning rate]{
    \includegraphics[clip, width=0.4\columnwidth]{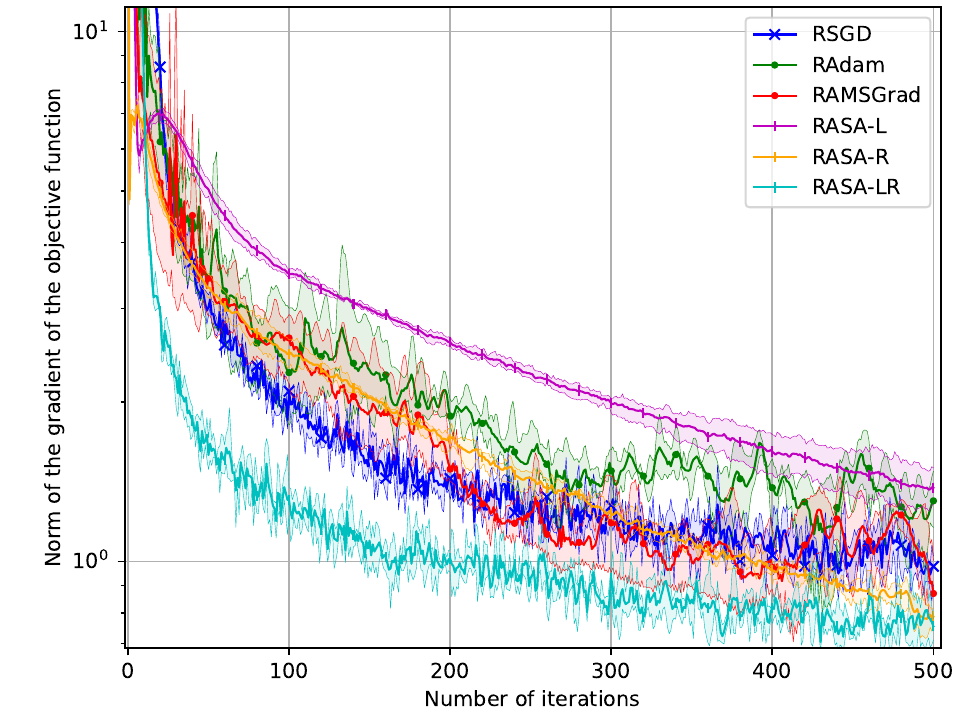}
    \label{fig:mnist-gn-d}
}
\caption{Norm of the gradient of objective function defined by \eqref{eq:pca} versus number of iterations on the training set of the MNIST datasets.}
\end{figure}

\begin{figure}[htbp]
\centering
\subfigure[constant learning rate]{
    \includegraphics[clip, width=0.4\columnwidth]{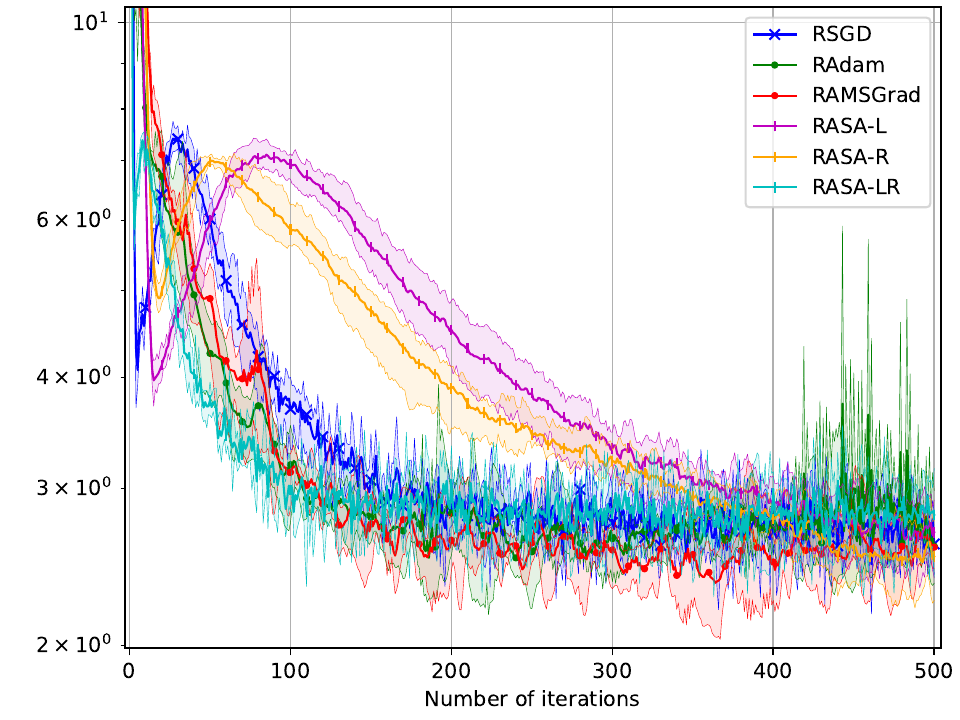}
    \label{fig:mnist-gn-c-test}
}    
\subfigure[diminishing learning rate]{
    \includegraphics[clip, width=0.4\columnwidth]{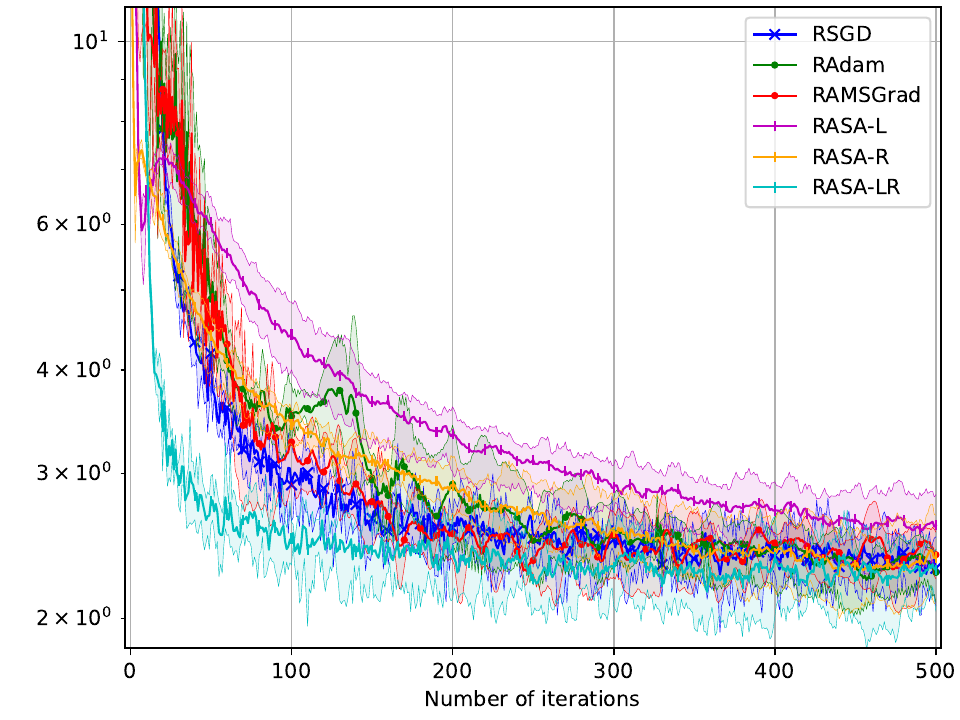}
    \label{fig:mnist-gn-d-test}
}
\caption{Norm of the gradient of objective function defined by \eqref{eq:pca} versus number of iterations on the test set of the MNIST datasets.}
\end{figure}

\begin{figure}[htbp]
\centering
\subfigure[constant learning rate]{
    \includegraphics[clip, width=0.4\columnwidth]{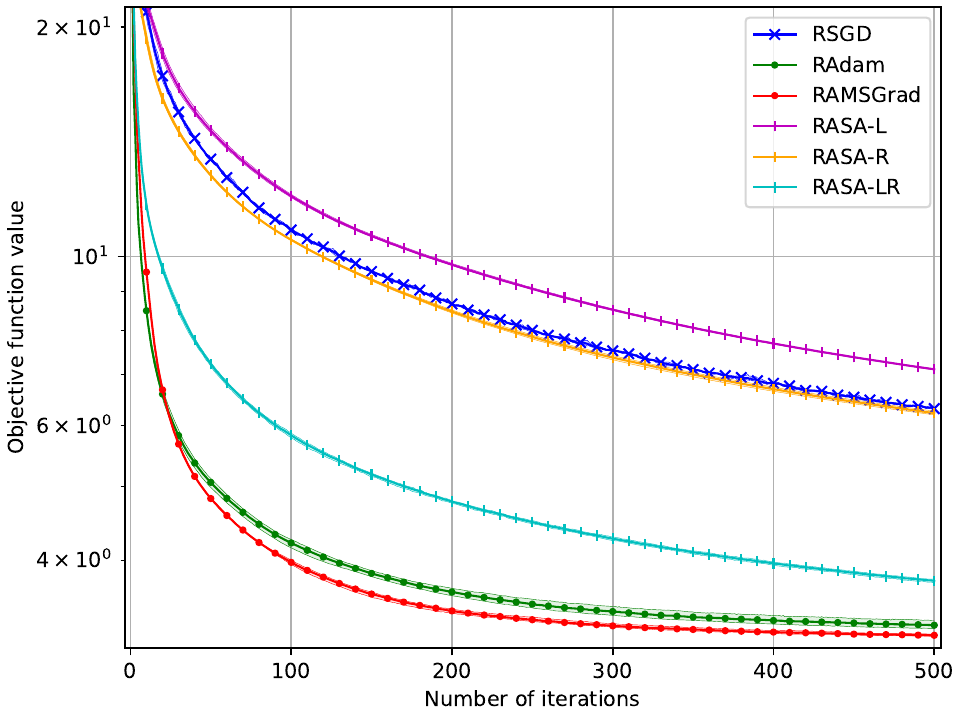}
    \label{fig:coil-fv-c}
}    
\subfigure[diminishing learning rate]{
    \includegraphics[clip, width=0.4\columnwidth]{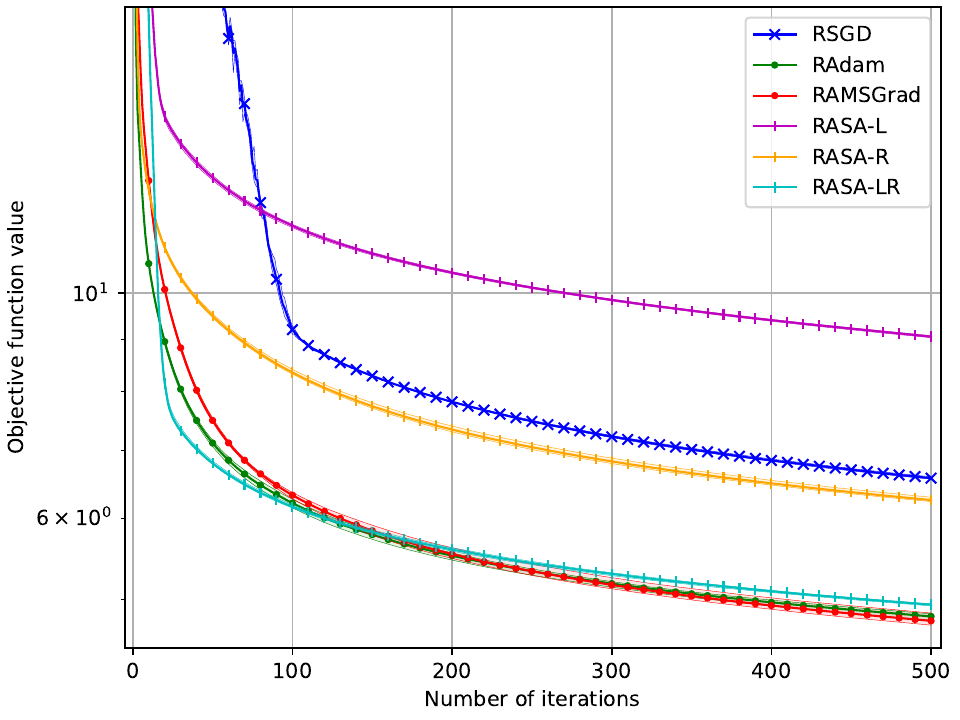}
    \label{fig:coil-fv-d}
}
\caption{Objective function value defined by \eqref{eq:pca} versus number of iterations on the training set of the COIL100 datasets.}
\end{figure}

\begin{figure}[htbp]
\centering
\subfigure[constant learning rate]{
    \includegraphics[clip, width=0.4\columnwidth]{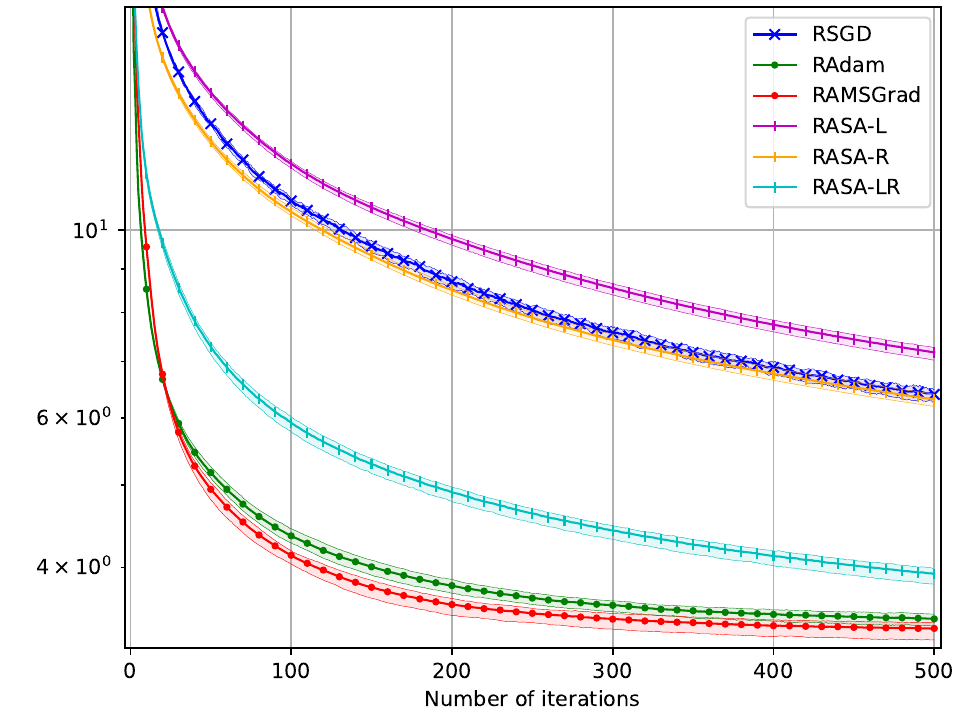}
    \label{fig:coil-fv-c-test}
}    
\subfigure[diminishing learning rate]{
    \includegraphics[clip, width=0.4\columnwidth]{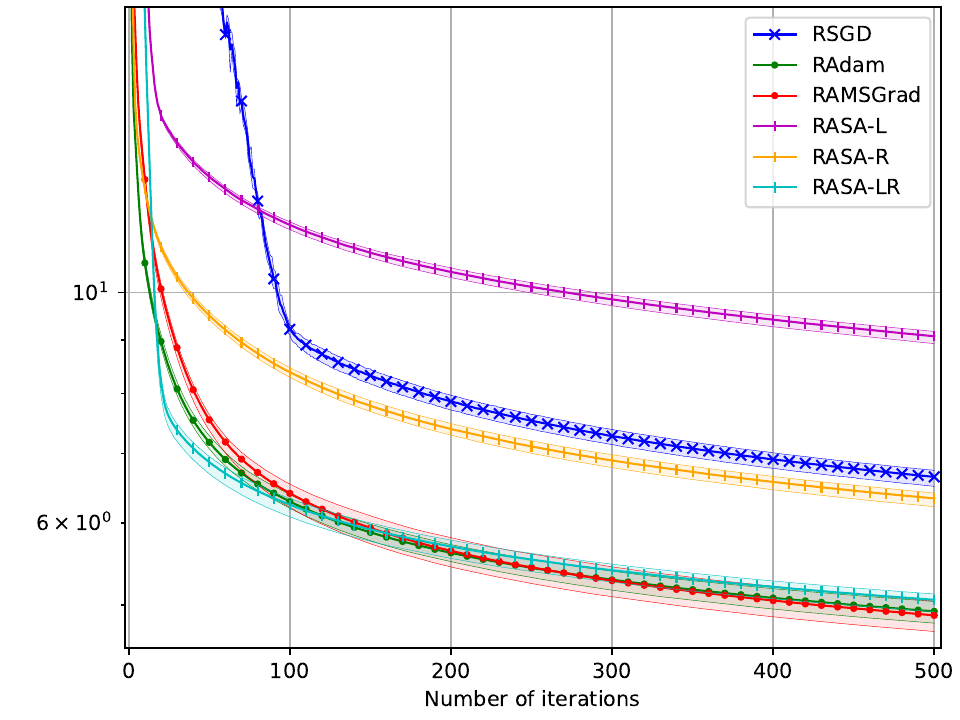}
    \label{fig:coil-fv-d-test}
}
\caption{Objective function value defined by \eqref{eq:pca} versus number of iterations on the test set of the COIL100 datasets.}
\end{figure}

\begin{figure}[htbp]
\centering
\subfigure[constant learning rate]{
    \includegraphics[clip, width=0.4\columnwidth]{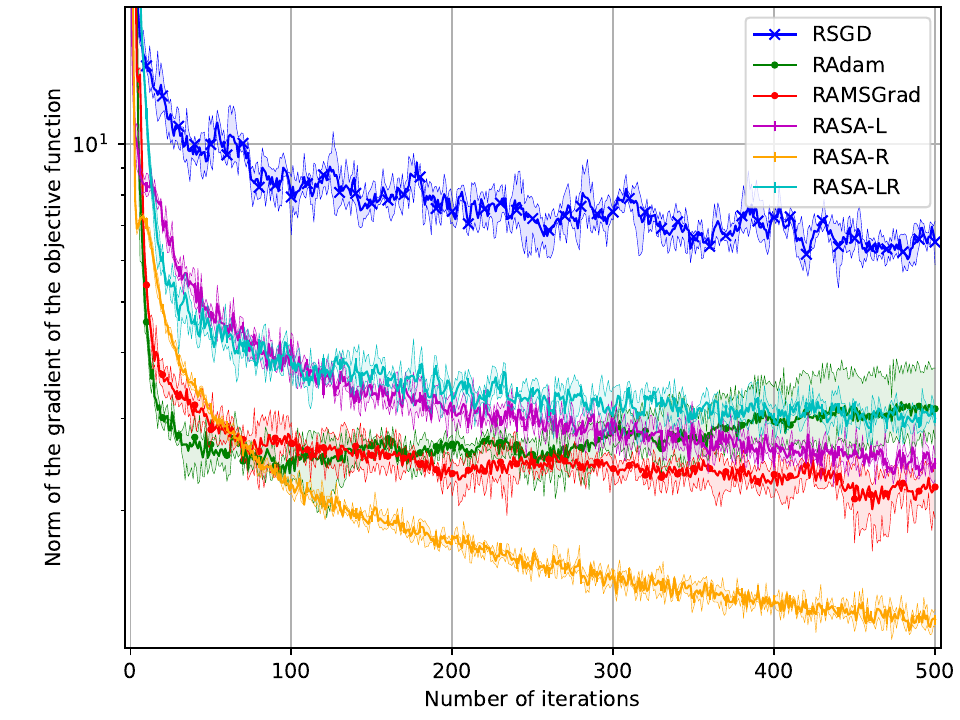}
    \label{fig:coil-gn-c}
}    
\subfigure[diminishing learning rate]{
    \includegraphics[clip, width=0.4\columnwidth]{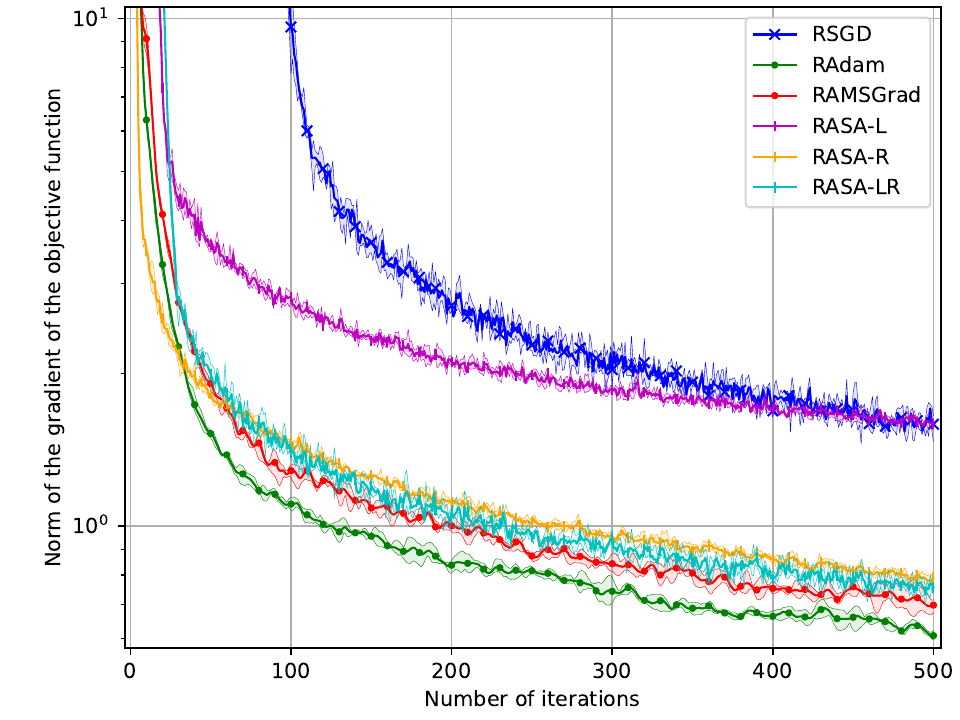}
    \label{fig:coil-gn-d}
}
\caption{Norm of the gradient of objective function defined by \eqref{eq:pca} versus number of iterations on the training set of the COIL100 datasets.}
\end{figure}

\begin{figure}[htbp]
\centering
\subfigure[constant learning rate]{
    \includegraphics[clip, width=0.4\columnwidth]{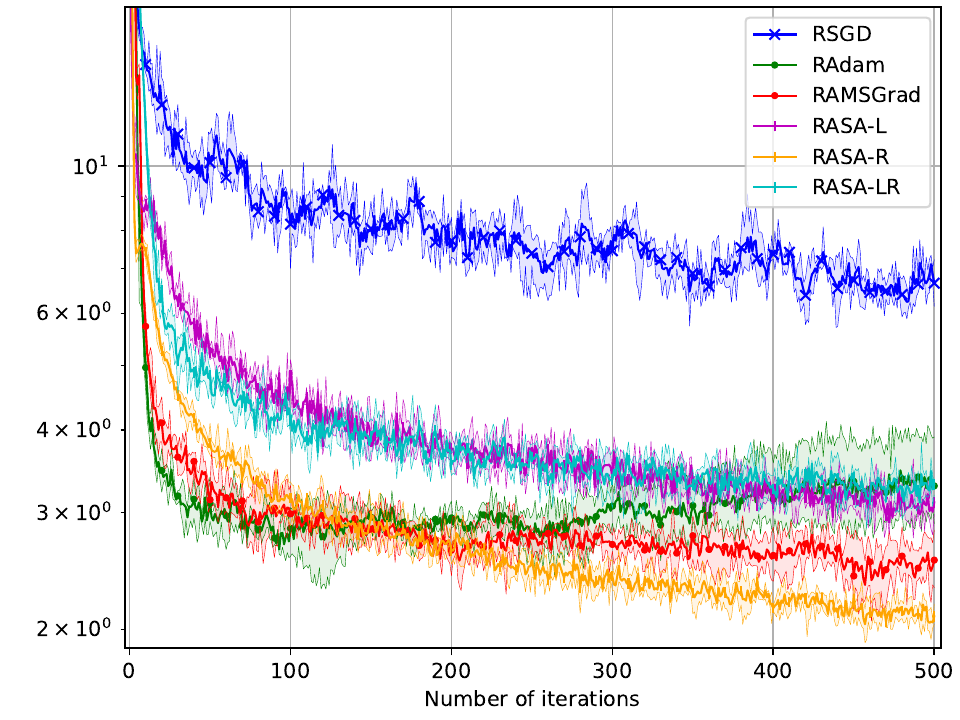}
    \label{fig:coil-gn-c-test}
}    
\subfigure[diminishing learning rate]{
    \includegraphics[clip, width=0.4\columnwidth]{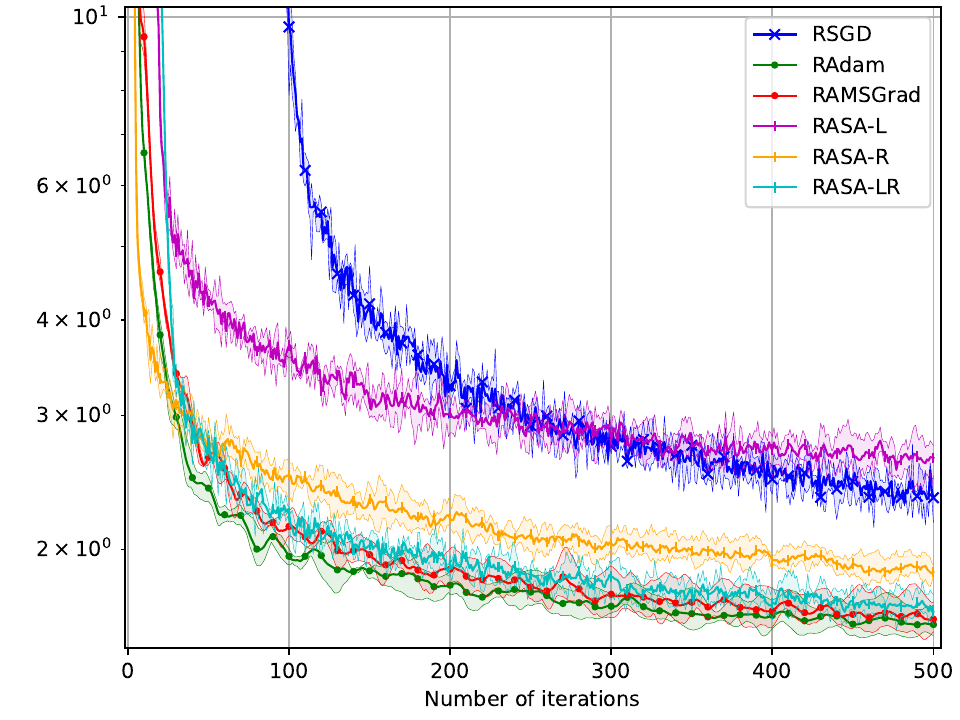}
    \label{fig:coil-gn-d-test}
}
\caption{Norm of the gradient of objective function defined by \eqref{eq:pca} versus number of iterations on the test set of the COIL100 datasets.}
\end{figure}

Figures \ref{fig:mnist-fv-c} and \ref{fig:mnist-fv-c-test} indicate that RAdam and RAMSGrad (Algorithm \ref{alg:AMSGrad}) performed comparably to RASA-LR in the sense of minimizing the objective function value.
Figures \ref{fig:mnist-fv-d} and \ref{fig:mnist-fv-d-test} indicate that RAdam and RAMSGrad (Algorithm \ref{alg:AMSGrad}) outperformed RASA-L and RASA-R.
Figures \ref{fig:mnist-gn-c} and \ref{fig:mnist-gn-c-test} show that RAMSGrad (Algorithm \ref{alg:AMSGrad}) performed better than RASA-LR in the sense of minimizing the full gradient norm of the objective function.
Figures \ref{fig:coil-fv-c} and \ref{fig:coil-fv-c-test} indicate that RAdam and RAMSGrad (Algorithm \ref{alg:AMSGrad}) had the best performance in the sense of minimizing the objective function value.
Figures \ref{fig:coil-fv-d} and \ref{fig:coil-fv-d-test} indicate that RAdam and RAMSGrad (Algorithm \ref{alg:AMSGrad}) performed comparably to RASA-LR.
Figures \ref{fig:coil-gn-c} and \ref{fig:coil-gn-c-test} shows that RAdam had the best performance in the sense of minimizing the full gradient norm of the objective function.
Figures \ref{fig:coil-gn-d} and \ref{fig:coil-gn-d-test} indicate that RAdam and RAMSGrad (Algorithm \ref{alg:AMSGrad}) performed comparably to RASA-R and RASA-LR.

Table \ref{tab:mnist-b-c} compares performances across different batch sizes, showing the number of iterations and CPU time (in seconds) required by each algorithm with a constant step size to reduce the gradient norm below 2 when solving the PCA problem on the MNIST dataset. As our analysis (Theorem \ref{thm:constant}) indicates, increasing the batch size reduces the number of iterations needed to decrease the norm.

\begin{table}[htbp]
\caption{Iterations and CPU time (seconds) required by each algorithm with a constant step size to reduce the gradient norm below 2 for solving the PCA problem on the MNIST dataset across different batch sizes. ``-'' indicates cases where the algorithm did not reach the threshold within the maximum allowed (1,000) iterations.}
\label{tab:mnist-b-c}
\centering
\begin{tabular}{c|c|cc}
\hline
& batch size $b$ & number of iterations & CPU time (seconds) \\ \hline\hline
\multirow{3}{*}{RSGD} & 256 & - & - \\
& 512 & 303 & 0.531 \\
& 1024 & 149 & 0.394 \\ \hline
\multirow{3}{*}{RAdam} & 256 & 391 & 0.532 \\
& 512 & 150 & 0.270 \\
& 1024 & 126 & 0.349 \\ \hline
\multirow{3}{*}{RAMSGrad} & 256 & 190 & 0.257 \\
& 512 & 140 & 0.251 \\
& 1024 & 114 & 0.323 \\ \hline
\multirow{3}{*}{RASA-L} & 256 & 493 & 1.228 \\
& 512 & 442 & 1.291 \\
& 1024 & 373 & 1.476 \\ \hline
\multirow{3}{*}{RASA-R} & 256 & 397 & 0.662 \\
& 512 & 371 & 0.809 \\
& 1024 & 300 & 0.901 \\ \hline
\multirow{3}{*}{RASA-LR} & 256 & - & - \\
& 512 & 289 & 0.862 \\
& 1024 & 113 & 0.447 \\ \hline
\end{tabular}
\end{table}

Table \ref{tab:mnist-b-d} compares performances across different batch sizes, showing the number of iterations and CPU time (in seconds) required by each algorithm with a diminishing step size to reduce the gradient norm below 2 when solving the PCA problem on the MNIST dataset. As our analysis (Theorem \ref{thm:diminishing}) indicates, increasing the batch size reduces the number of iterations needed to decrease the norm. Similar tables (Tables \ref{tab:coil100-b-c}--\ref{tab:jester-b-d}) for the remaining datasets and experiments can be found in Appendix \ref{apx:batch-comparisons}.

\begin{table}[htbp]
\caption{Iterations and CPU time (seconds) required by each algorithm with a diminishing step size to reduce the gradient norm below 2 for solving the PCA problem on the MNIST dataset across different batch sizes.}
\label{tab:mnist-b-d}
\centering
\begin{tabular}{c|c|cc}
\hline
& batch size $b$ & number of iterations & CPU time (seconds) \\ \hline\hline
\multirow{3}{*}{RSGD} & 256 & 239 & 0.315 \\
& 512 & 140 & 0.327 \\
& 1024 & 85 & 0.229 \\ \hline
\multirow{3}{*}{RAdam} & 256 & 292 & 0.434 \\
& 512 & 189 & 0.349 \\
& 1024 & 101 & 0.280 \\ \hline
\multirow{3}{*}{RAMSGrad} & 256 & 234 & 0.324 \\
& 512 & 224 & 0.437 \\
& 1024 & 141 & 0.384 \\ \hline
\multirow{3}{*}{RASA-L} & 256 & 409 & 1.008 \\
& 512 & 370 & 1.133 \\
& 1024 & 282 & 1.102 \\ \hline
\multirow{3}{*}{RASA-R} & 256 & 235 & 0.399 \\
& 512 & 194 & 0.418 \\
& 1024 & 156 & 0.472 \\ \hline
\multirow{3}{*}{RASA-LR} & 256 & 92 & 0.229 \\
& 512 & 63 & 0.192 \\
& 1024 & 36 & 0.145 \\ \hline
\end{tabular}
\end{table}

\subsection{Low-rank matrix completion}
We applied the algorithms to the low-rank matrix completion (LRMC) problem \citep{boumal2015low, kasai2019riemannian, hu2024riemannian}. The LRMC problem aims to recover a low-rank matrix from an incomplete matrix $X=(X_{ij})\in\real^{n\times N}$. We denote the set of observed entries by $\Omega\subset\{1,\ldots,n\}\times\{1,\ldots,N\}$, i.e., $(i,j)\in\Omega$ if and only if $X_{ij}$ is known. Here, we defined the orthogonal projection $P_{\Omega_i}:\real^n\to\real^n:a\mapsto P_{\Omega_i}(a)$ such that the $j$-th element of $P_{\Omega_i}(a)$ is $a_j$ if $(i,j)\in\Omega$, and 0 otherwise. Moreover, we defined $q_i:\real^{n\times p}\times\real^n\to\real^p$ as
\begin{align}\label{eq:lstsq}
    q_i(U,x):=\argmin_{a\in\real^p}\norm{P_{\Omega_i}(Ua-x)}_2,
\end{align}
for $i\geq 1$. By partitioning $X=(x_1,\ldots,x_N)$, the rank-$p$ LRMC problem is equivalent to minimizing
\begin{align}\label{eq:lmc}
    f(U):=\frac{1}{2N}\sum_{i=1}^N
    \norm{P_{\Omega_i}(Uq_i(U,x_i)-x_i)}^2_2,
\end{align}
on the Grassmann manifold $\Gr(p,n)$.
Therefore, the rank-$p$ LRMC problem can be considered to be
an optimization problem on the Grassmann manifold
(see \citet[Section 1]{hu2024riemannian} or
\citet[Section 6.3]{kasai2019riemannian} for details).

We evaluated the algorithms on the
MovieLens-1M\footnote{\url{https://grouplens.org/datasets/movielens/}} datasets \citep{harper2015movielens} and the
Jester\footnote{\url{https://grouplens.org/datasets/jester}} datasets for recommender systems. The MovieLens-1M datasets contains 1,000,209 ratings
given by 6,040 users on 3,952 movies.
Moreover, we split them into an $80\%$ training set and a $20\%$ test set.
Thus, we set $N=4832$ and $n=3952$.
The Jester dataset contains ratings of 100 jokes given by 24,983 users with scores from $-10$ to 10.
In addition, we split them into an $80\%$ training set and a $20\%$ test set.
Thus, we set $N=19986$ and $n=100$.

In the experiments, we set $p$ to $10$ and the batch size $b$ to $2^8$. We used \texttt{numpy.linalg.lstsq}\footnote{\url{https://numpy.org/doc/1.26/reference/generated/numpy.linalg.lstsq.html}} to solve the least squares problem \eqref{eq:lstsq}. We used a retraction based on a polar decomposition on the Grassmann manifold $\Gr(p,n)$ \citep[Example 4.1.3]{absil2008optimization}, which is defined through
\begin{align*}
    \overline{R_{[X]}(\eta)}:=(X+\bar{\eta}_X)
    (I_p+\bar{\eta}_X^\top\bar{\eta}_X)^{-\frac{1}{2}},
\end{align*}
for $[X]\in\Gr(p,n)$ and $\eta\in T_{[X]}\Gr(p,n)$.

Figure \ref{fig:ml1m-fv-c} (resp. Figure \ref{fig:ml1m-fv-d}) shows the performances of the algorithms with a constant (resp. diminishing) step size for objective function values defined by \eqref{eq:lmc} with respect to the number of iterations on the
training set of the MovieLens-1M dataset,
while Figures \ref{fig:ml1m-fv-c-test} and \ref{fig:ml1m-fv-d-test}
present those on the Jester dataset.
Moreover, Figures \ref{fig:jester-fv-c} and \ref{fig:jester-fv-d} show those on the training set of the Jester dataset, while
Figures \ref{fig:jester-fv-c-test} and \ref{fig:jester-fv-d-test}
present those on the test set of the Jester dataset.
Figure \ref{fig:ml1m-gn-c} (resp. \ref{fig:ml1m-gn-d}) shows the performances of the algorithms with a constant (resp. diminishing) step size for the norm of the gradient of the objective function defined by \eqref{eq:pca} with respect to the number of iterations on the
training set of the MovieLens-1M dataset,
while Figures \ref{fig:ml1m-gn-c-test} and \ref{fig:ml1m-gn-d-test}
present those on the test set of the MovieLens-1M dataset.
Moreover, Figures \ref{fig:jester-gn-c} and \ref{fig:jester-gn-d} show those on the training set of the Jester dataset, while
Figures \ref{fig:jester-gn-c-test} and \ref{fig:jester-gn-d-test}
present those on the test set of the Jester dataset.
The experiments were performed for three random initial points, and the thick line plots the average results of all experiments. The area bounded by the maximum and minimum values is painted the same color as the corresponding line.
Moreover, we experimented with a doubly increasing batch size every 20 steps from an initial batch size of $b_1=2^6$ and constant and diminishing step sizes.
The results are given in Appendix \ref{apx:inc}.

\begin{figure}[htbp]
\centering
\subfigure[constant learning rate]{
    \includegraphics[clip, width=0.4\columnwidth]{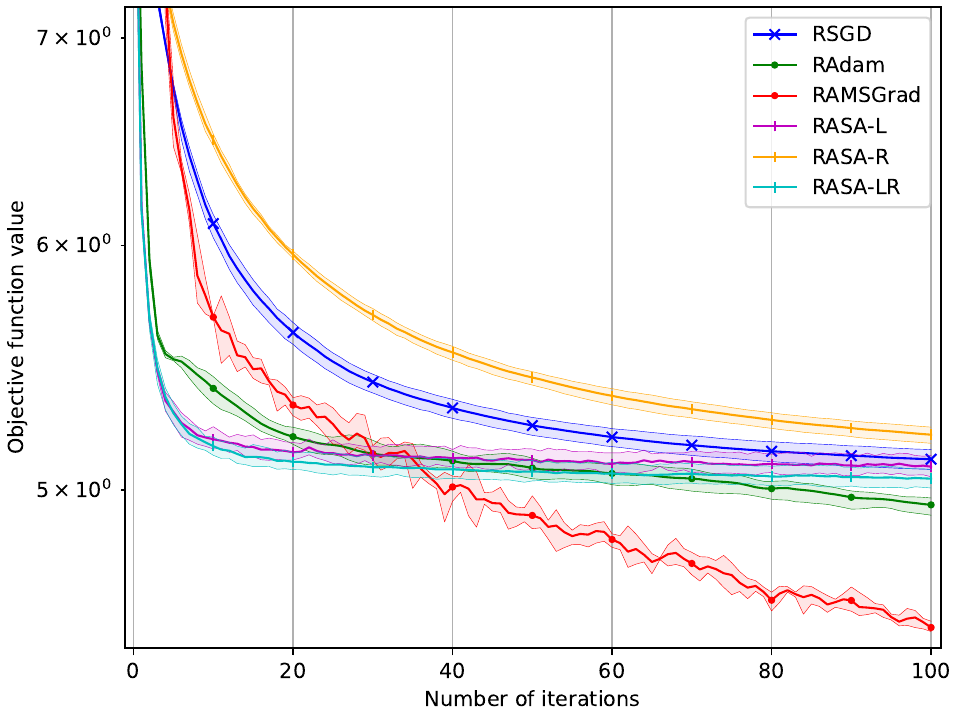}
    \label{fig:ml1m-fv-c}
}    
\subfigure[diminishing learning rate]{
    \includegraphics[clip, width=0.4\columnwidth]{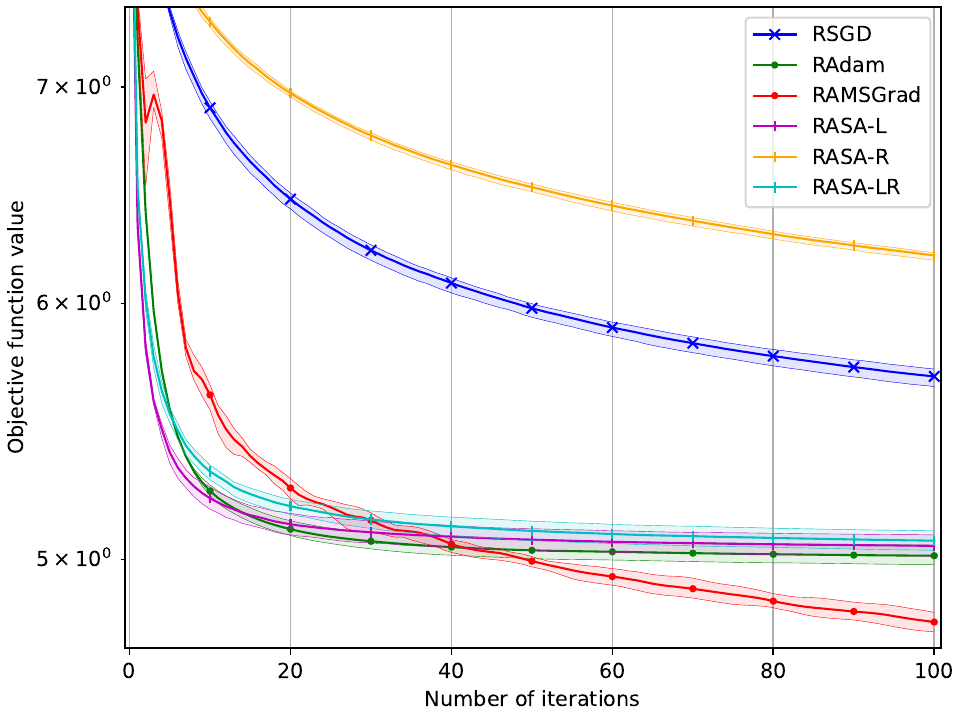}
    \label{fig:ml1m-fv-d}
}
\caption{Objective function value defined by \eqref{eq:lmc} versus number of iterations on the training set of the MovieLens-1M datasets.}
\end{figure}

\begin{figure}[htbp]
\centering
\subfigure[constant learning rate]{
    \includegraphics[clip, width=0.4\columnwidth]{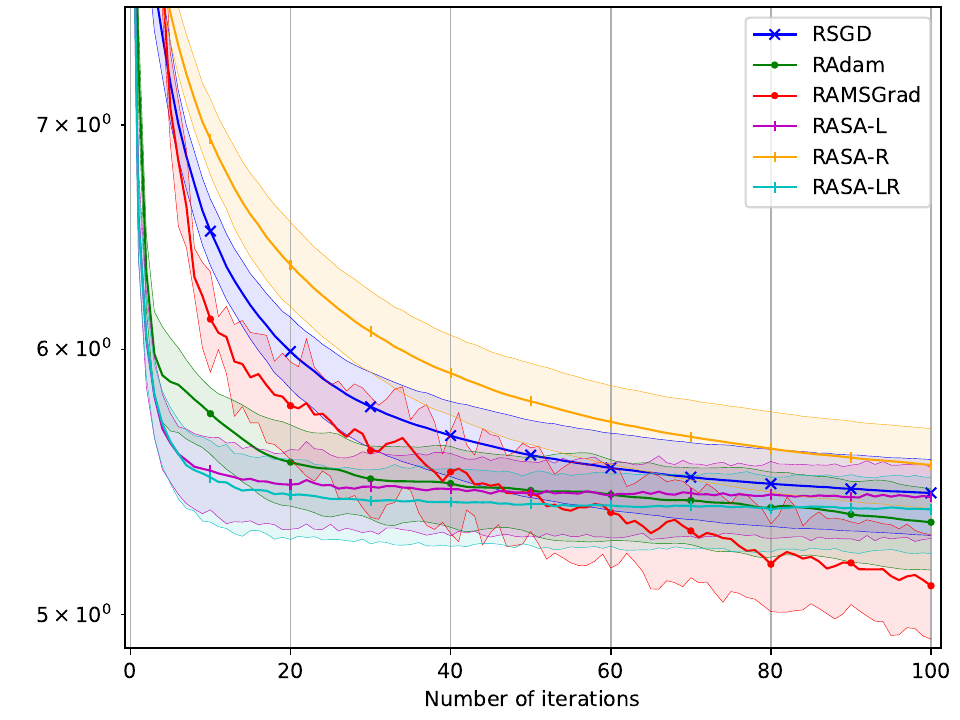}
    \label{fig:ml1m-fv-c-test}
}    
\subfigure[diminishing learning rate]{
    \includegraphics[clip, width=0.4\columnwidth]{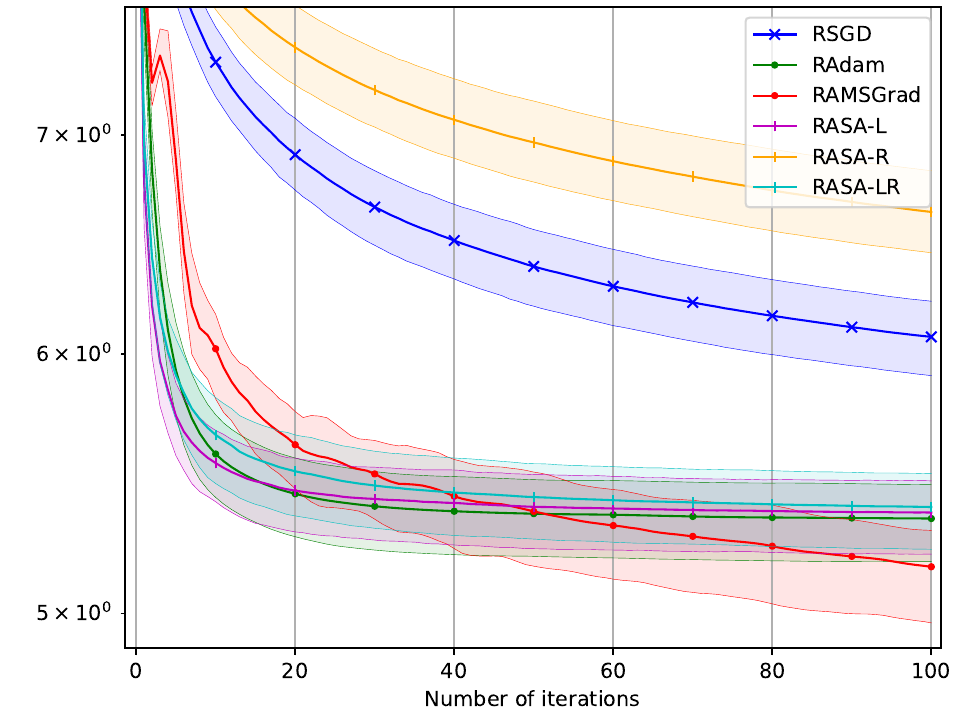}
    \label{fig:ml1m-fv-d-test}
}
\caption{Objective function value defined by \eqref{eq:lmc} versus number of iterations on the test set of the MovieLens-1M datasets.}
\end{figure}

\begin{figure}[htbp]
\centering
\subfigure[constant learning rate]{
    \includegraphics[clip, width=0.4\columnwidth]{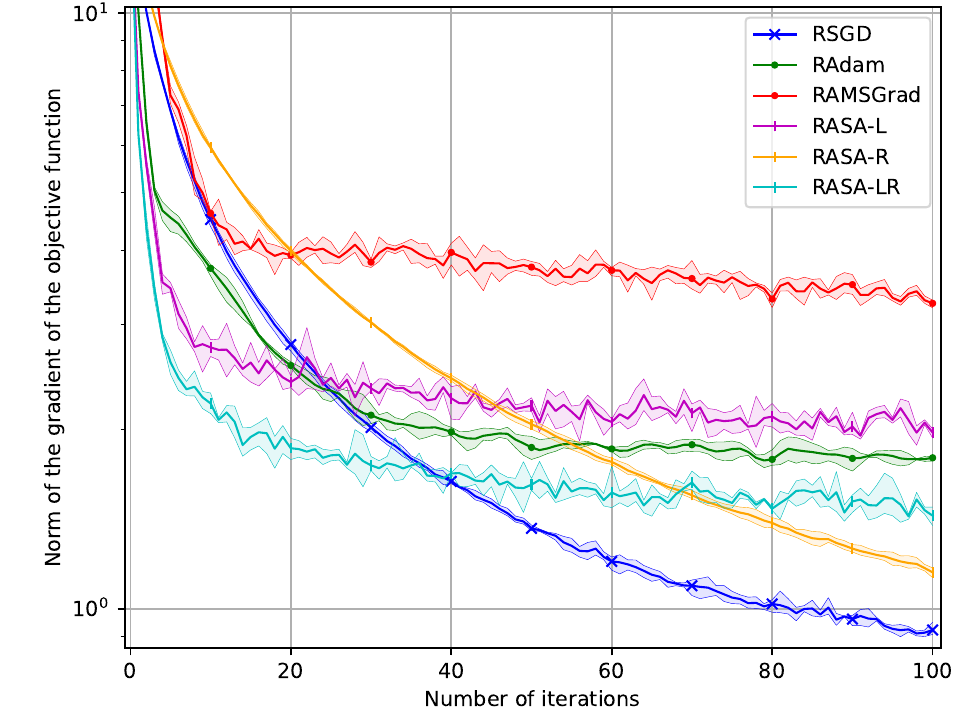}
    \label{fig:ml1m-gn-c}
}    
\subfigure[diminishing learning rate]{
    \includegraphics[clip, width=0.4\columnwidth]{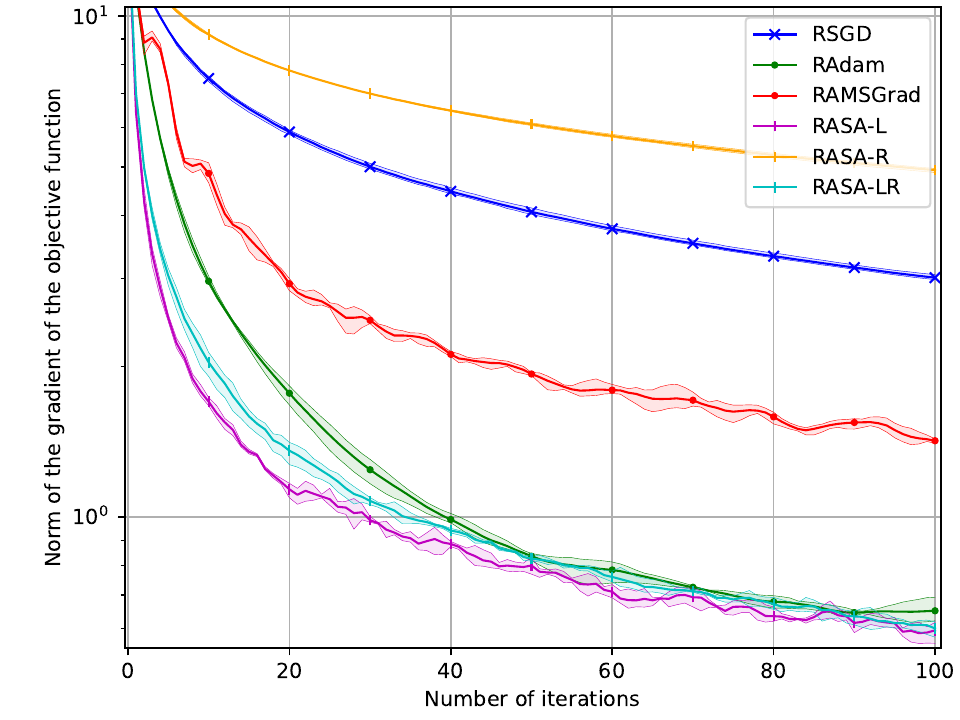}
    \label{fig:ml1m-gn-d}
}
\caption{Norm of the gradient of objective function defined by \eqref{eq:lmc} versus number of iterations on the training set of the MovieLens-1M datasets.}
\end{figure}

\begin{figure}[htbp]
\centering
\subfigure[constant learning rate]{
    \includegraphics[clip, width=0.4\columnwidth]{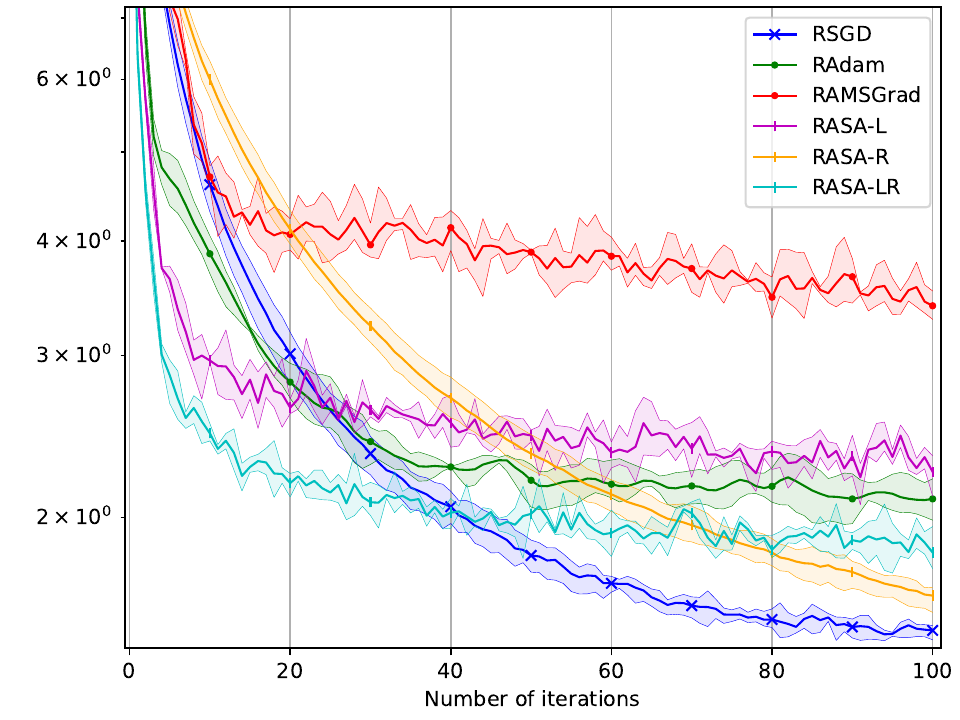}
    \label{fig:ml1m-gn-c-test}
}    
\subfigure[diminishing learning rate]{
    \includegraphics[clip, width=0.4\columnwidth]{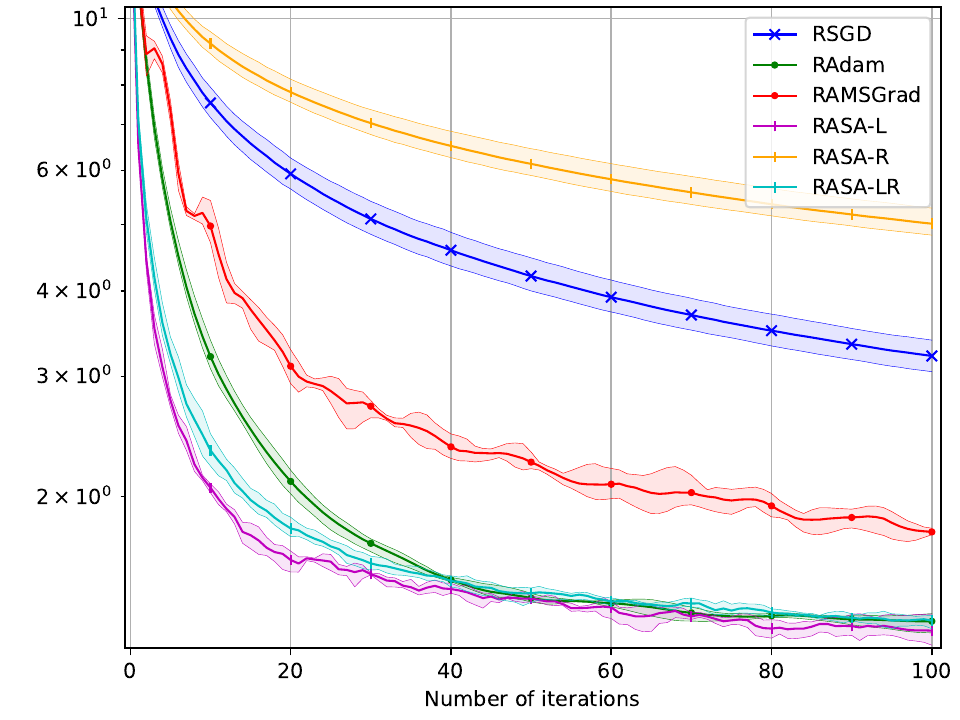}
    \label{fig:ml1m-gn-d-test}
}
\caption{Norm of the gradient of objective function defined by \eqref{eq:lmc} versus number of iterations on the test set of the MovieLens-1M datasets.}
\end{figure}

\begin{figure}[htbp]
\centering
\subfigure[constant learning rate]{
    \includegraphics[clip, width=0.4\columnwidth]{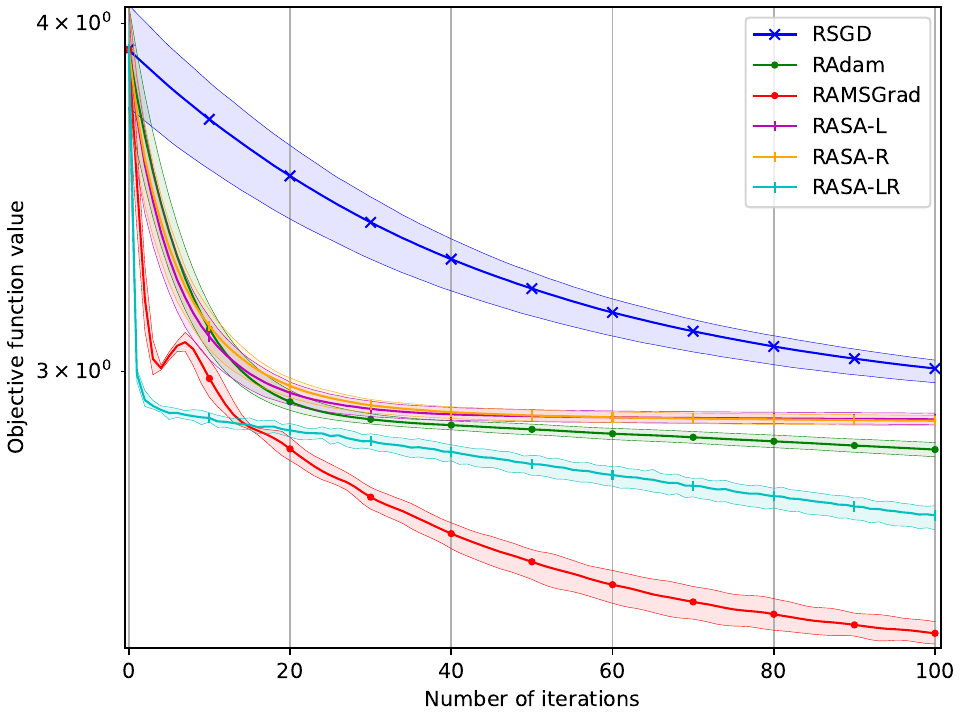}
    \label{fig:jester-fv-c}
}    
\subfigure[diminishing learning rate]{
    \includegraphics[clip, width=0.4\columnwidth]{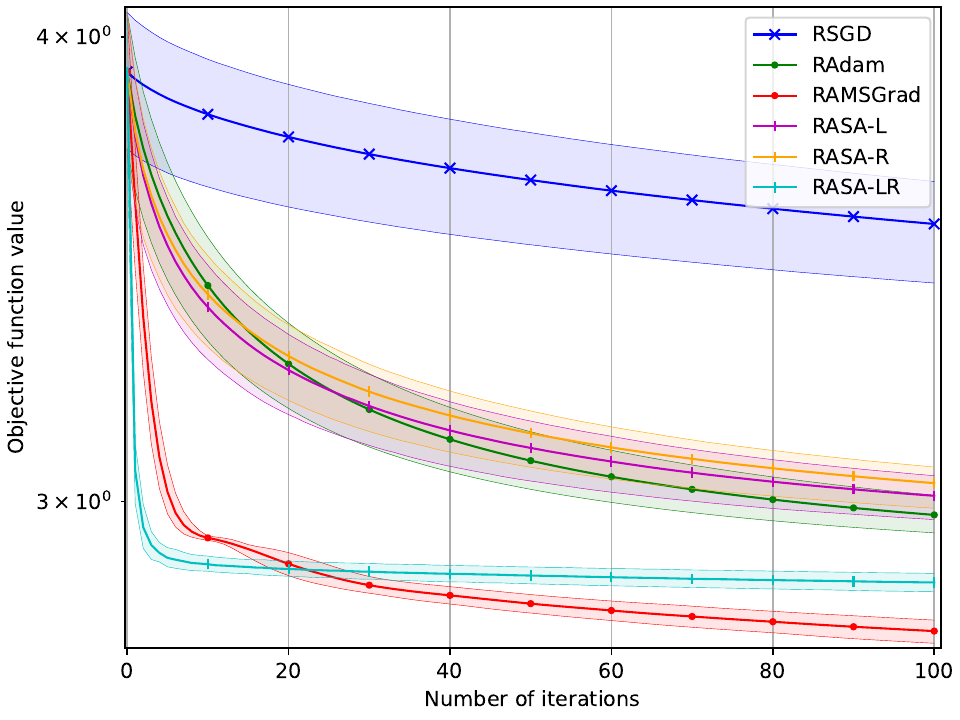}
    \label{fig:jester-fv-d}
}
\caption{Objective function value defined by \eqref{eq:lmc} versus number of iterations on the training set of the Jester datasets.}
\end{figure}

\begin{figure}[htbp]
\centering
\subfigure[constant learning rate]{
    \includegraphics[clip, width=0.4\columnwidth]{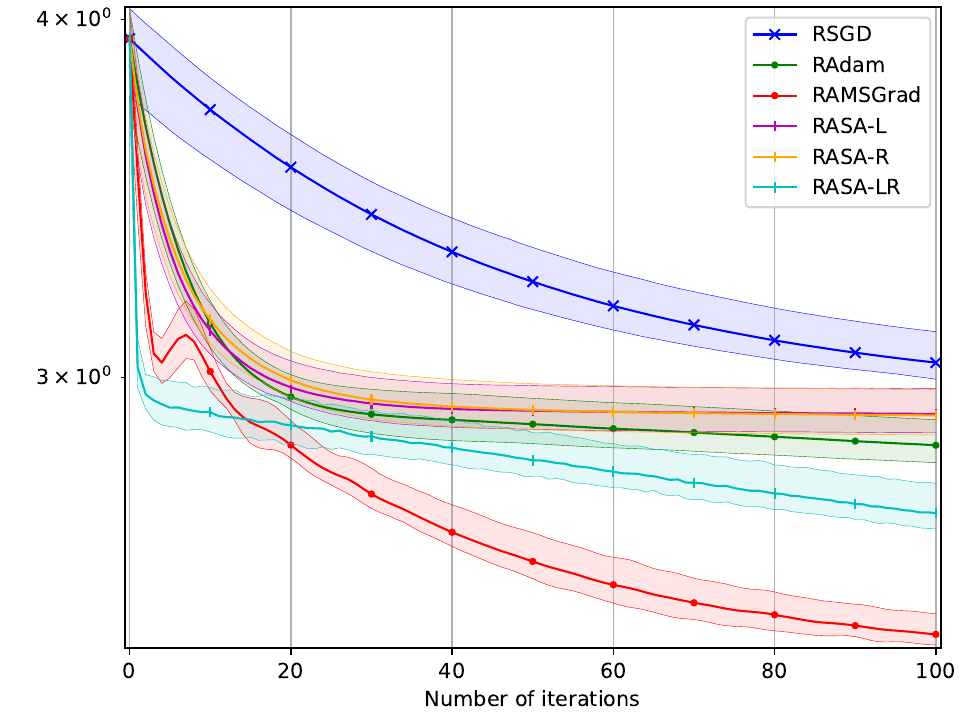}
    \label{fig:jester-fv-c-test}
}    
\subfigure[diminishing learning rate]{
    \includegraphics[clip, width=0.4\columnwidth]{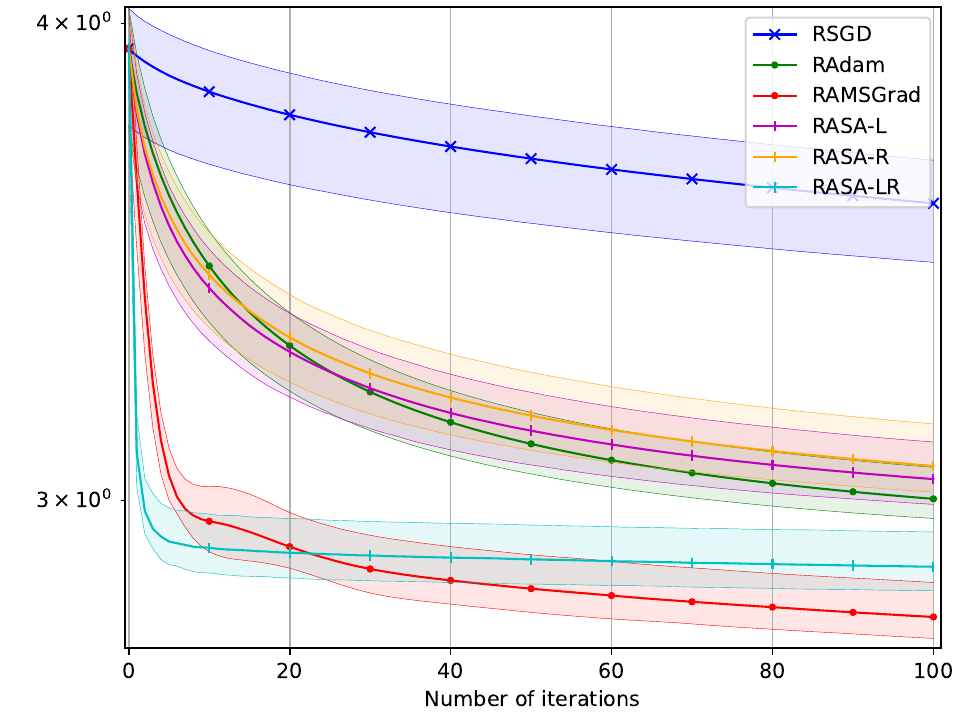}
    \label{fig:jester-fv-d-test}
}
\caption{Objective function value defined by \eqref{eq:lmc} versus number of iterations on the test set of the Jester datasets.}
\end{figure}

\begin{figure}[htbp]
\centering
\subfigure[constant learning rate]{
    \includegraphics[clip, width=0.4\columnwidth]{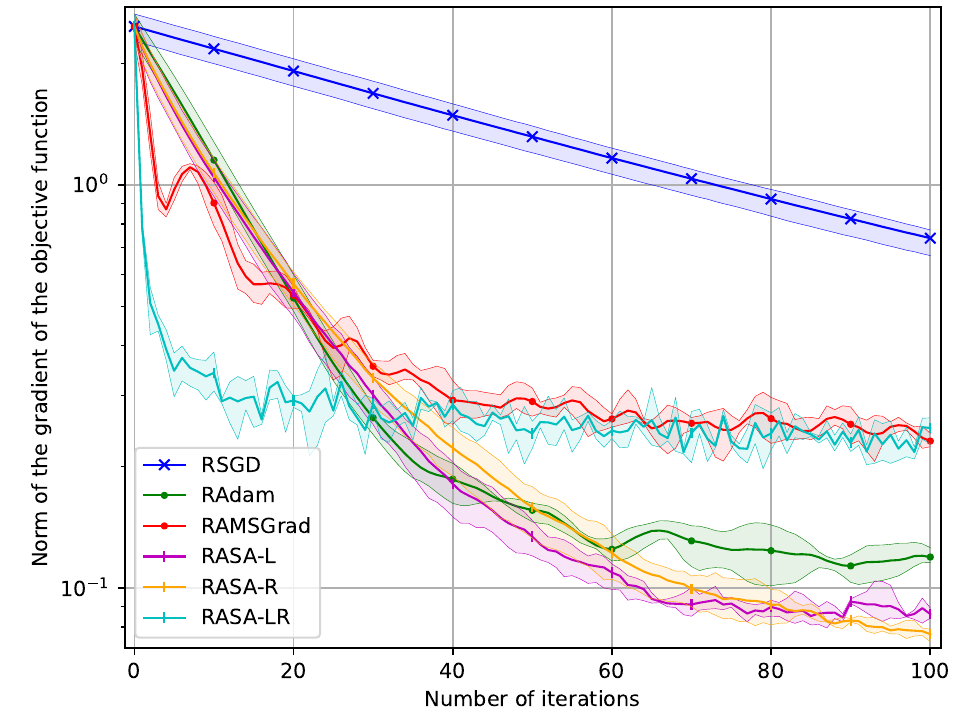}
    \label{fig:jester-gn-c}
}    
\subfigure[diminishing learning rate]{
    \includegraphics[clip, width=0.4\columnwidth]{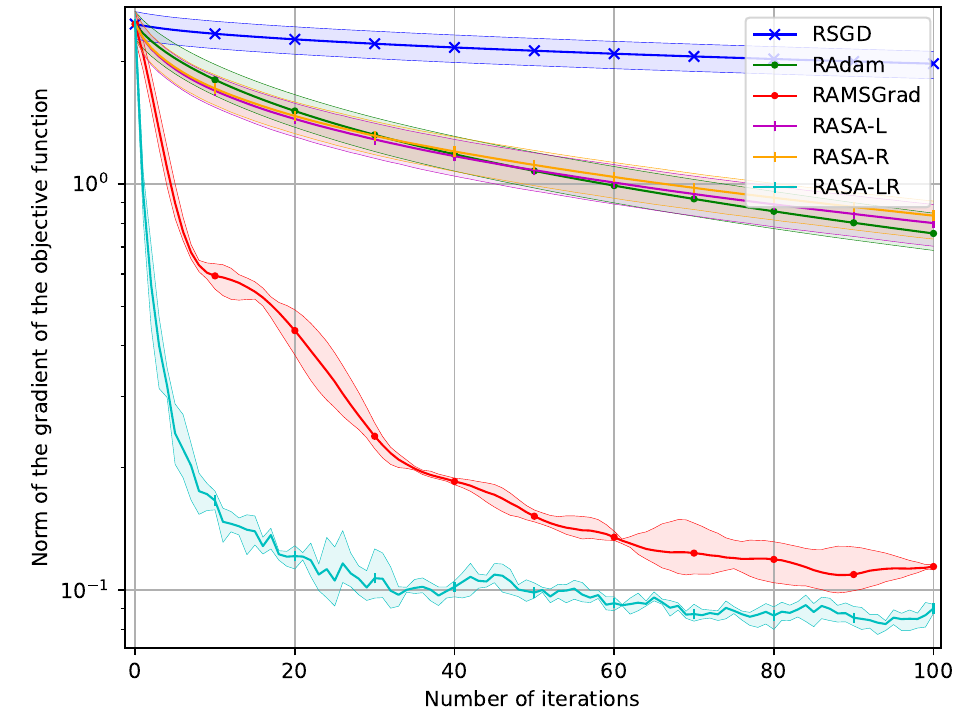}
    \label{fig:jester-gn-d}
}
\caption{Norm of the gradient of objective function defined by \eqref{eq:lmc} versus number of iterations on the training set of the Jester datasets.}
\end{figure}

\begin{figure}[htbp]
\centering
\subfigure[constant learning rate]{
    \includegraphics[clip, width=0.4\columnwidth]{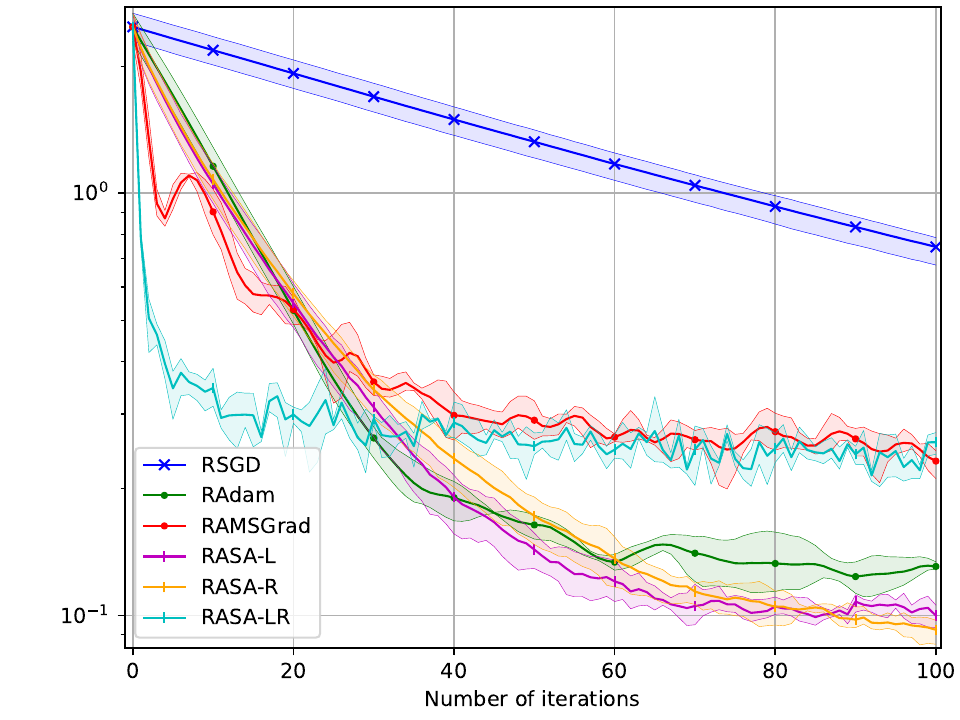}
    \label{fig:jester-gn-c-test}
}    
\subfigure[diminishing learning rate]{
    \includegraphics[clip, width=0.4\columnwidth]{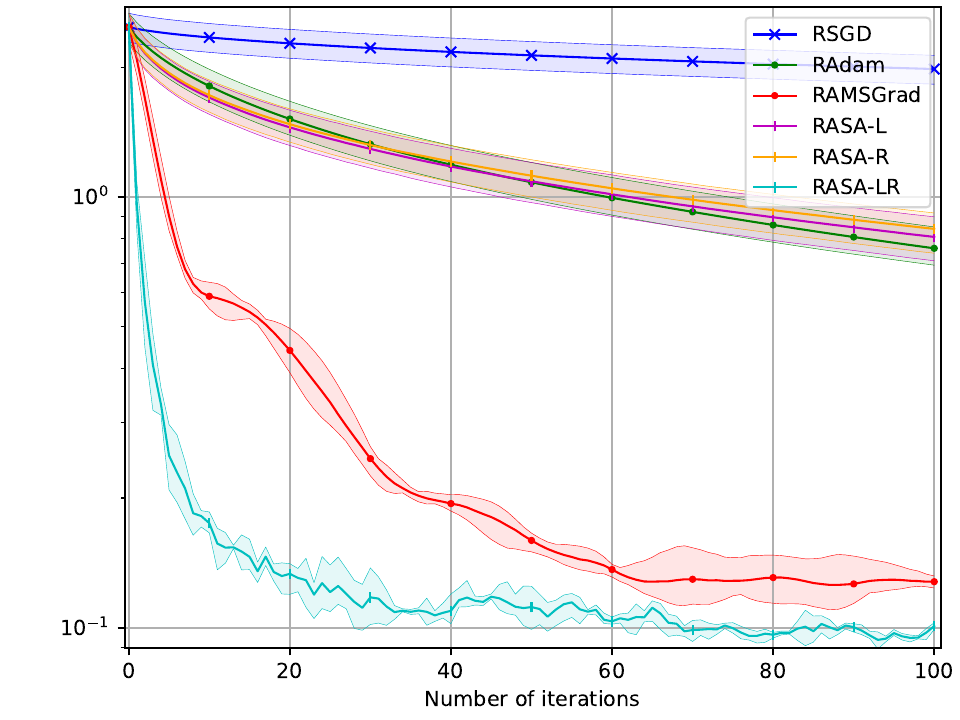}
    \label{fig:jester-gn-d-test}
}
\caption{Norm of the gradient of objective function defined by \eqref{eq:lmc} versus number of iterations on the test set of the Jester datasets.}
\end{figure}

Figures \ref{fig:ml1m-fv-c} and \ref{fig:ml1m-fv-c-test} indicate
that RAMSGrad (Algorithm \ref{alg:AMSGrad}) performed better than RASA-L and RASA-LR in the sense of minimizing the objective function value. Figures \ref{fig:ml1m-fv-d} and \ref{fig:ml1m-fv-d-test} show that RAdam and RAMSGRad (Algorithm \ref{alg:AMSGrad}) performed comparably to RASA-L and RASA-LR.
Figures \ref{fig:ml1m-gn-c} and \ref{fig:ml1m-gn-c-test} indicate that RAdam performed comparably to RASA-R in the sense of minimizing the full gradient norm of the objective function.
Figures \ref{fig:ml1m-gn-d} and \ref{fig:ml1m-gn-d-test} show that RAdam outperformed RASA-R.
Figures \ref{fig:jester-fv-c}, \ref{fig:jester-fv-d}, \ref{fig:jester-fv-c-test} and \ref{fig:jester-fv-d-test} indicate that RAMSGrad (Algorithm \ref{alg:AMSGrad}) performed better than RASA-LR in the sense of minimizing the objective function value.
Figures \ref{fig:jester-gn-c} and \ref{fig:jester-gn-c-test} indicate that RAdam performed comparably to RASA-L and RASA-R in the sense of minimizing the full gradient norm of the objective function.
Figures \ref{fig:jester-gn-d} and \ref{fig:jester-gn-d-test} show that RAMSGrad (Algorithm \ref{alg:AMSGrad}) outperformed RASA-L and RASA-R.

\section{Conclusion}
This paper proposed a general framework of Riemannian adaptive optimization methods, which encapsulates several stochastic optimization algorithms on Riemannian manifolds. The framework incorporates the mini-batch strategy often used in deep learning. We also proposed RAMSGrad (Algorithm \ref{alg:AMSGrad}) that works on embedded submanifolds in Euclidean space within our framework. In addition, we gave convergence analyses that are valid for both a constant and diminishing step size. The analyses also revealed the relationship between the convergence rate and mini-batch size. We numerically compared the RAMSGrad (Algorithm \ref{alg:AMSGrad}) with the existing algorithms by applying them to
principal component analysis and low-rank matrix completion problems, which can be considered to be Riemannian optimization problems. Numerical experiments showed that the proposed method performs well against PCA. RAdam and RAMSGrad performed well for constant and diminishing step sizes especially on the COIL100 dataset.

\subsubsection*{Acknowledgments}
We are sincerely grateful to the Action Editor, Stephen Becker, and the three anonymous reviewers for helping us improve the original manuscript.
This work was supported by a JSPS KAKENHI Grant, number JP23KJ2003.

\bibliography{main}
\bibliographystyle{tmlr}

\appendix
\section{Useful Lemmas}
\begin{lemma}\label{lem:ex_bgrad}
    Suppose that Assumption \ref{asm:main} \ref{asm:unbiased} holds.
    Let $(x_k)_{k=1}^\infty$ be a sequence generated
    by Algorithm \ref{alg:general}. Then,
    \begin{align*}
    \ex_k\left[\grad f_{B_k}(x_k)\right]=\grad f(x_k),
    \end{align*}
    for all $k\geq 1$.
\end{lemma}
\begin{proof}
    From \eqref{eq:bgrad}, Assumption \ref{asm:main} \ref{asm:unbiased}
    and the linearity of $\ex_k[\cdot]$, we have
    \begin{align*}
        \ex_k\left[\grad f_{B_k}(x_k)\right]
        =\frac{1}{b_k}\sum_{i=1}^{b_k}\ex_k
        \left[\grad f_{s_{k,i}}(x_k)\right]
        =\grad f(x_k).
    \end{align*}
    This completes the proof.    
\end{proof}

\begin{lemma}\label{lem:sn_bgrad}
    Suppose that Assumptions \ref{asm:main} \ref{asm:unbiased}
    and \ref{asm:variance} hold.
    Let $(x_k)_{k=1}^\infty$ be a sequence generated
    by Algorithm \ref{alg:general}. Then,
    \begin{align*}
    \ex_k\left[\norm{\grad f_{B_k}(x_k)}^2_2\right]
    \leq\frac{\sigma^2}{b_k}+\norm{\grad f(x_k)}^2_2
    \end{align*}
    for all $k\geq 1$.
\end{lemma}
\begin{proof}
    From $\norm{a+b}^2_2=\norm{a}^2_2+2\ip{a}{b}_2+\norm{b}_2^2$,
    we obtain
    \begin{align}
    \ex_k\left[\norm{\grad f_{B_k}(x_k)}^2_2\right]
    &=\ex_k\left[\norm{\grad f_{B_k}(x_k)
    -\grad f(x_k)}^2_2\right] \nonumber\\
    &\quad+2\ex_k\left[\ip{\grad f_{B_k}(x_k)-\grad f(x_k)}
    {\grad f(x_k)}_2\right] +\ex_k\left[\norm{\grad f(x_k)}^2_2\right]\label{eq:sn_bgrad},
    \end{align}
    for all $k\geq 1$. From \eqref{eq:bgrad} and
    Assumption \ref{asm:main} \ref{asm:unbiased},
    the first term on the right-hand side of \eqref{eq:sn_bgrad} yields
    \begin{align*}
        \ex_k\left[\norm{\grad f_{B_k}(x_k)-\grad f(x_k)}_2^2\right] 
        &=\ex_k\left[\norm{\frac{1}{b_k}\sum_{i=1}^{b_k}
        \grad f_{s_{k,i}}(x_k)-\grad f(x_k)}_2^2\right] \\
        &=\frac{1}{b_k^2}\ex_k\left[\sum_{i=1}^{b_k}\norm{
        \grad f_{s_{k,i}}(x_k)-\grad f(x_k)}_2^2\right]\\
        &\leq\frac{\sigma^2}{b_k},
    \end{align*}
    where the second equality comes from
    Assumption \ref{asm:main} \ref{asm:variance}.
    From Lemma \ref{lem:ex_bgrad},
    the second term on the right-hand side of \eqref{eq:sn_bgrad} yields
    \begin{align*}
        2\ex_k\left[\ip{\grad f_{B_k}(x_k)-\grad f(x_k)}
        {\grad f(x_k)}_2\right]
        &=2\ip{\ex_k[\grad f_{B_k}(x_k)]-\grad f(x_k)}
        {\grad f(x_k)}_2 \\
        &=2\ip{\grad f(x_k)-\grad f(x_k)}
        {\grad f(x_k)}_2\\
        &=0.
    \end{align*}
    Therefore, we obtain
    \begin{align*}
    \ex_k\left[\norm{\grad f_{B_k}(x_k)}^2_2\right]
    \leq\frac{\sigma^2}{b_k}+\norm{\grad f(x_k)}^2_2,
    \end{align*}
    for all $k\geq 1$. This completes the proof.
\end{proof}

\begin{lemma}\label{lem:v_bound}
Suppose that Assumption \ref{asm:main} \ref{asm:bound} holds.
Then, the sequence $(x_k)_{k=1}^{\infty}\subset M$ generated by 
Algorithm \ref{alg:AMSGrad} satisfies
\begin{align*}
    \hat{v}_{k,i}\leq B^2,
\end{align*}
for all $k\geq 1$ and $i=1,\ldots,d$.
\end{lemma}
\begin{proof}
    Note that from Assumption \ref{asm:main} \ref{asm:bound}, we have
    \begin{align*}
        g_{k,i}^2\leq g_{k,1}^2+\cdots+g_{k,d}^2=\norm{g_k}_2^2\leq B^2
    \end{align*}
    for all $k\geq 1$ and $i=1,\ldots,d$. The proof is by induction.
    For $k=1$, from $0\leq\beta_2<1$, we have
    \begin{align*}
        \hat{v}_{1,i}=v_{1,i}:=\beta_2v_{0,i}+(1-\beta_2)g_{1,i}^2
        =(1-\beta_2)g_{1,i}^2\leq g_{1,i}^2\leq B^2.
    \end{align*}
    Suppose that $\hat{v}_{k-1,i}\leq B^2$.
    From $v_{k-1,i}\leq\hat{v}_{k-1,i}\leq B^2$, we have
    \begin{align*}
        v_{k,i}=\beta_2v_{k-1,i}+(1-\beta_2)g_{k,i}^2
        \leq\beta_2B^2+(1-\beta_2)B^2=B^2.
    \end{align*}
    Thus, induction ensures that $v_{k,i}\leq B^2$ for all $k\geq 1$.
\end{proof}

\section{Proof of Lemma \ref{lem:embedded_Lipsitz}}\label{apx:embedded_Lipsitz}
\begin{proof}
We denote $\grad f(x_k)$ by $g(x_k)$.
From Proposition \ref{prp:ret_Lsmooth}, we have
\begin{align*}
    f(x_{k+1})&\leq f(x_k)
    +\ip{g(x_k)}{-\alpha_kP_{x_k}(H_k^{-1}g_{k})}_2
    +\frac{L}{2}\norm{-\alpha_kP_{x_k}(H_k^{-1}g_{k})}^2_2,
\end{align*}
for all $k\geq 1$.
Here, we note that the tangent space $T_{x_k}M$ of an embedded submanifold $M$ in Euclidean space is a subspace of Euclidean space. Therefore, according to the result in \citep[6.57 (a)]{axler2015linear},
the projection $P_{x_k}$ is a linear map.
From the linearity and symmetry of $P_{x_k}$, we obtain
\begin{align*}
    \ip{g(x_k)}{-\alpha_kP_{x_k}(H_k^{-1}g_{k})}_2=
    \ip{P_{x_k}(g(x_k))}{-\alpha_kH_k^{-1}g_{k}}_2
    =\ip{g(x_k)}{-\alpha_kH_k^{-1}g_{k}}_2.
\end{align*}
From the symmetry of $P_{x_k}$ and $P_{x_k}\circ P_{x_k}=P_{x_k}$, we have
\begin{align*}
    \norm{-\alpha_kP_{x_k}(H_k^{-1}g_{k})}^2_2
    &=\alpha_k^2\norm{P_{x_k}(H_k^{-1}g_{k})}^2_2 \\
    &=\alpha_k^2\ip{P_{x_k}(H_k^{-1}g_{k})}{P_{x_k}(H_k^{-1}g_{k})}_2 \\
    &=\alpha_k^2\ip{H_k^{-1}g_{k}}{P_{x_k}(P_{x_k}(H_k^{-1}g_{k}))}_2 \\
    &=\alpha_k^2\ip{H_k^{-1}g_{k}}{P_{x_k}(H_k^{-1}g_{k})}_2 \\
    &\leq\alpha_k^2\norm{H_k^{-1}g_{k}}_2\norm{P_{x_k}(H_k^{-1}g_{k})}_2.
\end{align*}
Here, when $P_{x_k}(H_k^{-1}g_{k})\neq 0\in\real^d$, it follows that
\begin{align*}
    \norm{-\alpha_kP_{x_k}(H_k^{-1}g_{k})}^2_2\leq
    \alpha_k^2\norm{H_k^{-1}g_{k}}_2^2
    \leq\alpha_k^2\nu^2\norm{g_{k}}^2_2,
\end{align*}
where
the first inequality comes from $\norm{P_{x_k}(H_k^{-1}g_{k})}_2\leq \norm{H_k^{-1}g_{k}}_2$ (see \citet[Corollary 3.24 (ii)]{bauschke2011covex} and \citet[6.57 (h)]{axler2015linear} for details) and
the second inequality comes from $O\prec H_k^{-1}\preceq\nu I_d$. On the other hand, this inequality clearly holds if $P_{x_k}(H_k^{-1}g_{k})=0\in\real^d$. Therefore, we obtain
\begin{align*}
    f(x_{k+1})&\leq f(x_k)
    +\ip{g(x_k)}{-\alpha_kH_k^{-1}g_{k}}_2
    +\frac{L\alpha_k^2\nu^2}{2}\norm{g_{k}}^2_2,
\end{align*}
for all $k\geq 1$. This completes the proof.
\end{proof}

\section{Linear algebra lemma}
\begin{lemma}\label{lem:linalg}
    Let $a=(a_1,\ldots,a_n)^\top\in\real^n$,
    $b=(b_1,\ldots,b_n)^\top\in\real^n$
    and $D=\diag(d_1,\ldots,d_n)\in\mathcal{S}^n_{+}\cap\mathcal{D}^n$.
    If $\norm{a}_2\leq A$ and $\norm{b}_2\leq B$, then
    \begin{align*}
        a^\top Db\leq AB\tr(D).
    \end{align*}
\end{lemma}
\begin{proof}
    From $\norm{a}_2\leq A$ and $\norm{b}_2\leq B$, we have
    $\abs{a_i}\leq A$ and $\abs{b_i}\leq B$ for all $i=1,\ldots,n$.
    Therefore, we obtain
    \begin{align*}
        a^\top Db=\sum_{i=1}^na_id_ib_i\leq
        \sum_{i=1}^n\abs{a_i}\cdot\abs{b_i}d_i
        \leq AB\sum_{i=1}^nd_i\leq AB\tr(D).
    \end{align*}
    This completes the proof.
\end{proof}

\section{Details of numerical experiments}
Tables \ref{tab:mnist-a-c}--\ref{tab:jester-a-d} summarize the initial step size, batch size, hyperparameters, and CPU time (seconds) per iteration used in the numerical experiments. Table \ref{tab:mnist-a-c} (resp. Table \ref{tab:mnist-a-d}) summarizes the details of solving the PCA problem for the MNIST dataset by using the algorithms with constant (resp. diminishing) step sizes. Table \ref{tab:coil-a-c} (resp. Table \ref{tab:coil-a-d}) summarizes the details of solving the PCA problem for the COIL100 dataset by using the algorithms with constant (resp. diminishing) step sizes. Table \ref{tab:ml1m-a-c} (resp. Table \ref{tab:ml1m-a-d}) summarizes the details of solving the LRMC problem for the MovieLens-1M dataset by using the algorithms with constant (resp. diminishing) step sizes. Table \ref{tab:jester-a-c} (resp. Table \ref{tab:jester-a-d}) summarizes the details of solving the LRMC problem for the Jester dataset by using the algorithms with constant (resp. diminishing) step sizes.

\begin{table}[htbp]
\caption{Summary of the initial step size, batch size, CPU time (seconds) per iteration, and hyperparameters used in the case of constant step sizes for the PCA problem on the MNIST dataset.}
\label{tab:mnist-a-c}
\centering
\begin{tabular}{c|ccccccc}
\hline
 & $\alpha$ & $\beta_1$ & $\beta_2$  & $\beta$ & $\epsilon$ & $b$ & CPU time per iteration (seconds) \\ \hline\hline
 RSGD & $10^{-2}$ & - & - & - & - & $2^{10}$ & $1.4\times 10^{-2}$ \\ \hline
 RAdam & $10^{-2}$ & 0.9 & 0.999 & - & $10^{-8}$ & $2^{10}$ & $1.481\times 10^{-2}$ \\ \hline
 RAMSGrad & $10^{-3}$ & 0.9 & 0.999 & - & $10^{-8}$ & $2^{10}$ & $1.563\times 10^{-2}$ \\ \hline
 RASA-L & $10^{-3}$ & - & - & 0.99 & $10^{-8}$ & $2^{10}$ & $2.155\times 10^{-2}$ \\ \hline
 RASA-R & $10^{-3}$ & - & - & 0.99 & $10^{-8}$ & $2^{10}$ & $1.836\times 10^{-2}$ \\ \hline
 RASA-LR & $10^{-3}$ & - & - & 0.99 & $10^{-8}$ & $2^{10}$ & $2.374\times 10^{-2}$ \\ \hline
\end{tabular}
\end{table}

\begin{table}[htbp]
\caption{Summary of the initial step size, batch size, CPU time (seconds) per iteration, and hyperparameters used in the case of diminishing step sizes for the PCA problem on the MNIST dataset.}
\label{tab:mnist-a-d}
\centering
\begin{tabular}{c|ccccccc}
\hline
 & $\alpha$ & $\beta_1$ & $\beta_2$  & $\beta$ & $\epsilon$ & $b$ & CPU time per iteration (seconds) \\ \hline\hline
 RSGD & $10^{-1}$ & - & - & - & - & $2^{10}$ & $1.19\times 10^{-2}$ \\ \hline
 RAdam & $10^{-1}$ & 0.9 & 0.999 & - & $10^{-8}$ & $2^{10}$ & $1.275\times 10^{-2}$ \\ \hline
 RAMSGrad & $10^{-2}$ & 0.9 & 0.999 & - & $10^{-8}$ & $2^{10}$ & $1.28\times 10^{-2}$ \\ \hline
 RASA-L & $10^{-2}$ & - & - & 0.99 & $10^{-8}$ & $2^{10}$ & $1.803\times 10^{-2}$ \\ \hline
 RASA-R & $10^{-2}$ & - & - & 0.99 & $10^{-8}$ & $2^{10}$ & $1.629\times 10^{-2}$ \\ \hline
 RASA-LR & $10^{-2}$ & - & - & 0.99 & $10^{-8}$ & $2^{10}$ & $2.222\times 10^{-2}$ \\ \hline
\end{tabular}
\end{table}

\begin{table}[htbp]
\caption{Summary of the initial step size, batch size, CPU time (seconds) per iteration, and hyperparameters used in the case of constant step sizes for the PCA problem on the COIL100 dataset.}
\label{tab:coil-a-c}
\centering
\begin{tabular}{c|ccccccc}
\hline
 & $\alpha$ & $\beta_1$ & $\beta_2$  & $\beta$ & $\epsilon$ & $b$ & CPU time per iteration (seconds) \\ \hline\hline
 RSGD & $10^{-2}$ & - & - & - & - & $2^{10}$ & $3.937\times 10^{-2}$ \\ \hline
 RAdam & $10^{-2}$ & 0.9 & 0.999 & - & $10^{-8}$ & $2^{10}$ & $5.256\times 10^{-2}$ \\ \hline
 RAMSGrad & $10^{-3}$ & 0.9 & 0.999 & - & $10^{-8}$ & $2^{10}$ & $4.531\times 10^{-2}$ \\ \hline
 RASA-L & $10^{-3}$ & - & - & 0.99 & $10^{-8}$ & $2^{10}$ & $5.697\times 10^{-2}$ \\ \hline
 RASA-R & $10^{-3}$ & - & - & 0.99 & $10^{-8}$ & $2^{10}$ & $5.13\times 10^{-2}$ \\ \hline
 RASA-LR & $10^{-3}$ & - & - & 0.99 & $10^{-8}$ & $2^{10}$ & $5.859\times 10^{-2}$ \\ \hline
\end{tabular}
\end{table}

\begin{table}[htbp]
\caption{Summary of the initial step size, batch size, CPU time (seconds) per iteration, and hyperparameters used in the case of diminishing step sizes for the PCA problem on the COIL100 dataset.}
\label{tab:coil-a-d}
\centering
\begin{tabular}{c|ccccccc}
\hline
 & $\alpha$ & $\beta_1$ & $\beta_2$  & $\beta$ & $\epsilon$ & $b$ & CPU time per iteration (seconds) \\ \hline\hline
 RSGD & $10^{-1}$ & - & - & - & - & $2^{10}$ & $3.638\times 10^{-2}$ \\ \hline
 RAdam & $10^{-2}$ & 0.9 & 0.999 & - & $10^{-8}$ & $2^{10}$ & $5.339\times 10^{-2}$ \\ \hline
 RAMSGrad & $10^{-3}$ & 0.9 & 0.999 & - & $10^{-8}$ & $2^{10}$ & $4.033\times 10^{-2}$ \\ \hline
 RASA-L & $10^{-2}$ & - & - & 0.99 & $10^{-8}$ & $2^{10}$ & $6.805\times 10^{-2}$ \\ \hline
 RASA-R & $10^{-2}$ & - & - & 0.99 & $10^{-8}$ & $2^{10}$ & $5.781\times 10^{-2}$ \\ \hline
 RASA-LR & $10^{-2}$ & - & - & 0.99 & $10^{-8}$ & $2^{10}$ & $6.938\times 10^{-2}$ \\ \hline
\end{tabular}
\end{table}

\begin{table}[htbp]
\caption{Summary of the initial step size, batch size, CPU time (seconds) per iteration, and hyperparameters used in the case of constant step sizes for the LRMC problem on the MovieLens-1M dataset.}
\label{tab:ml1m-a-c}
\centering
\begin{tabular}{c|ccccccc}
\hline
 & $\alpha$ & $\beta_1$ & $\beta_2$  & $\beta$ & $\epsilon$ & $b$ & CPU time per iteration (seconds) \\ \hline\hline
 RSGD & $10^{-3}$ & - & - & - & - & $2^8$ & $2.921\times 10^{-1}$ \\ \hline
 RAdam & $10^{-3}$ & 0.9 & 0.999 & - & $10^{-8}$ & $2^8$ & $2.813\times 10^{-1}$ \\ \hline
 RAMSGrad & $10^{-3}$ & 0.9 & 0.999 & - & $10^{-8}$ & $2^8$ & $2.806\times 10^{-1}$ \\ \hline
 RASA-L & $10^{-3}$ & - & - & 0.99 & $10^{-8}$ & $2^8$ & $3.834\times 10^{-1}$ \\ \hline
 RASA-R & $10^{-4}$ & - & - & 0.99 & $10^{-8}$ & $2^8$ & $8.512\times 10^{-2}$ \\ \hline
 RASA-LR & $10^{-4}$ & - & - & 0.99 & $10^{-8}$ & $2^8$ & $1.285\times 10^{-2}$ \\ \hline
\end{tabular}
\end{table}

\begin{table}[htbp]
\caption{Summary of the initial step size, batch size, CPU time (seconds) per iteration, and hyperparameters used in the case of diminishing step sizes for the LRMC problem on the MovieLens-1M dataset.}
\label{tab:ml1m-a-d}
\centering
\begin{tabular}{c|ccccccc}
\hline
 & $\alpha$ & $\beta_1$ & $\beta_2$  & $\beta$ & $\epsilon$ & $b$ & CPU time per iteration (seconds) \\ \hline\hline
 RSGD & $10^{-3}$ & - & - & - & - & $2^8$ & $1.931\times 10^{-1}$ \\ \hline
 RAdam & $10^{-3}$ & 0.9 & 0.999 & - & $10^{-8}$ & $2^8$ & $1.953\times 10^{-1}$ \\ \hline
 RAMSGrad & $10^{-3}$ & 0.9 & 0.999 & - & $10^{-8}$ & $2^8$ & $2.302\times 10^{-1}$ \\ \hline
 RASA-L & $10^{-3}$ & - & - & 0.99 & $10^{-8}$ & $2^8$ & $3.11\times 10^{-1}$ \\ \hline
 RASA-R & $10^{-4}$ & - & - & 0.99 & $10^{-8}$ & $2^8$ & $1.231\times 10^{-2}$ \\ \hline
 RASA-LR & $10^{-4}$ & - & - & 0.99 & $10^{-8}$ & $2^8$ & $1.86\times 10^{-2}$ \\ \hline
\end{tabular}
\end{table}

\begin{table}[htbp]
\caption{Summary of the initial step size, batch size, CPU time (seconds) per iteration, and hyperparameters used in the case of constant step sizes for the LRMC problem on the Jester dataset.}
\label{tab:jester-a-c}
\centering
\begin{tabular}{c|ccccccc}
\hline
 & $\alpha$ & $\beta_1$ & $\beta_2$  & $\beta$ & $\epsilon$ & $b$ & CPU time per iteration (seconds) \\ \hline\hline
 RSGD & $10^{-3}$ & - & - & - & - & $2^8$ & $6.03\times 10^{-2}$ \\ \hline
 RAdam & $10^{-3}$ & 0.9 & 0.999 & - & $10^{-8}$ & $2^8$ & $5.973\times 10^{-2}$ \\ \hline
 RAMSGrad & $10^{-3}$ & 0.9 & 0.999 & - & $10^{-8}$ & $2^8$ & $5.966\times 10^{-2}$ \\ \hline
 RASA-L & $10^{-3}$ & - & - & 0.99 & $10^{-8}$ & $2^8$ & $5.955\times 10^{-2}$ \\ \hline
 RASA-R & $10^{-3}$ & - & - & 0.99 & $10^{-8}$ & $2^8$ & $5.88\times 10^{-2}$ \\ \hline
 RASA-LR & $10^{-3}$ & - & - & 0.99 & $10^{-8}$ & $2^8$ & $5.958\times 10^{-2}$ \\ \hline
\end{tabular}
\end{table}

\begin{table}[htbp]
\caption{Summary of the initial step size, batch size, CPU time (seconds) per iteration, and hyperparameters used in the case of diminishing step sizes for the LRMC problem on the Jester dataset.}
\label{tab:jester-a-d}
\centering
\begin{tabular}{c|ccccccc}
\hline
 & $\alpha$ & $\beta_1$ & $\beta_2$  & $\beta$ & $\epsilon$ & $b$ & CPU time per iteration (seconds) \\ \hline\hline
 RSGD & $10^{-3}$ & - & - & - & - & $2^8$ & $5.815\times 10^{-2}$ \\ \hline
 RAdam & $10^{-3}$ & 0.9 & 0.999 & - & $10^{-8}$ & $2^8$ & $5.836\times 10^{-2}$ \\ \hline
 RAMSGrad & $10^{-3}$ & 0.9 & 0.999 & - & $10^{-8}$ & $2^8$ & $5.885\times 10^{-2}$ \\ \hline
 RASA-L & $10^{-3}$ & - & - & 0.99 & $10^{-8}$ & $2^8$ & $5.746\times 10^{-2}$ \\ \hline
 RASA-R & $10^{-3}$ & - & - & 0.99 & $10^{-8}$ & $2^8$ & $5.853\times 10^{-2}$ \\ \hline
 RASA-LR & $10^{-3}$ & - & - & 0.99 & $10^{-8}$ & $2^8$ & $5.829\times 10^{-2}$ \\ \hline
\end{tabular}
\end{table}

\section{Extended results on batch size comparisons}\label{apx:batch-comparisons}
This appendix presents a detailed comparison of using different batch sizes across various datasets and experiments. Tables \ref{tab:coil100-b-c}--\ref{tab:jester-b-d} show the performance of the algorithms under varying batch sizes, highlighting the impact on convergence speed and computational cost. The results further support our analyses (Theorems \ref{thm:constant} and \ref{thm:diminishing}) that larger batch sizes lead to faster convergence.
\begin{table}[htbp]
\caption{Iterations and CPU time (seconds) required by each algorithm with a constant step size to reduce the gradient norm below 4 for solving the PCA problem on the COIL100 dataset, across different batch sizes. A ``-'' indicates cases where the algorithm did not reach the threshold within the maximum allowed 1,000 iterations.}
\label{tab:coil100-b-c}
\centering
\begin{tabular}{c|c|cc}
\hline
& batch size $b$ & number of iterations & CPU time (seconds) \\ \hline\hline
\multirow{3}{*}{RSGD} & 256 & - & - \\
& 512 & - & - \\
& 1024 & - & - \\ \hline
\multirow{3}{*}{RAdam} & 256 & - & - \\
& 512 & 19 & 0.178  \\
& 1024 & 14 & 0.159 \\ \hline
\multirow{3}{*}{RAMSGrad} & 256 & 132 & 1.105 \\
& 512 & 23 & 0.220 \\
& 1024 & 14 & 0.158 \\
\hline
\multirow{3}{*}{RASA-L} & 256 & 297 & 3.558 \\
& 512 & 122 & 1.636 \\
& 1024 & 75 & 1.169 \\
\hline
\multirow{3}{*}{RASA-R} & 256 & 53 & 0.490 \\
& 512 & 34 & 0.358 \\
& 1024 & 29 & 0.372 \\
\hline
\multirow{3}{*}{RASA-LR} & 256 & 797 & 9.741 \\
& 512 & 225 & 2.968 \\
& 1024 & 60 & 0.936 \\
\hline
\end{tabular}
\end{table}

\begin{table}[htbp]
\caption{Iterations and CPU time (seconds) required by each algorithm with a diminishing step size to reduce the gradient norm below 4 for solving the PCA problem on the COIL100 dataset across different batch sizes.}
\label{tab:coil100-b-d}
\centering
\begin{tabular}{c|c|cc}
\hline
& batch size $b$ & number of iterations & CPU time (seconds) \\ \hline\hline
\multirow{3}{*}{RSGD} & 256 & 237 & 1.695 \\
& 512 & 171 & 1.526 \\
& 1024 & 132 & 1.562 \\
\hline
\multirow{3}{*}{RAdam} & 256 & 21 & 0.160 \\
& 512 & 17 & 0.153 \\
& 1024 & 16 & 0.179 \\
\hline
\multirow{3}{*}{RAMSGrad} & 256 & 32 & 0.245 \\
& 512 & 23 & 0.207 \\
& 1024 & 20 & 0.247 \\
\hline
\multirow{3}{*}{RASA-L} & 256 & 94 & 1.082 \\
& 512 & 52 & 0.670 \\
& 1024 & 36 & 0.542 \\
\hline
\multirow{3}{*}{RASA-R} & 256 & 19 & 0.168 \\
& 512 & 12 & 0.113 \\
& 1024 & 9 & 0.123 \\
\hline
\multirow{3}{*}{RASA-LR} & 256 & 36 & 0.420 \\
& 512 & 31 & 0.400 \\
& 1024 & 27 & 0.411 \\
\hline
\end{tabular}
\end{table}

\begin{table}[htbp]
\caption{Iterations and CPU time (seconds) required by each algorithm with a constant step size to reduce the gradient norm below 3 for solving the LRMC problem on the MovieLens-1M dataset across different batch sizes. ``-'' indicates cases where the algorithm did not reach the threshold within the maximum allowed (300) iterations.}
\label{tab:ml1m-b-c}
\centering
\begin{tabular}{c|c|cc}
\hline
& batch size $b$ & number of iterations & CPU time (seconds) \\ \hline\hline
\multirow{3}{*}{RSGD} & 64 & 23 & 0.143 \\
& 128 & 20 & 0.249 \\
& 256 & 21 & 0.481 \\ \hline
\multirow{3}{*}{RAdam} & 64 & 45 & 0.292 \\
& 128 & 26 & 0.312 \\
& 256 & 16 & 0.365 \\ \hline
\multirow{3}{*}{RAMSGrad} & 64 & - & - \\
& 128 & - & - \\
& 256 & 155 & 3.717 \\ \hline
\multirow{3}{*}{RASA-L} & 64 & 110 & 4.490 \\
& 128 & 19 & 0.950 \\
& 256 & 8 & 0.447 \\
\hline
\multirow{3}{*}{RASA-R} & 64 & 37 & 0.637 \\
& 128 & 34 & 0.745 \\
& 256 & 34 & 1.106 \\
\hline
\multirow{3}{*}{RASA-LR} & 64 & 19 & 0.723 \\
& 128 & 9 & 0.424 \\
& 256 & 5 & 0.232 \\
\hline
\end{tabular}
\end{table}

\begin{table}[htbp]
\caption{Iterations and CPU time (seconds) required by each algorithm with a diminishing step size to reduce the gradient norm below 3 for solving the LRMC problem on the MovieLens-1M dataset across different batch sizes. ``-'' indicates cases where the algorithm did not reach the threshold within the maximum allowed (300) iterations.}
\label{tab:ml1m-b-d}
\centering
\begin{tabular}{c|c|cc}
\hline
& batch size $b$ & number of iterations & CPU time (seconds) \\ \hline\hline
\multirow{3}{*}{RSGD} & 64 & 108 & 0.691 \\
& 128 & 97 & 1.172 \\
& 256 & 110 & 2.595 \\
\hline
\multirow{3}{*}{RAdam} & 64 & 15 & 0.093 \\
& 128 & 13 & 0.149 \\
& 256 & 11 & 0.238 \\
\hline
\multirow{3}{*}{RAMSGrad} & 64 & 50 & 0.322 \\
& 128 & 32 & 0.385 \\
& 256 & 22 & 0.503 \\
\hline
\multirow{3}{*}{RASA-L} & 64 & 8 & 0.299 \\
& 128 & 7 & 0.305 \\
& 256 & 5 & 0.261 \\
\hline
\multirow{3}{*}{RASA-R} & 64 & - & - \\
& 128 & - & - \\
& 256 & - & - \\
\hline
\multirow{3}{*}{RASA-LR} & 64 & 22 & 0.874 \\
& 128 & 16 & 0.705 \\
& 256 & 13 & 0.688 \\
\hline
\end{tabular}
\end{table}

\begin{table}[htbp]
\caption{Iterations and CPU time (seconds) required by each algorithm with a constant step size to reduce the gradient norm below $1/4$ for solving the LRMC problem on the Jester dataset across different batch sizes. ``-'' indicates cases where the algorithm did not reach the threshold within the maximum allowed (300) iterations.}
\label{tab:jester-b-c}
\centering
\begin{tabular}{c|c|cc}
\hline
& batch size $b$ & number of iterations & CPU time (seconds) \\ \hline\hline
\multirow{3}{*}{RSGD} & 64 & 204 & 0.391 \\
& 128 & 210 & 0.785 \\
& 256 & 204 & 1.500 \\
\hline
\multirow{3}{*}{RAdam} & 64 & 37 & 0.070 \\
& 128 & 41 & 0.150 \\
& 256 & 34 & 0.245 \\
\hline
\multirow{3}{*}{RAMSGrad} & 64 & - & - \\
& 128 & 66 & 0.244 \\
& 256 & 54 & 0.392 \\
\hline
\multirow{3}{*}{RASA-L} & 64 & 32 & 0.062 \\
& 128 & 36 & 0.134 \\
& 256 & 32 & 0.231 \\
\hline
\multirow{3}{*}{RASA-R} & 64 & 38 & 0.072 \\
& 128 & 37 & 0.137 \\
& 256 & 34 & 0.244 \\
\hline
\multirow{3}{*}{RASA-LR} & 64 & 174 & 0.342 \\
& 128 & 48 & 0.184 \\
& 256 & 43 & 0.312 \\
\hline
\end{tabular}
\end{table}

\begin{table}[htbp]
\caption{Iterations and CPU time (seconds) required by each algorithm with a diminishing step size to reduce the gradient norm below $1/2$ for solving the LRMC problem on the Jester dataset across different batch sizes. ``-'' indicates cases where the algorithm did not reach the threshold within the maximum allowed (300) iterations.}
\label{tab:jester-b-d}
\centering
\begin{tabular}{c|c|cc}
\hline
& batch size $b$ & number of iterations & CPU time (seconds) \\ \hline\hline
\multirow{3}{*}{RSGD} & 64 & - & - \\
& 128 & - & - \\
& 256 & - & - \\
\hline
\multirow{3}{*}{RAdam} & 64 & 188 & 0.376 \\
& 128 & 225 & 0.840 \\
& 256 & 167 & 1.232 \\
\hline
\multirow{3}{*}{RAMSGrad} & 64 & 21 & 0.039 \\
& 128 & 18 & 0.064 \\
& 256 & 19 & 0.136 \\
\hline
\multirow{3}{*}{RASA-L} & 64 & 184 & 0.372 \\
& 128 & 249 & 0.943 \\
& 256 & 178 & 1.325 \\
\hline
\multirow{3}{*}{RASA-R} & 64 & 219 & 0.427 \\
& 128 & - & - \\
& 256 & 195 & 1.441 \\
\hline
\multirow{3}{*}{RASA-LR} & 64 & 5 & 0.008 \\
& 128 & 4 & 0.012 \\
& 256 & 3 & 0.015 \\
\hline
\end{tabular}
\end{table}

\section{Conditions under which Assumption \ref{asm:main} \ref{asm:variance} holds}\label{apx:asm-variance}
We introduce the assumptions (Assumptions \ref{asm:sub} \ref{asm:Lipschitz-sub} and \ref{asm:below-sub}) under which Assumption \ref{asm:main} \ref{asm:variance} from Section \ref{sec:asm} holds,
and demonstrate that Assumption \ref{asm:main} \ref{asm:variance} is valid under these assumptions (Lemma \ref{lem:sub-variance}).
Note that Assumptions \ref{asm:sub} \ref{asm:Lipschitz-sub} and \ref{asm:below-sub} are simply the counterparts of Assumptions \ref{asm:main} \ref{asm:Lipschitz} and \ref{asm:below} for $f$,
applied to each $f_i$ $(i=1,\ldots,N)$.

\begin{assumption}\label{asm:sub}
For all $i=1,\ldots,N$,
\begin{enumerate}[label=(B\arabic*)]
\item\label{asm:Lipschitz-sub} There exists a constant $L_i>0$ such that
\begin{align*}
    \abs{\mathrm{D}(f_i\circ R_x)(\eta)[\eta]-\mathrm{D}f_i(x)
    [\eta]}\leq L_i\norm{\eta}_2^2,
\end{align*}
for all $x\in M$, $\eta\in T_xM$.
\item\label{asm:below-sub} $f_i$ is bounded below by $f_\star\in\real$.
\end{enumerate}
\end{assumption}

\begin{lemma}\label{lem:sub-variance}
    Suppose that Assumptions \ref{asm:sub} \ref{asm:Lipschitz-sub} and \ref{asm:below-sub} hold.
    We assume that a random variable $s_{k,i}$
    takes a value from 1 to $N$ following a uniform distribution.
    Then, the sequence $(x_k)_{k=1}^{\infty}\subset M$ generated by 
    Algorithm \ref{alg:general} satisfies
    \begin{align*}
        \ex_k\left[\norm{\grad f_{s_{k,i}}(x_k)-\grad f(x_k)}_2^2\right]
        \leq\frac{2}{N}\sum_{i=1}^NL_iM_i,
    \end{align*}
    where
    \begin{align*}
        M_i:=\sup\lbrace f_i(x_k)-f_\star\mid k\geq 1\rbrace
        \in[0,\infty).
    \end{align*}
\end{lemma}
\begin{proof}
    From Assumption \ref{asm:main} \ref{asm:below} and Proposition \ref{prp:ret_Lsmooth}, we have
    \begin{align*}
        f_\ast&\leq f_i\left(R_{x_k}\left(-\frac{1}{L_i}\grad f_i(x_k)\right)\right) \\
        &\leq f_i(x_k)+\ip{\grad f_i(x_k)}{-\frac{1}{L_i}\grad f_i(x_k)}_{x_k}
        +\frac{L_i}{2}\norm{-\frac{1}{L_i}\grad f_i(x_k)}_{x_k}^2 \\
        &=f_i(x_k)-\frac{1}{L_i}\norm{\grad f_i(x_k)}_{x_k}^2
        +\frac{1}{2L_i}\norm{\grad f_i(x_k)}_{x_k}^2 \\
        &=f_i(x_k)-\frac{1}{2L_i}\norm{\grad f_i(x_k)}_{x_k}^2.
    \end{align*}
    Therefore, we obtain
    \begin{align}\label{eq:sub-variance}
        \norm{\grad f_i(x_k)}_{x_k}^2\leq 2L_i(f_i(x_k)-f_\star)\leq 2L_iM_i,
    \end{align}
    where
    \begin{align*}
        M_i:=\sup\lbrace f_i(x_k)-f_\star\mid k\geq 1\rbrace
        \in[0,\infty).
    \end{align*}
    From $\norm{a-b}^2_2=\norm{a}^2_2-2\ip{a}{b}_2+\norm{b}_2^2$,
    we obtain
    \begin{align*}
    &\ex_k\left[\norm{\grad f_{s_{k,i}}(x_k)-\grad f(x_k)}^2_2\right] \\
    &=\ex_k\left[\norm{\grad f_{s_{k,i}}(x_k)}^2_2\right]
    -2\ex_k\left[\ip{\grad f_{s_{k,i}}(x_k)}{\grad f(x_k)}_2\right]
    +\ex_k\left[\norm{\grad f(x_k)}^2_2\right] \\
    &=\frac{1}{N}\sum_{i=1}^N\norm{\grad f_i(x_k)}^2_2
    -2\ip{\ex_k[\grad f_{s_{k,i}}(x_k)]}{\grad f(x_k)}_2
    +\norm{\grad f(x_k)}^2_2 \\
    &=\frac{1}{N}\sum_{i=1}^N\norm{\grad f_i(x_k)}^2_2
    -\norm{\grad f(x_k)}^2_2 \\
    &\leq\frac{1}{N}\sum_{i=1}^N\norm{\grad f_i(x_k)}^2_2.
    \end{align*}
    Here, by using \eqref{eq:sub-variance}, it follows that
    \begin{align*}
        \ex_k\left[\norm{\grad f_{s_{k,i}}(x_k)-\grad f(x_k)}^2_2\right]
        \leq \frac{2}{N}\sum_{i=1}^NL_iM_i.
    \end{align*}
\end{proof}
\section{Experimental results on the PCA and LRMC with increasing batch sizes}\label{apx:inc}
\begin{figure}[htbp]
\centering
\subfigure[constant learning rate]{
    \includegraphics[clip, width=0.4\columnwidth]{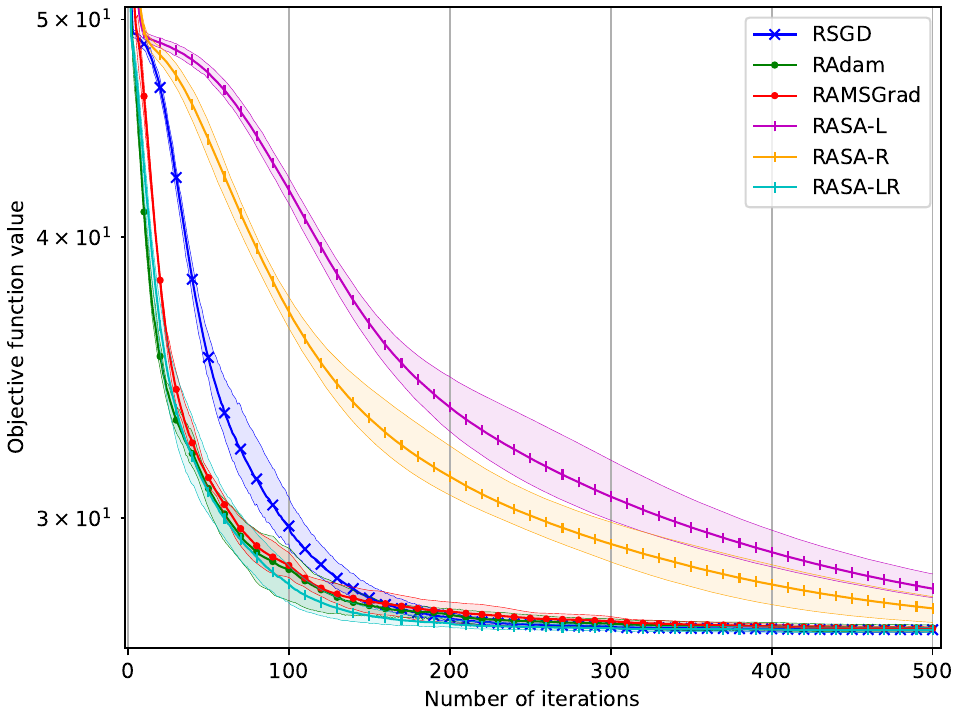}
}    
\subfigure[diminishing learning rate]{
    \includegraphics[clip, width=0.4\columnwidth]{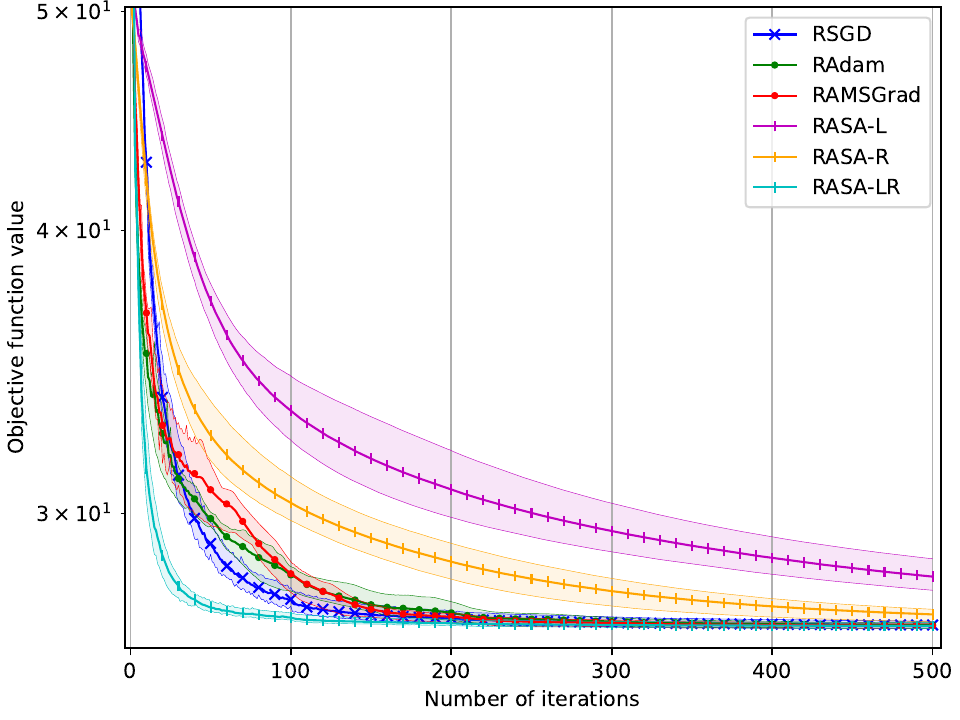}
}
\caption{Objective function value defined by \eqref{eq:pca} versus number of iterations on the training set of the MNIST datasets.}
\end{figure}

\begin{figure}[htbp]
\centering
\subfigure[constant learning rate]{
    \includegraphics[clip, width=0.4\columnwidth]{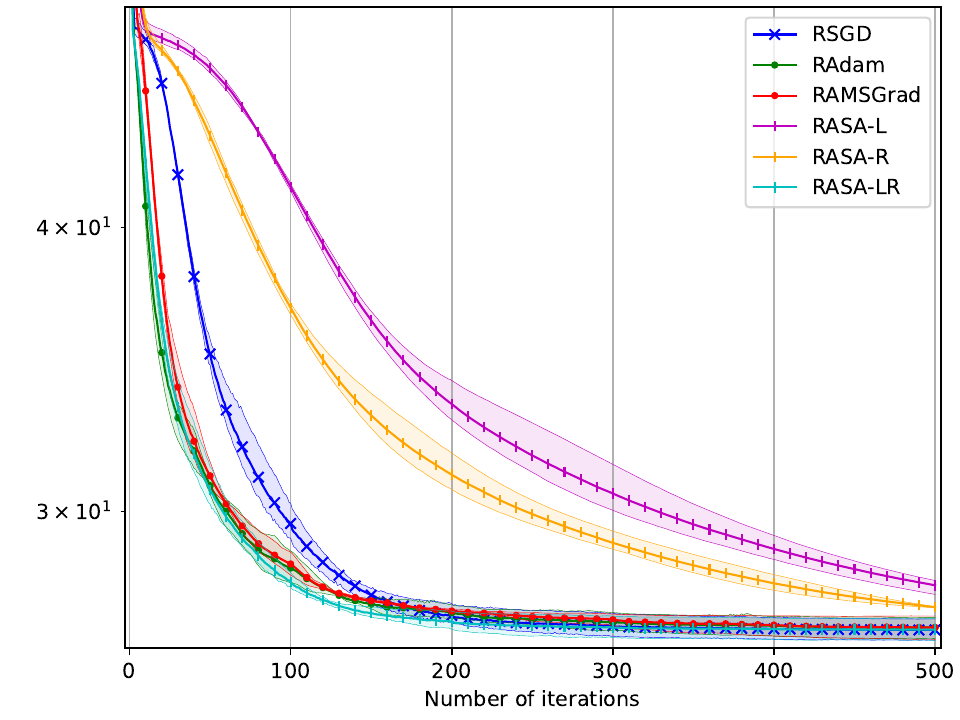}
}    
\subfigure[diminishing learning rate]{
    \includegraphics[clip, width=0.4\columnwidth]{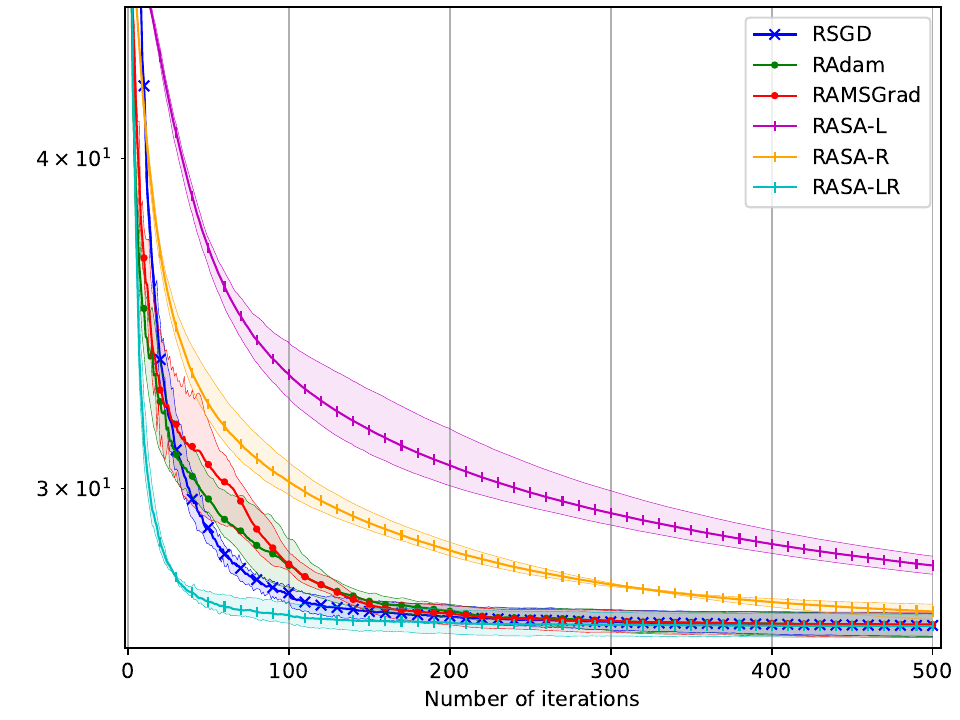}
}
\caption{Objective function value defined by \eqref{eq:pca} versus number of iterations on the test set of the MNIST datasets.}
\end{figure}

\begin{figure}[htbp]
\centering
\subfigure[constant learning rate]{
    \includegraphics[clip, width=0.4\columnwidth]{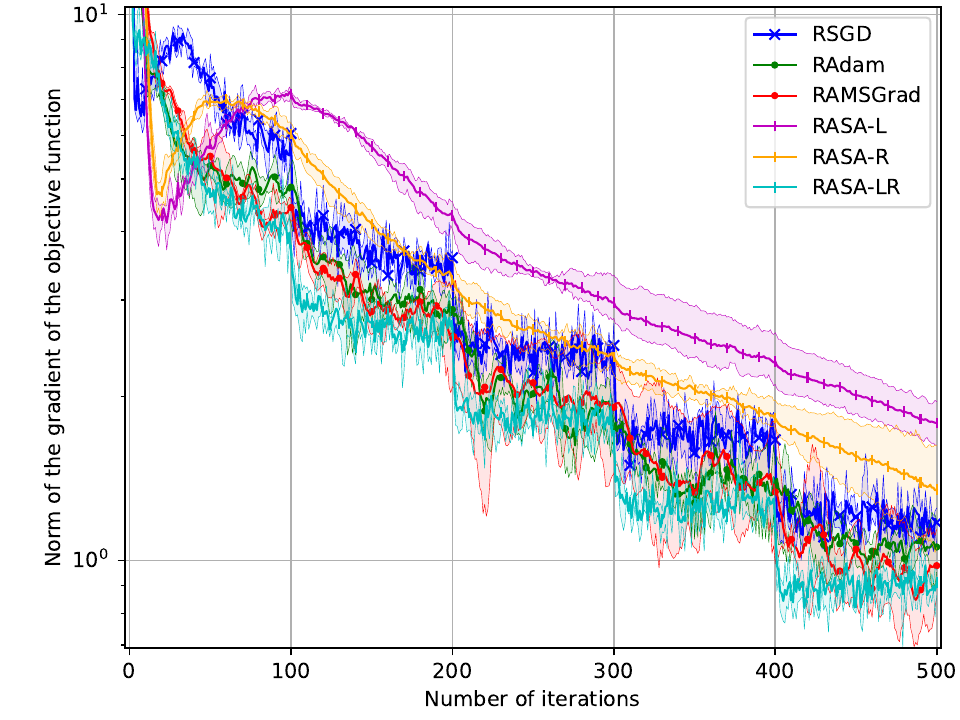}
}    
\subfigure[diminishing learning rate]{
    \includegraphics[clip, width=0.4\columnwidth]{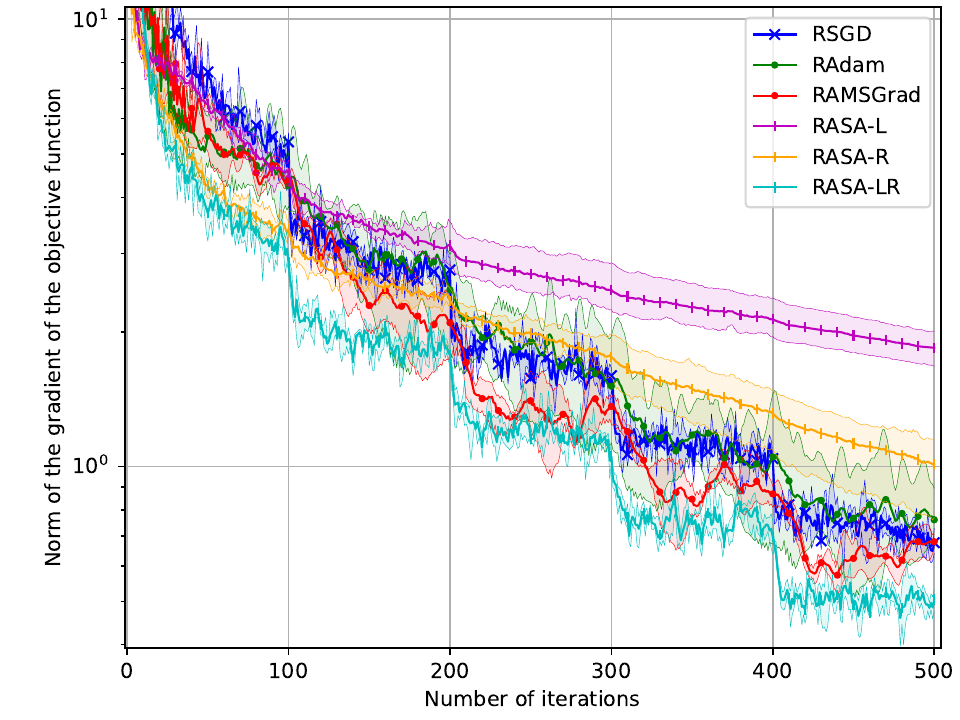}
}
\caption{Norm of the gradient of objective function defined by \eqref{eq:pca} versus number of iterations on the training set of the MNIST datasets.}
\end{figure}

\begin{figure}[htbp]
\centering
\subfigure[constant learning rate]{
    \includegraphics[clip, width=0.4\columnwidth]{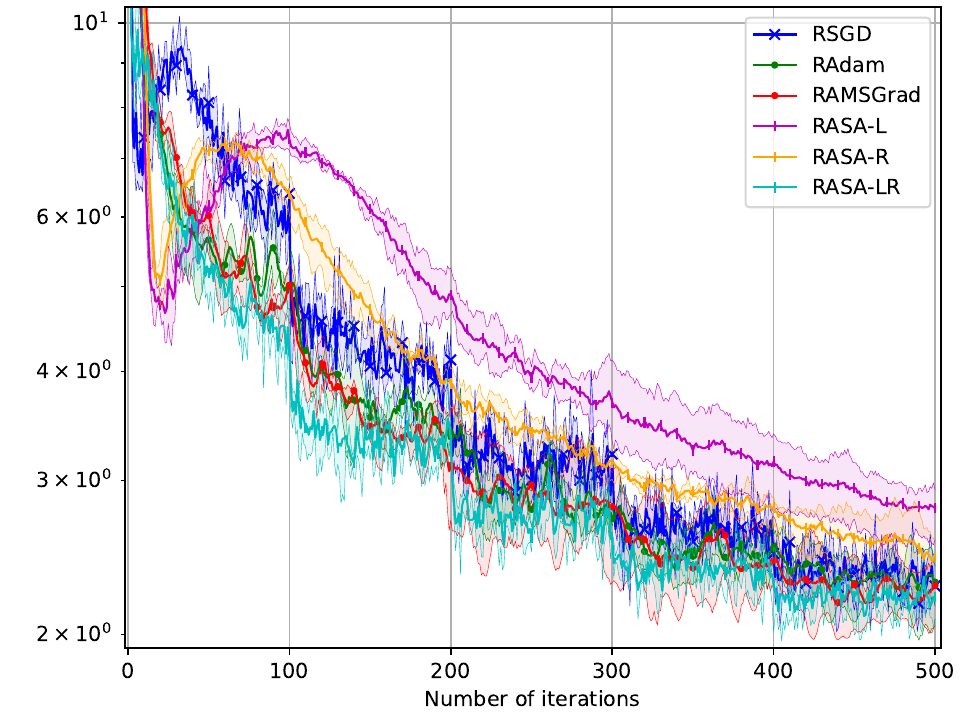}
}    
\subfigure[diminishing learning rate]{
    \includegraphics[clip, width=0.4\columnwidth]{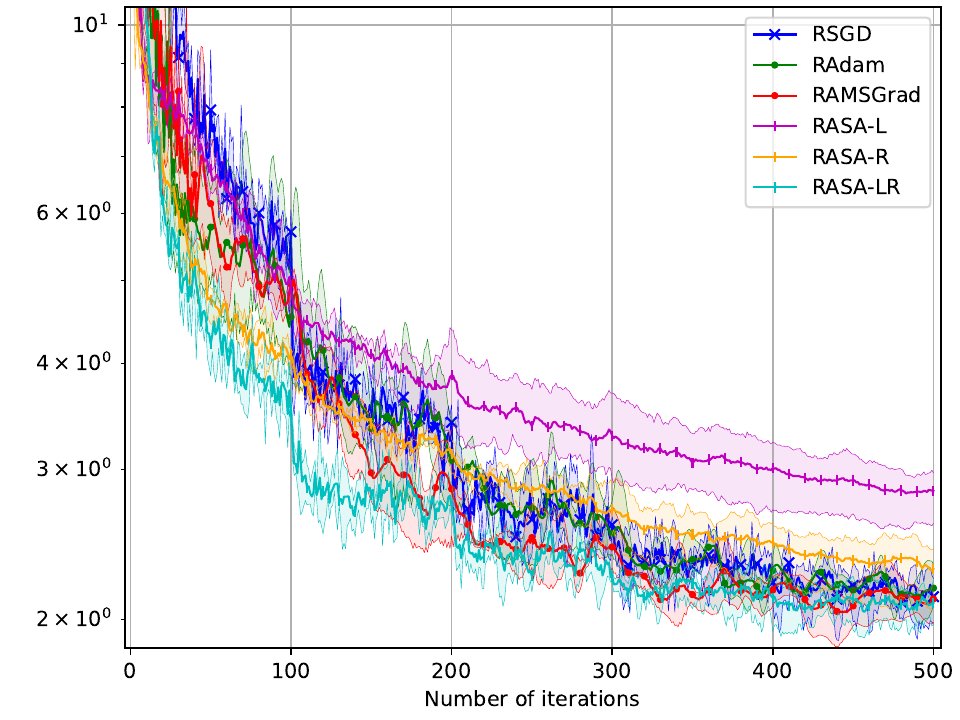}
}
\caption{Norm of the gradient of objective function defined by \eqref{eq:pca} versus number of iterations on the test set of the MNIST datasets.}
\end{figure}

\begin{figure}[htbp]
\centering
\subfigure[constant learning rate]{
    \includegraphics[clip, width=0.4\columnwidth]{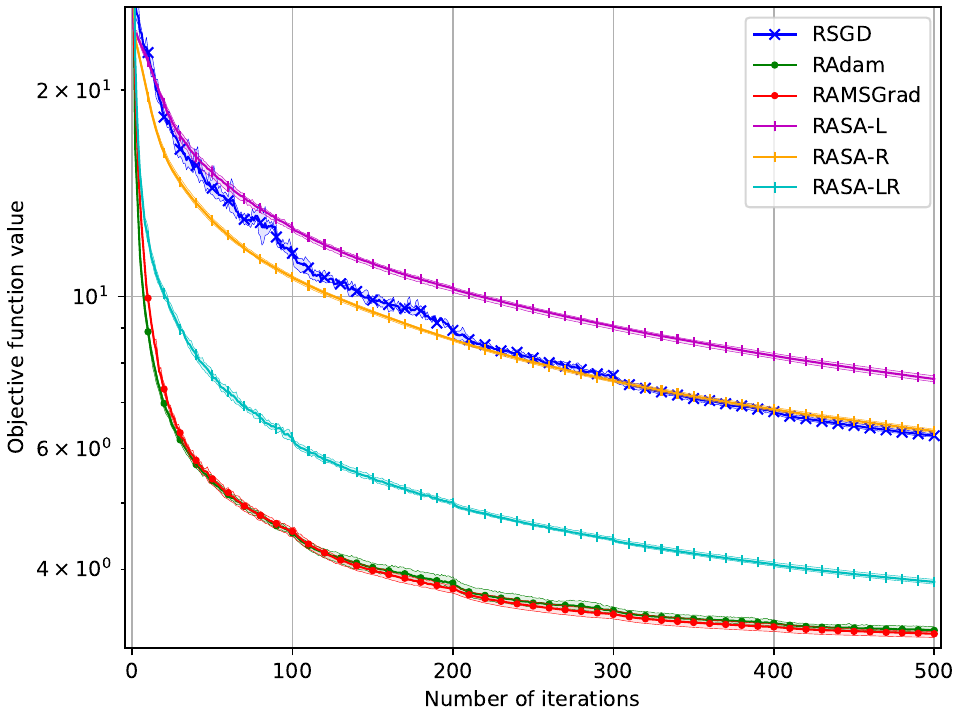}
}    
\subfigure[diminishing learning rate]{
    \includegraphics[clip, width=0.4\columnwidth]{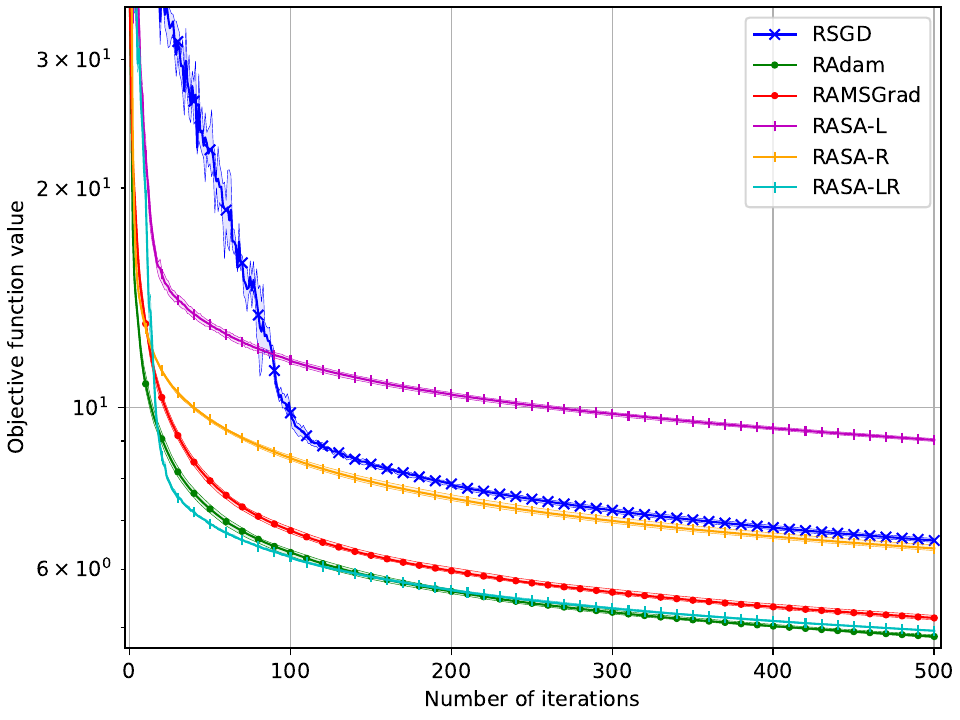}
}
\caption{Objective function value defined by \eqref{eq:pca} versus number of iterations on the training set of the COIL100 datasets.}
\end{figure}

\begin{figure}[htbp]
\centering
\subfigure[constant learning rate]{
    \includegraphics[clip, width=0.4\columnwidth]{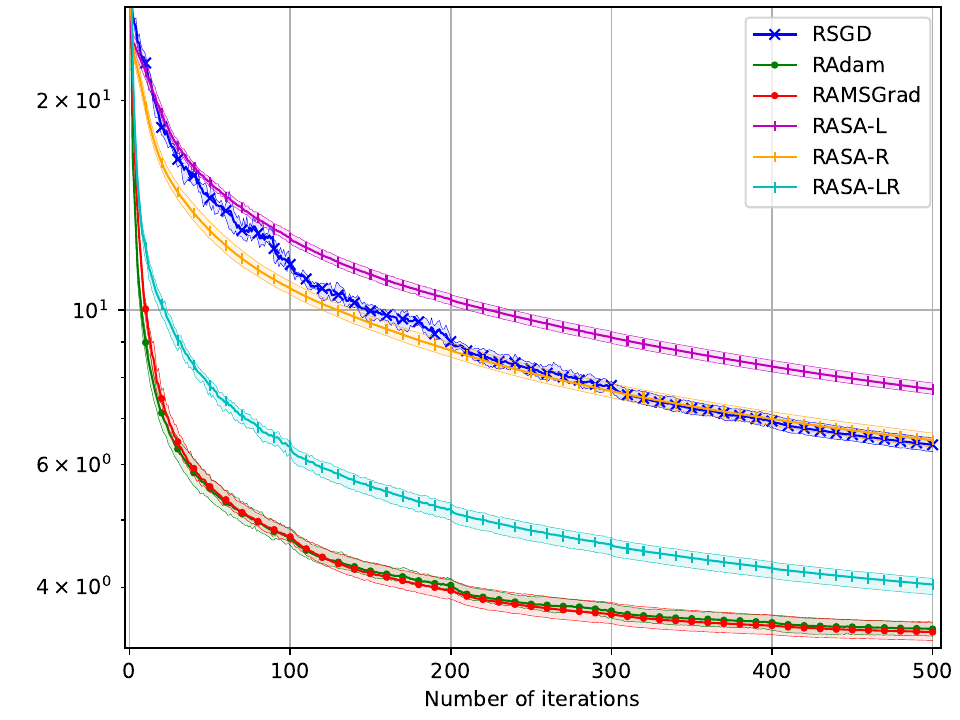}
}    
\subfigure[diminishing learning rate]{
    \includegraphics[clip, width=0.4\columnwidth]{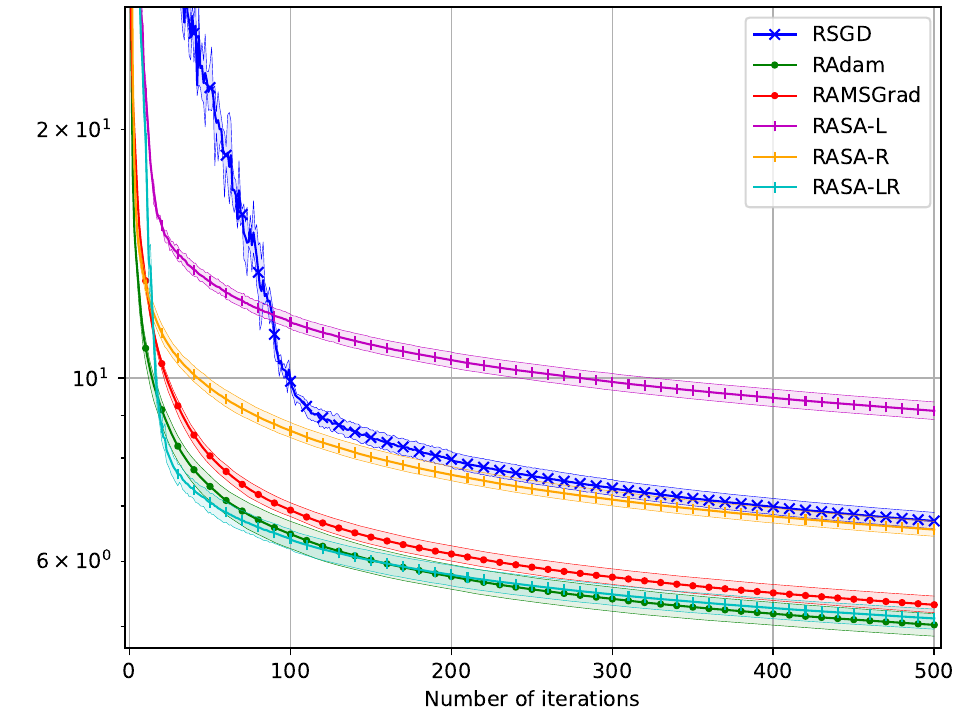}
}
\caption{Objective function value defined by \eqref{eq:pca} versus number of iterations on the test set of the COIL100 datasets.}
\end{figure}

\begin{figure}[htbp]
\centering
\subfigure[constant learning rate]{
    \includegraphics[clip, width=0.4\columnwidth]{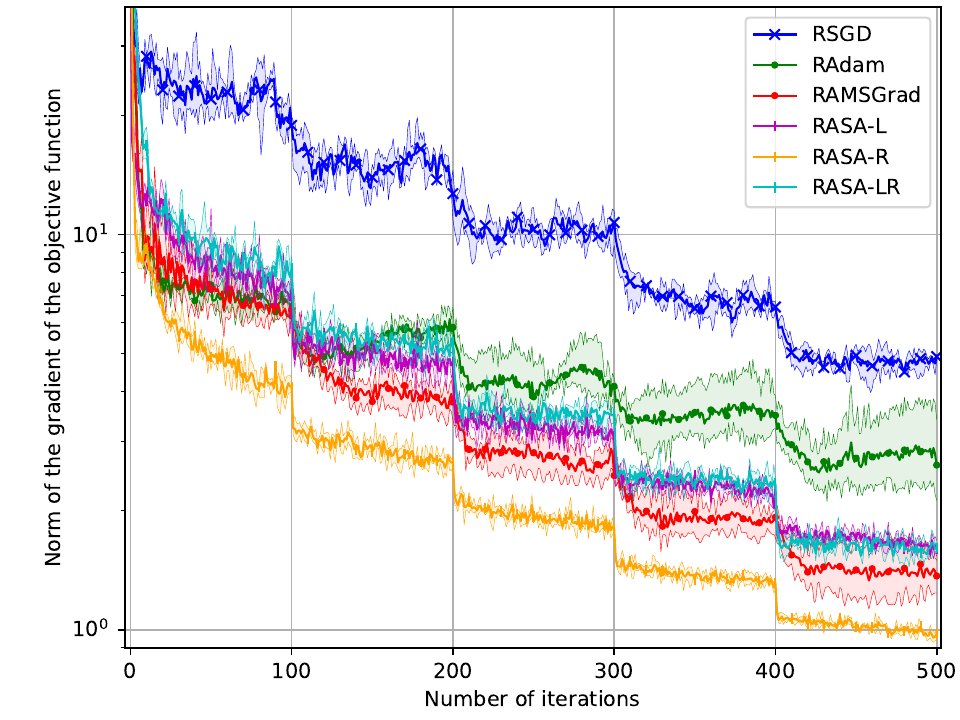}
}    
\subfigure[diminishing learning rate]{
    \includegraphics[clip, width=0.4\columnwidth]{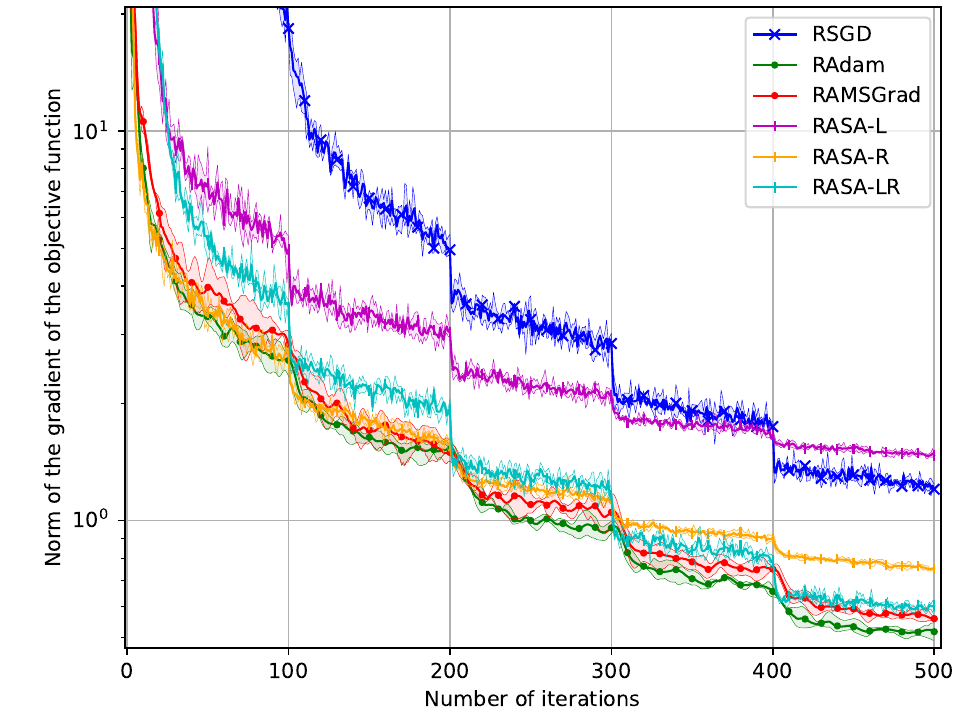}
}
\caption{Norm of the gradient of objective function defined by \eqref{eq:pca} versus number of iterations on the training set of the COIL100 datasets.}
\end{figure}

\begin{figure}[htbp]
\centering
\subfigure[constant learning rate]{
    \includegraphics[clip, width=0.4\columnwidth]{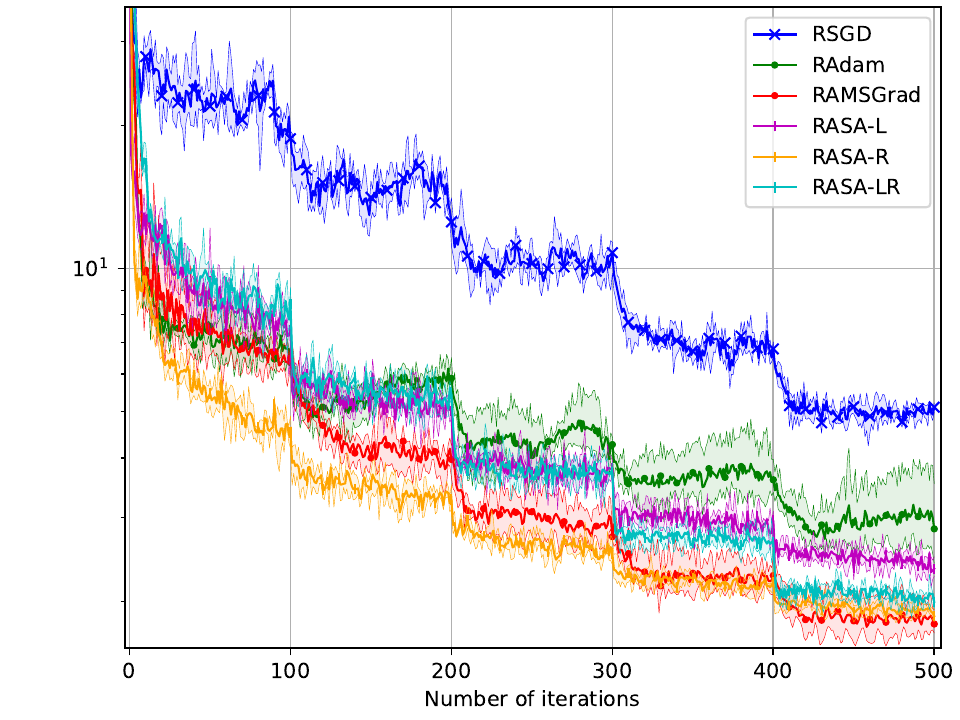}
}    
\subfigure[diminishing learning rate]{
    \includegraphics[clip, width=0.4\columnwidth]{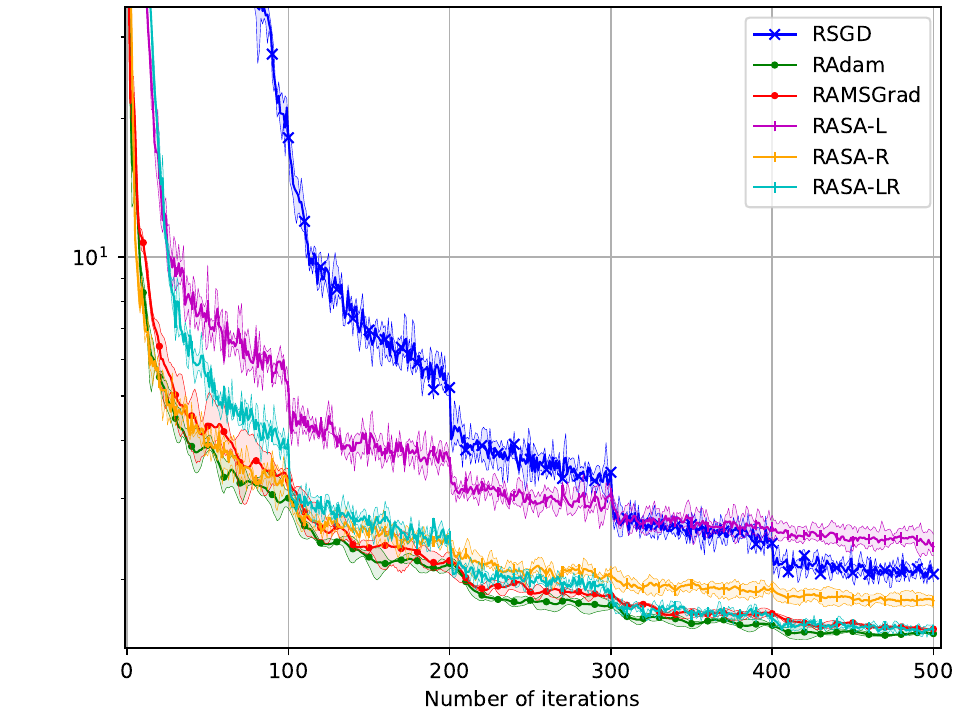}
}
\caption{Norm of the gradient of objective function defined by \eqref{eq:pca} versus number of iterations on the test set of the COIL100 datasets.}
\end{figure}

\begin{figure}[htbp]
\centering
\subfigure[constant learning rate]{
    \includegraphics[clip, width=0.4\columnwidth]{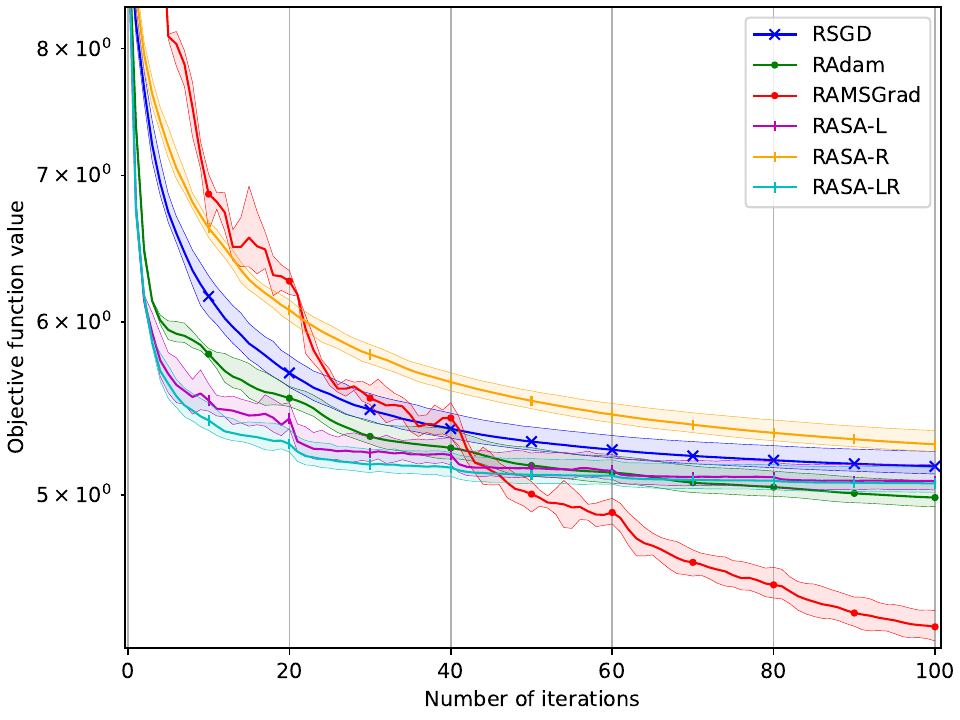}
}    
\subfigure[diminishing learning rate]{
    \includegraphics[clip, width=0.4\columnwidth]{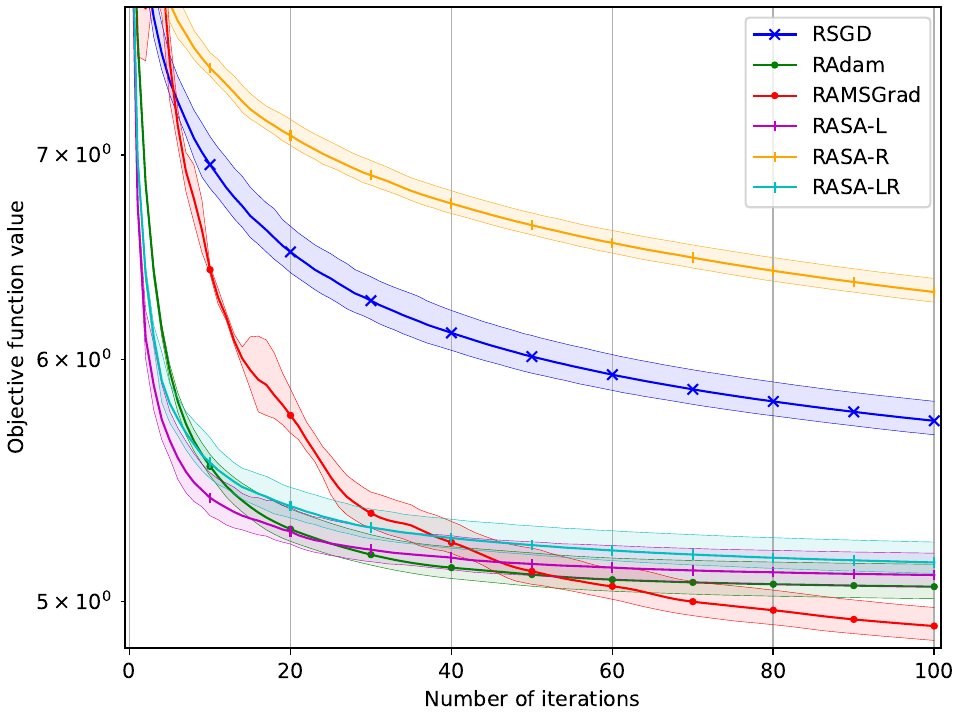}
}
\caption{Objective function value defined by \eqref{eq:lmc} versus number of iterations on the training set of the MovieLens-1M datasets.}
\end{figure}

\begin{figure}[htbp]
\centering
\subfigure[constant learning rate]{
    \includegraphics[clip, width=0.4\columnwidth]{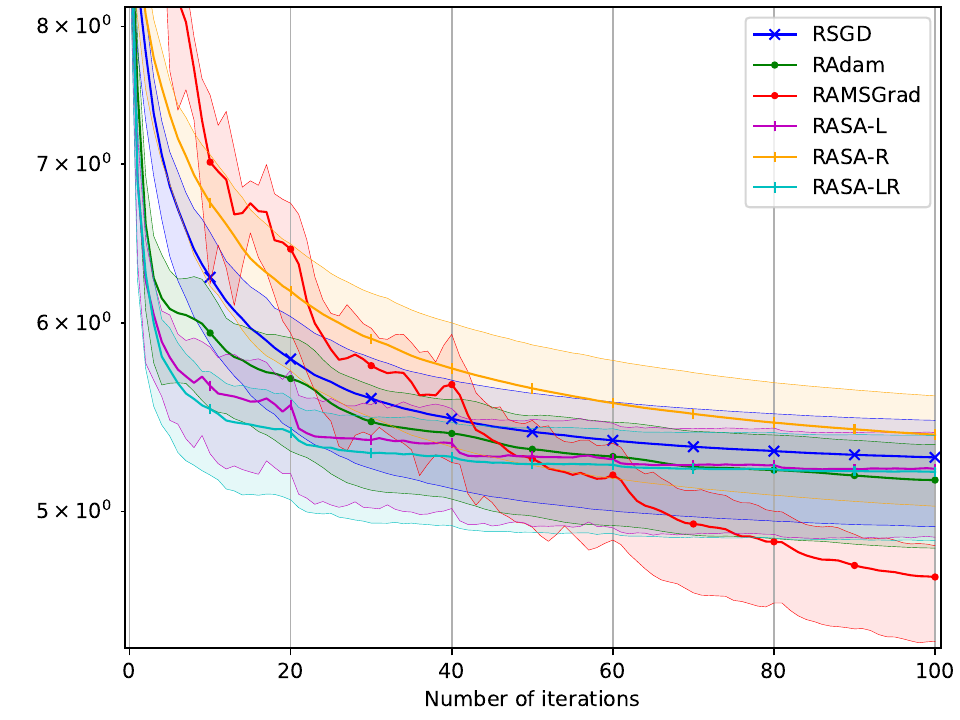}
}    
\subfigure[diminishing learning rate]{
    \includegraphics[clip, width=0.4\columnwidth]{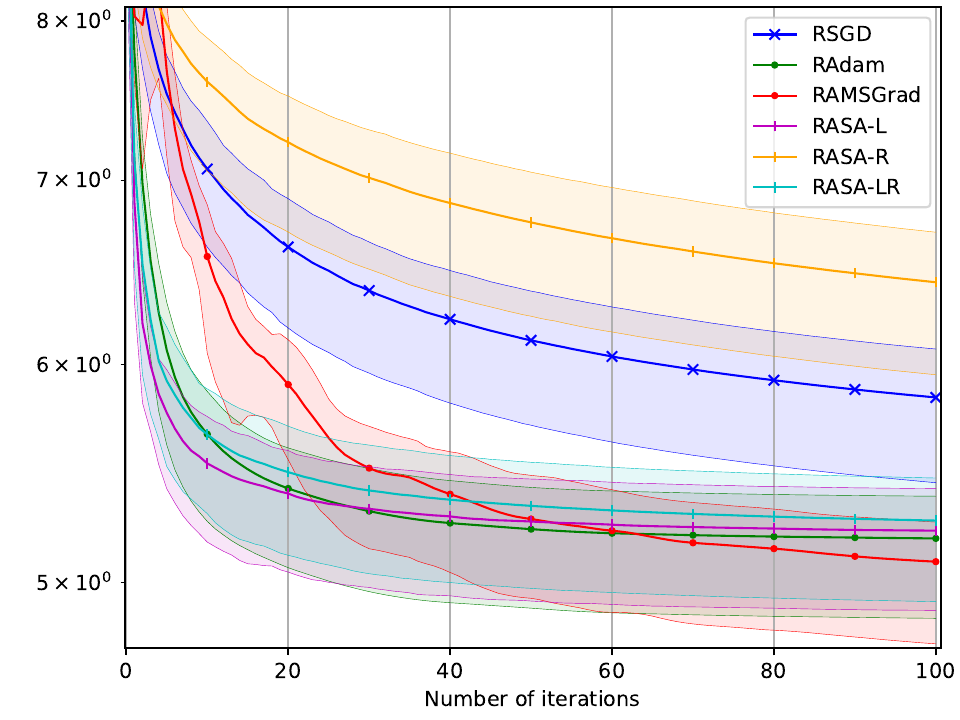}
}
\caption{Objective function value defined by \eqref{eq:lmc} versus number of iterations on the test set of the MovieLens-1M datasets.}
\end{figure}

\begin{figure}[htbp]
\centering
\subfigure[constant learning rate]{
    \includegraphics[clip, width=0.4\columnwidth]{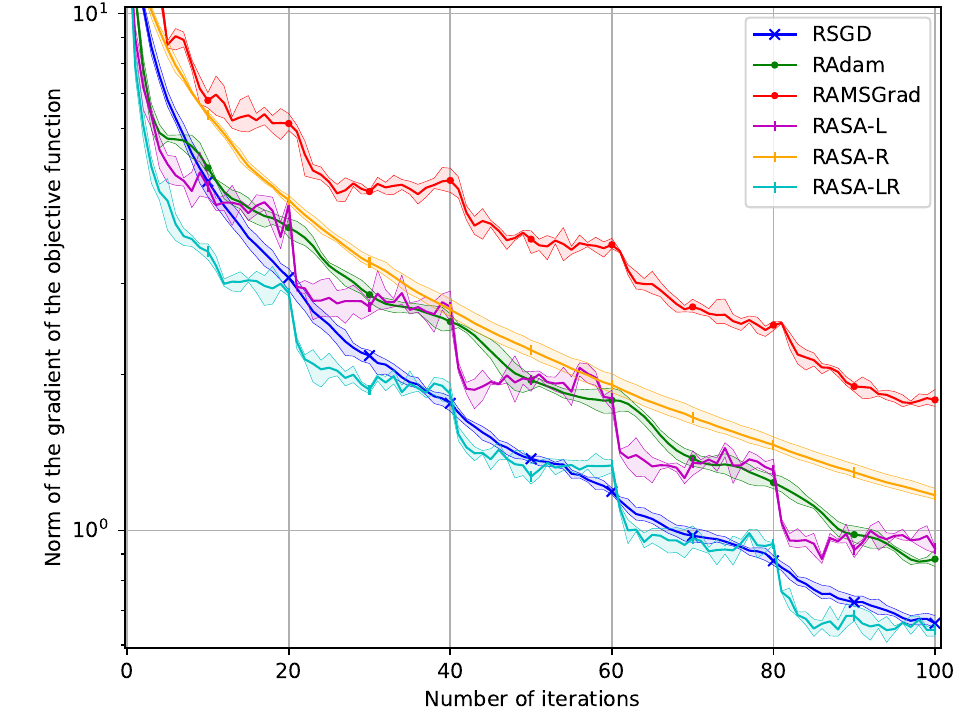}
}    
\subfigure[diminishing learning rate]{
    \includegraphics[clip, width=0.4\columnwidth]{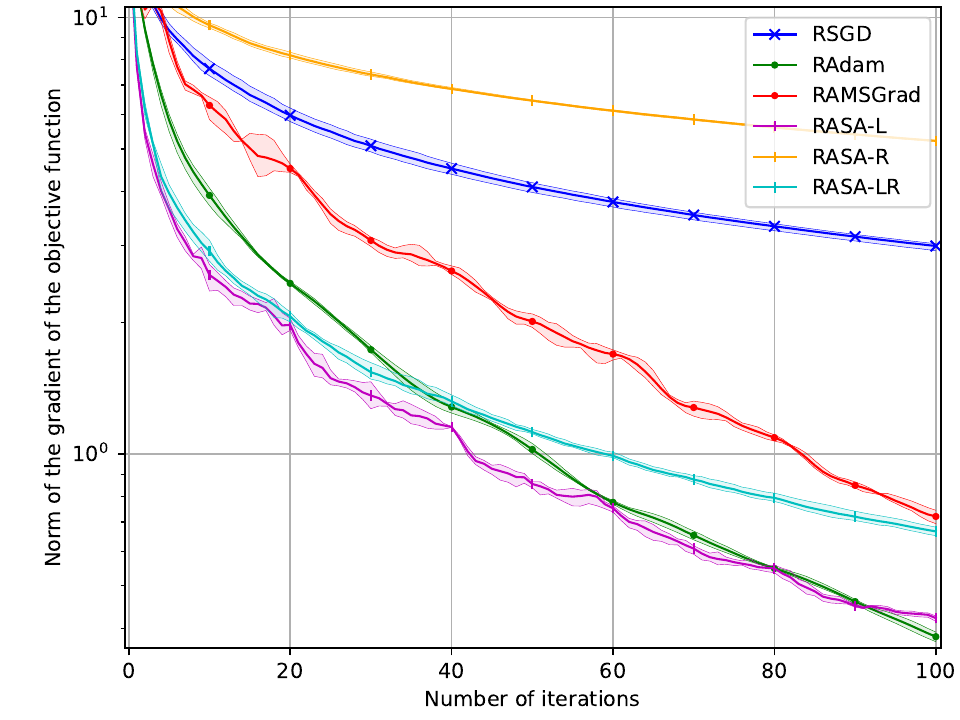}
}
\caption{Norm of the gradient of objective function defined by \eqref{eq:lmc} versus number of iterations on the training set of the MovieLens-1M datasets.}
\end{figure}

\begin{figure}[htbp]
\centering
\subfigure[constant learning rate]{
    \includegraphics[clip, width=0.4\columnwidth]{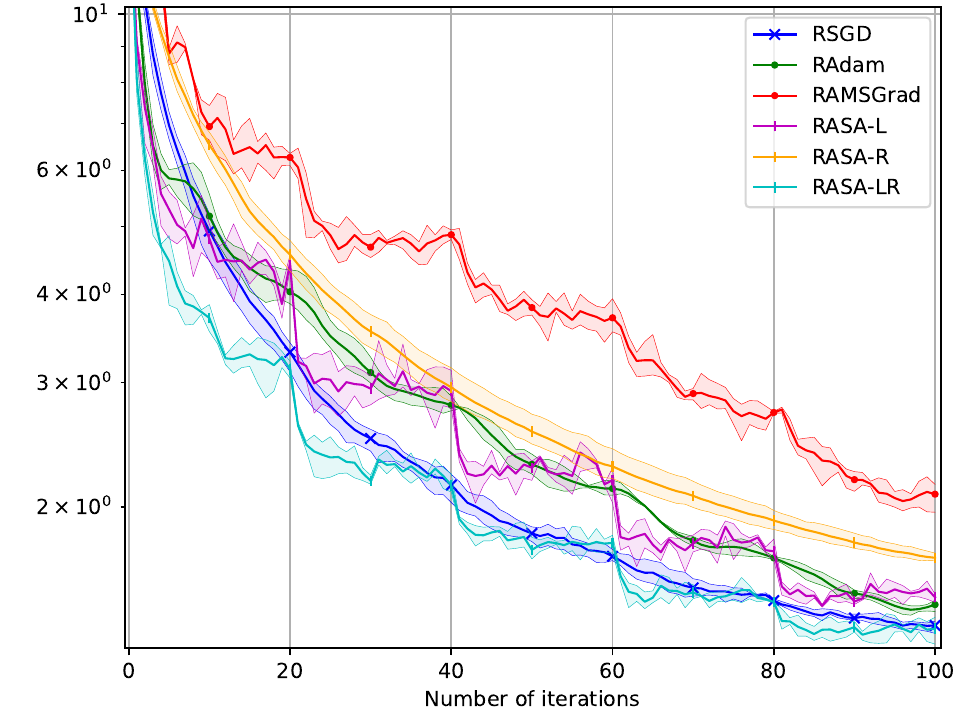}
}    
\subfigure[diminishing learning rate]{
    \includegraphics[clip, width=0.4\columnwidth]{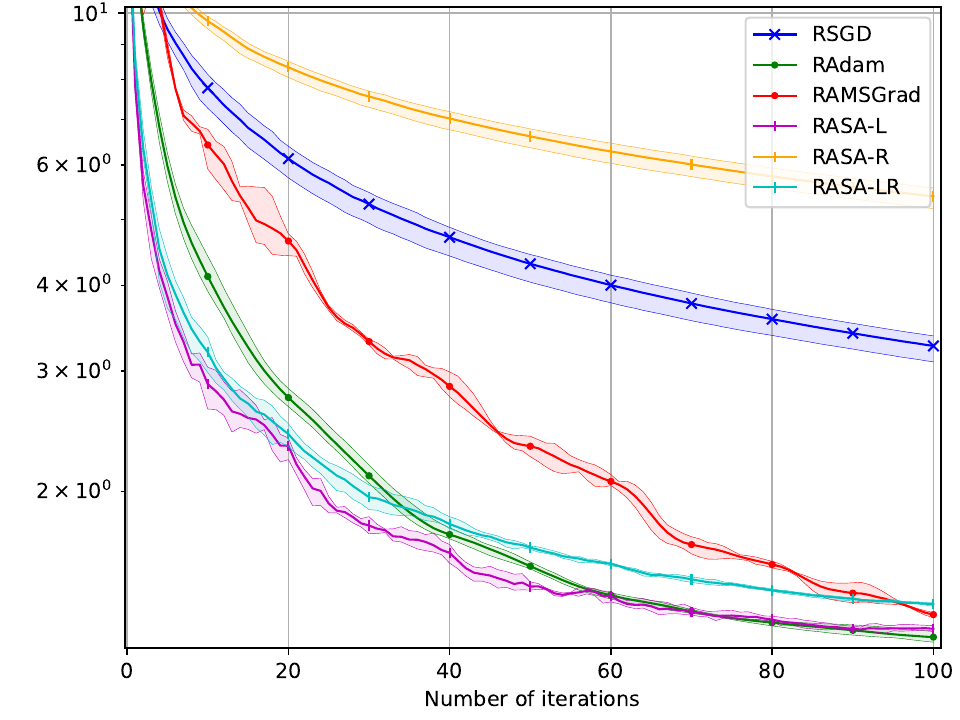}
}
\caption{Norm of the gradient of objective function defined by \eqref{eq:lmc} versus number of iterations on the test set of the MovieLens-1M datasets.}
\end{figure}

\begin{figure}[htbp]
\centering
\subfigure[constant learning rate]{
    \includegraphics[clip, width=0.4\columnwidth]{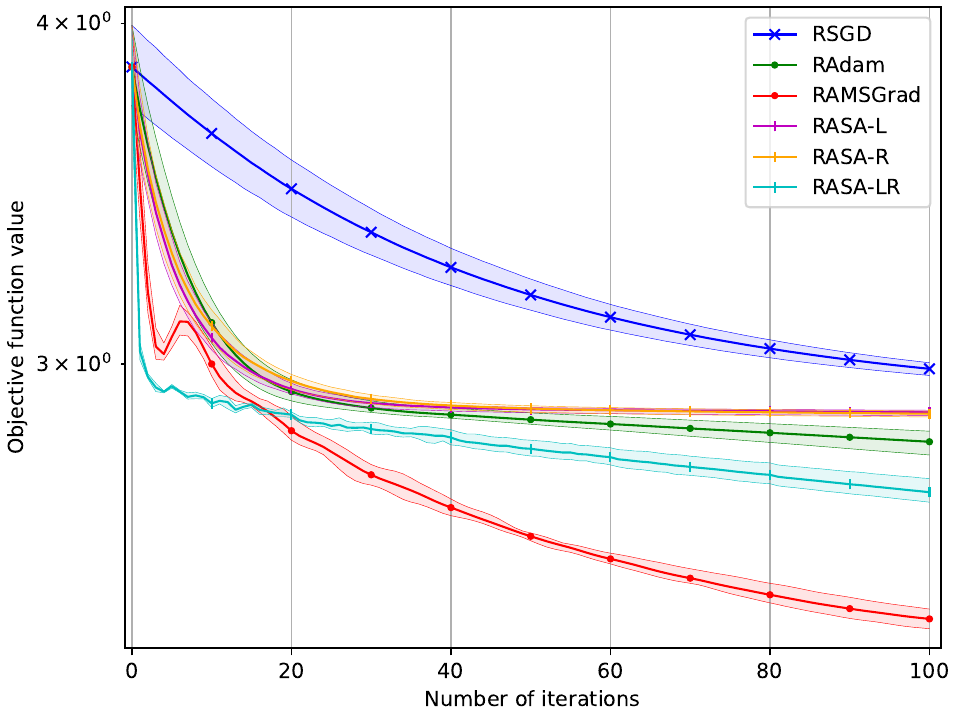}
}    
\subfigure[diminishing learning rate]{
    \includegraphics[clip, width=0.4\columnwidth]{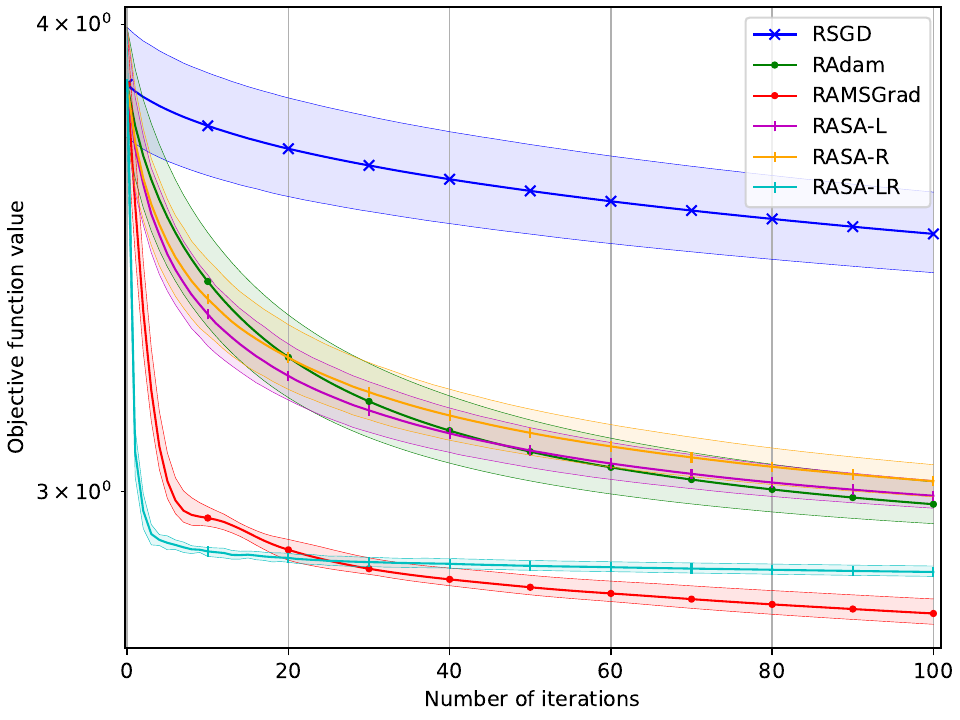}
}
\caption{Objective function value defined by \eqref{eq:lmc} versus number of iterations on the training set of the Jester datasets.}
\end{figure}

\begin{figure}[htbp]
\centering
\subfigure[constant learning rate]{
    \includegraphics[clip, width=0.4\columnwidth]{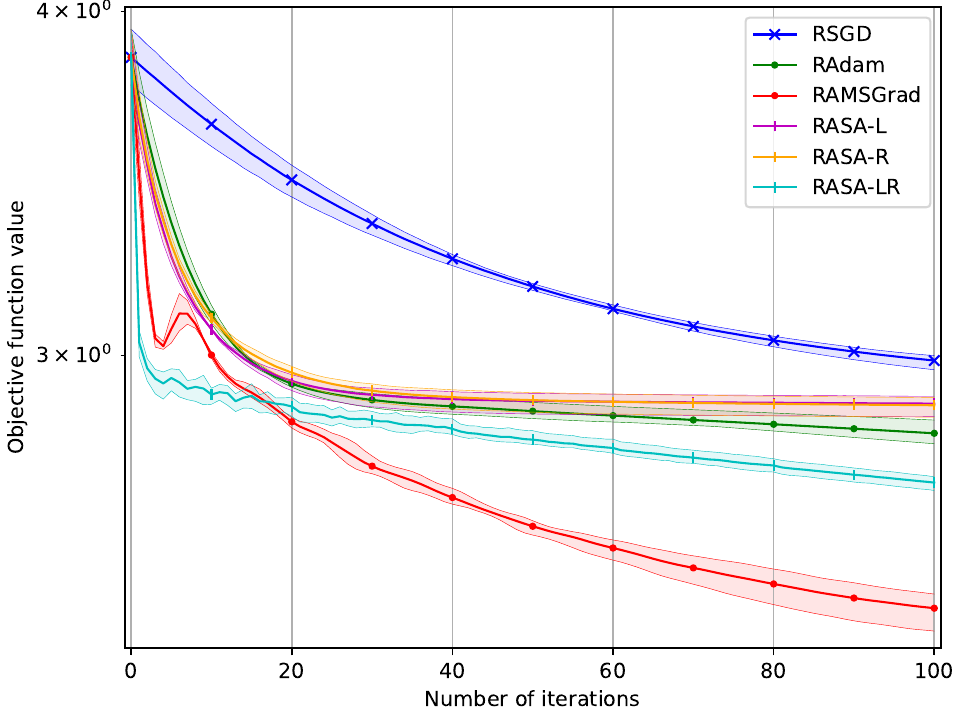}
}    
\subfigure[diminishing learning rate]{
    \includegraphics[clip, width=0.4\columnwidth]{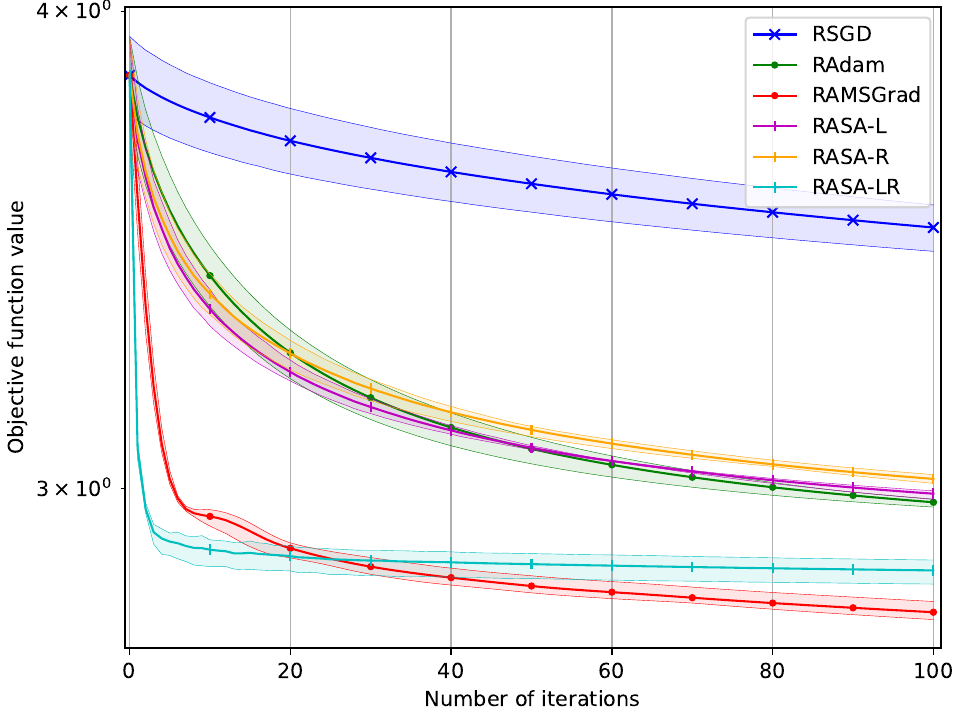}
}
\caption{Objective function value defined by \eqref{eq:lmc} versus number of iterations on the test set of the Jester datasets.}
\end{figure}

\begin{figure}[htbp]
\centering
\subfigure[constant learning rate]{
    \includegraphics[clip, width=0.4\columnwidth]{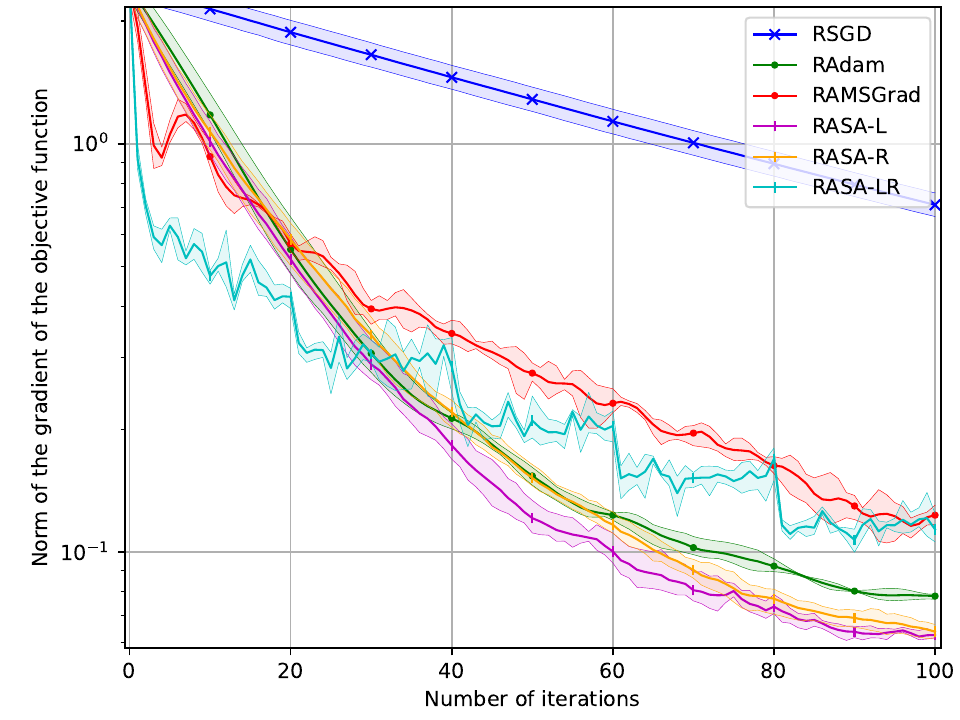}
}    
\subfigure[diminishing learning rate]{
    \includegraphics[clip, width=0.4\columnwidth]{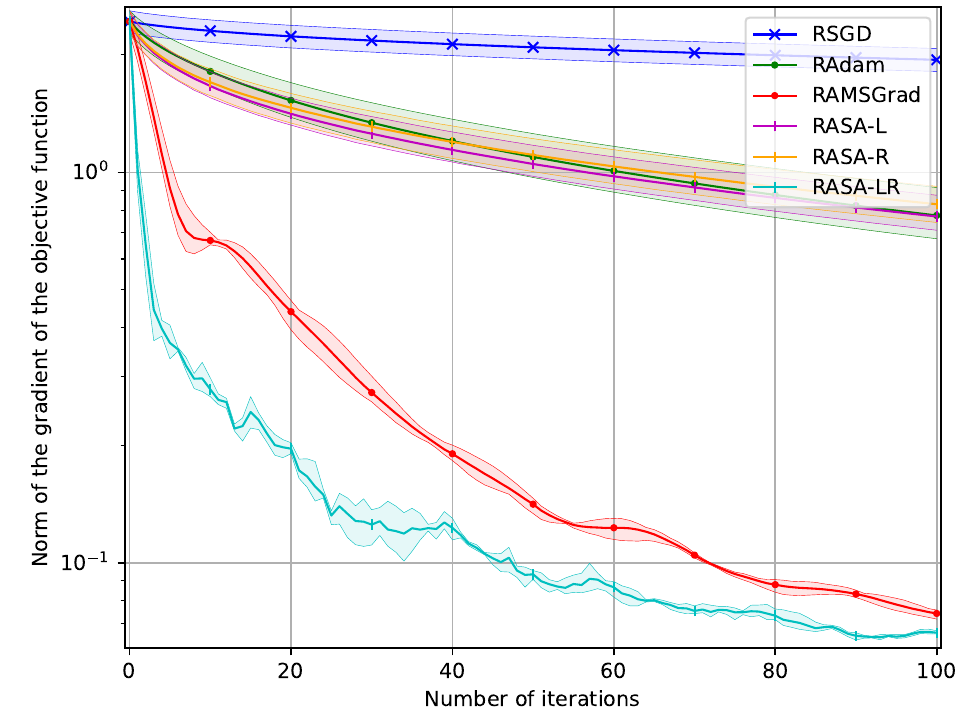}
}
\caption{Norm of the gradient of objective function defined by \eqref{eq:lmc} versus number of iterations on the training set of the Jester datasets.}
\end{figure}

\begin{figure}[htbp]
\centering
\subfigure[constant learning rate]{
    \includegraphics[clip, width=0.4\columnwidth]{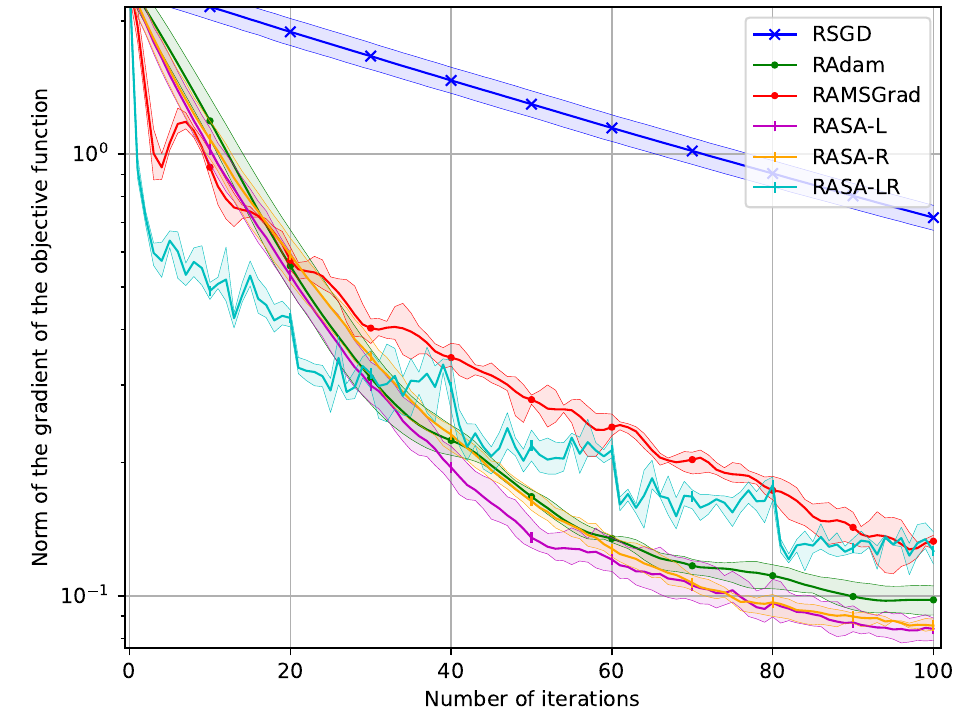}
}    
\subfigure[diminishing learning rate]{
    \includegraphics[clip, width=0.4\columnwidth]{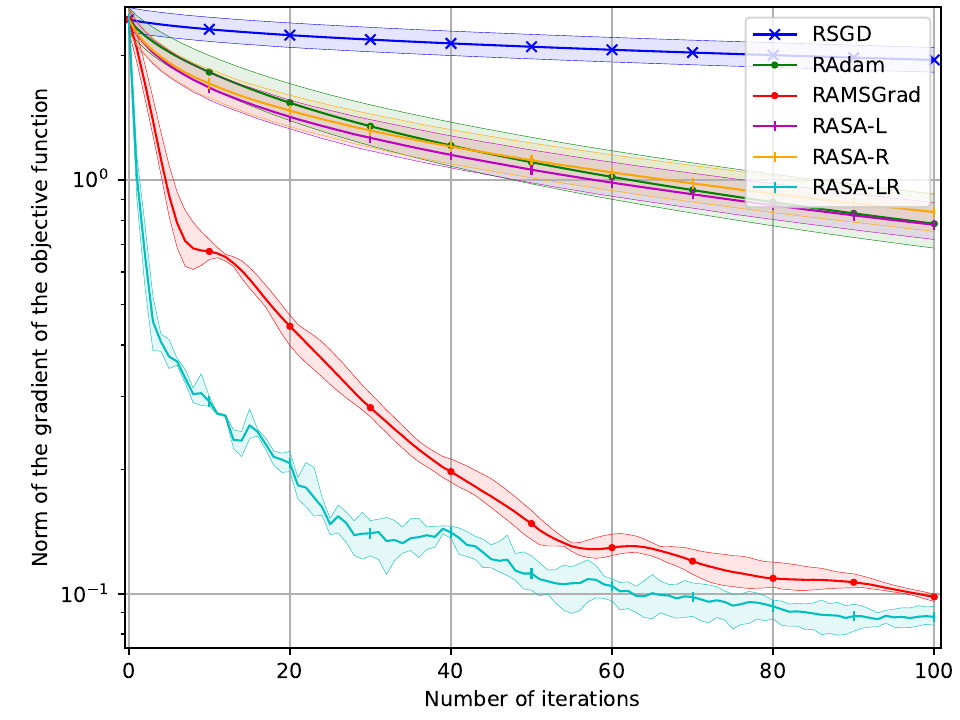}
}
\caption{Norm of the gradient of objective function defined by \eqref{eq:lmc} versus number of iterations on the test set of the Jester datasets.}
\end{figure}

\end{document}